\documentclass[table,11pt,a4paper]{amsart}
\usepackage{amsmath}
\usepackage{amsthm}
\usepackage{amssymb}
\usepackage{amsfonts}
\usepackage{changepage}
\calclayout 

\usepackage{color}
\usepackage[most]{tcolorbox}
\usepackage{scrextend} 
\usepackage{csvsimple}
\usepackage{tikz}
\usetikzlibrary{patterns}
\usetikzlibrary{arrows.meta}

\allowdisplaybreaks

\usepackage{enumerate}

\usepackage{graphics,epic,eepic,graphicx}
\usepackage{pgfplots}
    \pgfplotsset{compat=1.15}
    \usepgfplotslibrary{polar}

\usepackage{algorithm2e}

\usepackage{hyperref}

\newtheorem{theorem}{Theorem}

\newtheorem{lemma}[theorem]{Lemma}
\newtheorem{proposition}[theorem]{Proposition}
\newtheorem{corollary}[theorem]{Corollary}
\newtheorem{fact}[theorem]{Fact}
\newtheorem{claim}[theorem]{Claim}
\newtheorem{remark}[theorem]{Remark}

\newenvironment{poc}{\begin{proof}[Proof of claim.]}{\end{proof}}

\newcommand{\calA}{\mathcal{A}}
\newcommand{\calB}{\mathcal{B}}
\newcommand{\calC}{\mathcal{C}}
\newcommand{\calE}{\mathcal{E}}
\newcommand{\calF}{\mathcal{F}}
\newcommand{\calG}{\mathcal{G}}
\newcommand{\calH}{\mathcal{H}}
\newcommand{\calI}{\mathcal{I}}
\newcommand{\calO}{\mathcal{O}}
\newcommand{\calP}{\mathcal{P}}

\newcommand{\calR}{\mathcal{R}}
\newcommand{\calT}{\mathcal{T}}

\newcommand{\EE}{\mathbb{E}}
\newcommand{\HH}{\mathbb{H}}
\newcommand{\NN}{\mathbb{N}}
\newcommand{\PP}{\mathbb{P}}
\newcommand{\RR}{\mathbb{R}}

\newcommand{\HRG}{\mathrm{\mathcal{G}}}

\renewcommand{\leq}{\leqslant}
\renewcommand{\geq}{\geqslant}
\renewcommand{\epsilon}{\varepsilon}
\newcommand{\eps}{\varepsilon}

\newcommand{\dmin}{d_{\mathsf{min}}}

 \newcommand{\Prh}[2]{\mathbf{P}^{#1} \!\left( \,#2\,\right)}
 \newcommand{\Pruh}[3]{\mathbf{P}_{#1}^{#2} \!\left( \,#3\,\right)}
 
 \newcommand{\Exu}[2]{\mathbf{E}_{#1} \!\left( \,#2\,\right)}
 \newcommand{\Exh}[2]{\mathbf{E}^{#1} \!\left( \,#2\,\right)}
 \newcommand{\Exuh}[3]{\mathbf{E}_{#1}^{#2} \!\left( \,#3\,\right)}
 \newcommand{\tmix}{t_{\mathsf{mix}}}
 \newcommand{\thit}{t_{\mathsf{hit}}}
 
 \newcommand{\tcov}{t_{\mathsf{cov}}}
 \newcommand{\Resu}[3]{\mathcal{R}_{#1}\left(#2 \leftrightarrow #3\right)}
 \newcommand{\Res}[2]{\mathcal{R}\left(#1 \leftrightarrow #2\right)}

 \usepackage{dsfont}

 \newcommand{\hide}[1]{}
\newcommand\extrafootertext[1]{%
    \bgroup
    \renewcommand\thefootnote{\fnsymbol{footnote}}%
    \renewcommand\thempfootnote{\fnsymbol{mpfootnote}}%
    \footnotetext[0]{#1}%
    \egroup
}

\title{Cover and Hitting Times of Hyperbolic Random Graphs}

\author{Marcos Kiwi}
\address{Depto.~de Ingenier\'ia Matem\'atica y Ctr.~de Modelamiento Matem\'atico (CNRS IRL2807), Universidad de Chile, Santiago, Chile}
\email{mkiwi@dim.uchile.cl}

\author{Markus Schepers}
\address{Institut f\"ur Medizinische Biometrie, Epidemiologie und Informatik, Johannes-Gutenberg-U.~Mans, Mainz, Germany}
\email{markus.schepers@uni-mainz.de}

\author{John Sylvester}
\address{Department Of Computer Science, University of Liverpool, Liverpool, UK}
\email{john.sylvester@liverpool.ac.uk}

\begin{document}

\begin{abstract}We study random walks on the giant component of Hyperbolic Random Graphs (HRGs), in the regime when the degree distribution obeys a power law with exponent in the range $(2,3)$. In particular, we first focus on the expected time for a random walk to hit a given vertex or visit, i.e.~cover, all vertices. We show that, a.a.s.~(with respect to the HRG), and up to multiplicative constants: the cover time is $n(\log n)^2$, the maximum hitting time is $n\log n$, and the average hitting time is $n$. We then determine the expected time to commute between two given vertices a.a.s., up to a small factor polylogarithmic in $n$, and under some mild hypothesis on the pair of vertices involved. Our results are proved by controlling effective resistances using the energy dissipated by carefully designed network flows associated to a tiling of the hyperbolic plane, on which we overlay a forest-like structure.

	\medskip
	
	\noindent \textbf{\emph{Keywords---}} Random walk, hyperbolic random graph, cover time, hitting time, average hitting time, target time, effective resistance, Kirchhoff index.\\
	\textbf{\emph{AMS MSC 2010---}} 05C80, 60J10, 60G40.
\end{abstract}

\maketitle
\section{Introduction}
In 2010, Krioukov et al.~\cite{Krioukov2010} proposed the Hyperbolic Random Graph (HRG) as a model of
``real-world'' networks such as the Internet (also referred to as complex networks).\extrafootertext{An extended abstract of this paper appeared at \textit{RANDOM 2022} \cite{KiwiSS22}.}  Early results via non-rigorous methods indicated
that HRGs exhibited several key properties empirically observed in frequently studied networks 
(such as networks of acquaintances, citation networks, networks of autonomous systems, etc.). Many of these properties were later established formally, among these are power-law degree distribution~\cite{GPP12}, short graph distances~\cite{ABF15,KM15} (a.k.a.~small world phenomena),
and strong clustering~\cite{chellig2021modularity,fountoulakisClusteringHyperbolicModel2021,GPP12}. 
Many other fundamental parameters of the HRG model have been studied since its introduction (see the related work section), however notable exceptions are key quantities concerning the behaviour of random walks. This paper is a first step in redressing this situation. The random walk is the quintessential random process, and studies of random walks have proven relevant for algorithm design and analysis; this coupled with the aforementioned appealing aspects of the HRG model motivates this research.

The (simple) random walk is a stochastic process on the vertices of a graph, which at each time step uniformly samples a neighbour of the current vertex as its next state~\cite{aldousfill,peresmix2}. A key property of the random walk is that, for any connected graph, the expected time it takes for the walk to visit a given vertex (or to visit all vertices) is polynomial in the number of vertices in the graph. These times are known as the hitting and cover times, respectively. This ability of a random walk to explore an unknown connected graph efficiently using a small amount of memory was, for example, used to solve the undirected $s-t$ connectivity problem in logarithmic space~\cite{AKLLR}.  Other properties such as the ability to sample a vertex independently of the start vertex after a polynomial (often logarithmic) number of steps (mixing time) helped random walks become a fundamental primitive in the design of randomized and approximation algorithms~\cite{lovasz1993random}. In particular, random walks have been applied in tasks such as load balancing~\cite{load1}, searching~\cite{search1}, resource location~\cite{gossip}, property testing~\cite{Testing}, graph parameter estimation~\cite{esti} and biological applications~\cite{Korman20}.

One issue to keep in mind when working with HRGs is that for the most relevant range of parameters of the model (the range for which it exhibits the properties observed in `real-world' networks) the graphs obtained are disconnected with probability that tends to $1$ as the order of the graph goes to infinity. Quantities such as average hitting time and commute time are not meaningful for disconnected graphs, as they are trivially equal to infinity. However, again for the range of parameters we are interested in, Bode, Fountoulakis and M\"uller~\cite{BFM15} showed that it is very likely the graph has a component of linear size. This result was then complemented  by the first author and Mitsche~\cite{KM15} who showed that all other connected components were of size at most polylogarithmic in the order of the graph. This justifies referring to the linear size component as the \emph{giant component}. With this work being among the first studies of characteristics of simple random walks in HRGs, it is thus natural and relevant to understand their behavior in the giant component of such graphs. This is the main challenge we undertake in this paper. 

Among our main contributions are the determination of the order of growth (as a function of the parameters of the model) of the expectation of the target time and also of the hitting, cover and commute times of random walks on the giant component of HRGs. To achieve this, we appeal to a connection of the former to effective resistances in the graph~\cite[Section 9.4]{peresmix2}. The effective resistance has also found applications to graph clustering~\cite{AlevALG18},  spectral sparsification~\cite{SpielmanS11}, graph convolutional networks~\cite{AhmadJLT21}, and flow-based problems in combinatorial optimization~\cite{AnariG15}.

\subsection{Main Results}\label{sec:results}Our main contributions are to determine several quantities related to random walks on the largest connected component $\mathcal{C}_{\alpha,\nu}(n)$ of the (Poissonized) hyperbolic random graph~$\mathcal{G}_{\alpha,\nu}(n)$. We refer to this component as the \emph{giant} and note that it is known to have $\Theta(n)$ vertices a.a.s.~\cite{BFM15}. 
The primary probability space we will be working in is the one induced by the HRG. We use $\mathbb{P}$ for the associated measure. We also deal with the expected stopping times of random walks, and we use bold type (e.g.~$\mathbf{E}$) for the expectation with respect to the random walk on a fixed graph. A sequence of events (w.r.t.~the HRG) holds \emph{asymptotically almost surely (a.a.s.)} if it occurs with
probability going to $1$ as $n\to\infty$. We give brief descriptions of the objects we study here, for full definitions see Section~\ref{sec:prelim}. 

The \emph{hitting time} of a vertex $v$ from a vertex $u$ in a graph $G$ is the expected time it takes a random walk started from $u$ to first visit $v$. Let $t_{\mathsf{hit}}(G)$ denote the \emph{maximum hitting time}, that is, the maximum over all pairs of vertices $u,v$ in $V(G)$ of the hitting time of $v$ from $u$.

The \emph{target time} $t_{\odot}(G)$ of a graph $G$ (also known as \emph{Kemeny's constant}) is the expected time for a random walk to travel between two vertices chosen independently from the stationary distribution $\pi$, see Section~\ref{sec:rw}. So it can be seen as the weighted average of the hitting times, where the weights are implied by the stationary distribution, i.e., in the long-term, the target time is the most natural `average' hitting time.
Our main result, which is also technically the most difficult to establish, shows that on the giant of the HRG the target time is of order $n$ in expectation and a.a.s.\footnote{See Remark \ref{rem:randomO}, and more generally Section \ref{sec:conventions}, for our conventions on asymptotic notation and high probability statements.}. Formally, we prove the following.

\def\tstat{For any $\frac12< \alpha <1$ and  $\nu >0$, if $\calC:=\calC_{\alpha,\nu}(n)$, then $t_{\mathsf{\odot}}(\calC)=\Theta(n)$ a.a.s and in expectation.}

\begin{theorem}\label{flow:thm:tstat}  
\tstat 
\end{theorem}

The \emph{cover time $t_{\mathsf{cov}}(G)$} is the expected time for the walk to visit all vertices of $G$ (taken from a worst case start vertex), see Section~\ref{sec:rw}. We show that both maximum hitting time $\thit (\calC)$ and cover time $\tcov(\calC)$ concentrate on the giant of the HRG. 
\begin{theorem}\label{thm:maxhitcov}
 For any $\frac{1}{2}< \alpha < 1$ and  $\nu >0$, if $\calC:=\calC_{\alpha,\nu}(n)$, then
 a.a.s.~and in expectation \[\thit (\calC)=\Theta(n\log n) \qquad \text{and } \qquad \tcov(\calC)=\Theta(n\log^2 n).\]
\end{theorem}  

The \emph{commute time} between vertices $s$ and $t$ is the expectation of the number of steps~$\tau_{s,t}$ it takes the random walk to visit $t$ starting from $s$ and then return to $s$. Besides being of intrinsic interest, the commute time can be used as a measure of network search efficiency, proximity of data, and cluster cohesion. We give sharp estimates for the commute time between two additional vertices under mild hypotheses on these two vertices are added to the HRG.

	\begin{theorem}\label{ap:thm:commut}  
	Consider $\frac12<\alpha<1$ and $\nu>0$. Suppose $s,t$ are vertices added to $\calC:=\calC_{\alpha,\nu}(n)$ at points with radial coordinates $r_s, r_t$, respectively.
	If the degrees of $s$ and $t$ are $d(s)$ and $d(t)$, respectively, and  $d(s),d(t)\geq 1$, then the following holds:
	\begin{enumerate}[(i)]
		\item\label{ap:thm:commut:itm1} If $(1-\frac{1}{2\alpha})R< r_s,r_t\leq R-\frac{\ln(R/\nu)}{2(\alpha-1/2)(1-\alpha)}-\Theta(1)$, then a.a.s.,
		\[\Exu{s}{\tau_{s,t}}=\calO\Big(n\cdot \Big(\frac{1}{d(s)^{4\alpha(1-\alpha)}}+\frac{1}{d(t)^{4\alpha(1-\alpha)}}\Big)\Big).
		\]
		
		\item\label{ap:thm:commut:itm2} If the angle formed at the origin between $s$ and $t$ is at least $\min\{\phi_{r_s},\phi_{r_t}\}$ where $\phi_r:=o(e^{\alpha(R-r)}/n)$, then a.a.s.,
		\[
		\Exu{s}{\tau_{s,t}}=\Omega\Big(\frac{n}{(\ln n\cdot\ln\ln n)^3}\cdot \Big(\frac{1}{d(s)^{4\alpha(1-\alpha)}}+\frac{1}{d(t)^{4\alpha(1-\alpha)}}\Big)\Big).
		\]
	\end{enumerate}
	\end{theorem}
There are several aspects of the preceding theorem that are worth stressing in order to better appreciate the mildness of the conditions required and the sharpness of the bounds. 
Since the expected number of vertices of  $\HRG_{\alpha,\nu}(n)$ with radial coordinate at most $(1-\frac{1}{2\alpha})R$ is $\calO(1)$, the lower bound on the radial coordinates required 
in Item~\eqref{ap:thm:commut:itm1} of Theorem~\ref{ap:thm:commut} is satisfied by almost all vertices of the HRG. On the other hand, 
from the hitting time bound of Theorem~\ref{thm:maxhitcov} it follows that the commute time between any pair of vertices $s$ and $t$ of $\HRG_{\alpha,\nu}(n)$ is $\calO(n\log n)$ in expectation.
Vertices that do not satisfy the upper bound on the radial coordinates in Item~\eqref{ap:thm:commut:itm1} of Theorem~\ref{ap:thm:commut} have, in expectation, degrees that are at most polylogarithmic in $n$ so, even  for these vertices, we have good (up to polylogarithmic factors) bounds on commute times. 
To conclude our discussion of Theorem~\ref{ap:thm:commut}, note that the restriction on the angle between vertices $s$ and $t$ in Item~\eqref{ap:thm:commut:itm2} of Theorem~\ref{ap:thm:commut} is satisfied a.a.s.~for any random pair of points (angular coordinates chosen uniformly in $[0,2\pi)$) whose coordinates are prescribed to be $r_s, r_t\geq (1-\frac{1}{2\alpha})R$.

The \emph{effective resistance} $\Res{u}{v}$ between two vertices $u,v$ of a graph $G$ is the energy dissipated by a unit current flow from $u$ to $v$. 
Effective resistances are of key interest when interpreting a graph as an electrical network with some electrical current flow through the edges, representing wires, each with a certain resistance.
The study of simple random walks corresponds to the case when all edges have unit resistances, see Section~\ref{sec:electralnetworks} for a formal definition. 
The sum of all resistances in $G$ is the \emph{Kirchhoff index} $\mathcal{K}(G)$, this has found uses in centrality~\cite{LiZ18}, noisy consensus problems~\cite{PattersonB14}, and social recommender systems~\cite{WongLC16}.
Our main result here gives the expected order of the Kirchhoff index and shows that the expected effective resistance between two vertices of the giant chosen uniformly at random is bounded. 
	\def\resistancethm{For any $\frac12< \alpha <1$ and  $\nu >0$, if $\calC:=\calC_{\alpha,\nu}(n)$, then  
	\[\mathcal{K}(\calC) = \Theta(n^2) \qquad \text{and} \qquad \frac{1}{|V(\calC)|^2}\sum_{u,v\in V(\calC)} \Res{u}{v}= \Theta(1) \]
a.a.s.~and in expectation.}

\begin{theorem}\label{thm:resistance} 	
\resistancethm \end{theorem}
The two results in the preceding theorem are very closely related but do not directly imply each other as $|V(\mathcal{C})|$, the size of the largest connected component, is a random variable that is not independent of the resistances.  

Observe that the gap between maximum (as implied by Theorem~\ref{thm:maxhitcov}) and average effective resistance (given in Theorem~\ref{thm:resistance}) differs slightly from the previously established gap for the HRG between maximum and average graph distance (i.e., between diameter and typical graph distance): for the graph distances in the giant, it has been known that the maximum distance, i.e.~diameter, is $\Theta(\log n)$~\cite{MS19} and the average, or typical, distance is $\Theta(\log\log n)$~\cite{ABF15} a.a.s. Now, our result on the maximum hitting time means  that the maximum resistance between two vertices of the giant is $\Theta(\log n)$ a.a.s., while the average resistance (over a uniform pair of vertices from the giant) is $\Theta(1)$ a.a.s. by Theorem~\ref{thm:resistance}. 

Furthermore, note the interesting observation there are enough (polynomially many) pairs of vertices with resistance matching the maximum to ensure that the cover time is a factor $\Theta(\log^2 n)$ larger than the average hitting time, many random graphs (e.g., connected Erd\H{o}s-R\'{e}nyi, preferential attachment)  are expanders and do not have this feature. 

Finally, regarding the commute time result (Theorem~\ref{ap:thm:commut}), in a graph with good expansion one would expect that for a pair of vertices $s,t\in V$ we have $\Exu{s}{\tau_{s,t}}$ of the order $|E|\cdot \big(\tfrac{1}{d(s)} + \tfrac{1}{d(t)}\big)  $, where $E$ is the edge set of the graph. However, the HRG is not an expander and we show that, a.a.s.\ up-to polylogs for most pairs $s,t$ in the giant of the HRG, $\Exu{s}{\tau_{s,t}}$ is instead of the order $|E|\cdot \big(\tfrac{1}{d(s)^{4\alpha(1-\alpha)}} + \tfrac{1}{d(t)^{4\alpha(1-\alpha)}}\big)  $. In particular, for $1/2<\alpha <1$, the power on the inverse degrees is strictly less than one, so the commute times are penalised by the lack of expansion in the HRG.   

Stating additional contributions of this paper (in their proper location within the paper), as well as providing more detail about the main results already stated, requires a bit more terminology and notation, which we introduce below after discussing the related literature.

\subsection{Proof Strategy}
The overall approach of the paper is to derive the main results concerning random walks, in particular the target time, by proving corresponding results for effective resistances. Other standard approaches do not seem to readily provide as good results. To illustrate this, we next discuss how we tackle the determination of the cover time.

 Over the last two decades, the cover time of many random graph models has been determined  using Cooper and Frieze's \emph{first visit lemma}, see references in the following related work section or~\cite{manzo2021probabilistic}. This result is based on expressing the probability that a vertex has been visited up-to a given time by a function of the return probabilities. One (simplified) condition required to easily apply the first visit lemma is that $t_{\mathsf{rel}} \cdot \max_{v\in V}\pi(v) = o(1)$, where $t_{\mathsf{rel}}(G) := \frac{1}{1-\lambda_2}$ is the \emph{relaxation time} of $G$, where $\lambda_2$ is the second-largest eigenvalue of the transition matrix of the (lazy) random walk on~$G$. However, inserting the best known bounds on $t_{\mathsf{rel}}$ and $\max_{v\in V}\pi(v) $ for the HRG, by~\cite{KiwiMitscheSpectral} and~\cite{GPP12} respectively, gives $t_{\mathsf{rel}}\cdot \max_{v\in V}\pi(v) \leq (n^{2\alpha -1 }\log n)\cdot n^{\frac{1}{2\alpha} -1 +o(1)}$ which is not $o(1)$ for any $1/2\leq \alpha\leq 1$. 

Another key ingredient of the first visit lemma is good bounds on the expected number of returns to a vertex before the walk mixes, i.e.~$\sum_{t=0}^{\tmix} P_{v,v}^t$ where $P_{x,y}^t$ is the probability a (lazy) random walk from $x$ is at vertex $y$ after exactly $t$ steps. Obtaining such bounds in the HRG appears challenging due to the large mixing time and irregular local structure of the HRG. This also effects arguably the most natural approach to obtaining bounds on the average hitting time, that is applying the formula $\pi(v) E_{\pi} [\tau_v]  = \sum_{t=0}^\infty [P_{v,v}^t -\pi(v)]$, see~\cite[Lemma 2.11]{aldousfill}, as this involves the same sum (which only needs to be considered up to relaxation/mixing time).

Given the perceived difficulty in determining the cover time using the return probabilities as described above, the approach taken in this paper is to determine upper bounds on the cover times (see Theorem~\ref{thm:maxhitcov}) via the effective resistances $\{\Res{u}{v}\}_{u,v\in V}$, and similarly for the other studied quantities. 
This relies on the intimate connection between reversible Markov chains and electrical networks, which arises as certain quantities in each setting are determined by the same harmonic equations. 
Classically this connection has been exploited to determine whether random walks on infinite graphs are transient or recurrent~\cite[Chapter 2]{LyonsPeres}, and more recently the effective resistance metric has been understood to relate the blanket times of random walks on finite graphs to the Gaussian free field~\cite{DingLeePeres}.
Among the random graph models whose cover times have been obtained via determination of effective resistances is the binomial random graph~\cite{Jonasson98,Sylvester21} and the geometric random graph~\cite{AvinE07}.
The main connection we shall use is that the commute time (sum of hitting times in either direction) between two vertices is equal to the number of edges times the effective resistance between the two points~\cite{Chandra,Te91}.  Luxburg et al.~\cite{luxburgHittingCommuteTimes2014} recently refined a previous bound of Lov\'asz~\cite{lovasz1993random} to give
\begin{equation}\label{eq:vonlux}\left|\Res{u}{v} - \frac{1}{d(u)} - \frac{1}{d(v)}  \right| \leq \frac{t_{\mathsf{rel}} + 2}{\dmin}\left(\frac{1}{d(u)} + \frac{1}{d(v)}  \right),  \end{equation}for any non-bipartite graph $G$ 
and $u,v\in V(G)$, where $d(w)$ is the degree of the vertex $w$ and $\dmin:=\min_{w\in V} d(w)$ is the minimum degree. For the HRG with parameter $1/2<\alpha <1 $, a.a.s.\ we have $t_{\mathsf{rel}} \geq n^{2\alpha -1 }/(\log n)^{1+o(1)
}$~\cite{KiwiMitscheSpectral} and the average degree is constant - thus~\eqref{eq:vonlux} does not give a good bound.  

The upper bounds in Theorems~\ref{thm:maxhitcov} through~\ref{thm:resistance} rely on the fact that hitting times are bounded above  by commute times and that commutes times are determined by effective resistances (via the Commute time identity).
We then exploit the fact that effective resistances can be expressed as the minimum of the energy dissipated by network flows. 
A significant part of our efforts are therefore focused on designing low energy flows. 
Our constructions rely 
on a tiling of the hyperbolic plane. In this regard, it is similar to how various authors have obtained estimates of the size of the giant and upper bounds on the diameter of the HRG~\cite{FM18,MS19}. However, when constructing a desirable flow one often needs multiple paths (as opposed to just one when bounding the diameter) or else the energy dissipated by the flow could be too large to get a tight bound on the effective resistance. Abdullah et al.~\cite{ABF15} showed that hyperbolic random graphs of expected size $\Theta(n)$ have typical distances of length $\Theta(\log \log n)$ (within the same component), in contrast we show that typical resistances are $\Theta(1)$. 
The diameter of the HRG when $\frac{1}{2}< \alpha < 1$ was only recently determined precisely~\cite{MS19}, though the lower bound, non-tight upper bounds, and the diameter for other values of $\alpha$, were established earlier~\cite{FriedrichK18,KM15}. 
The tight $\calO(\log n)$ upper bound for the diameter of the giant of the HRG when $\frac{1}{2}<\alpha <1$~\cite{MS19} was proved using a coupling with the Fountoulakis-M\" uller upper half-plane HRG model~\cite{FM18} and is also based on a tiling-construction. 
The tiling on which we rely is closely related to the 
Fountoulakis-M\"uller tiling of the half-plane model. In fact, our tiling is approximately equal to the latter (see the end of Section~\ref{ssec:tiling} for a detailed discussion). As for other instances in the literature where low energy flows are constructed in order to estimate effective resistances, the flows we construct elicit the combinatorial structure of the underlying graph, in our case, the bottlenecks (small edge-cuts) and expansion versus clustering trade-off of the HRG. We believe this by itself is of independent interest.

A technical complication that arises multiple times in our investigation (in particular in the proof of Theorem~\ref{flow:thm:tstat}) is the need to carefully control raw moments and higher order moments of the number of vertices within certain regions of the underlying geometric space. The arguments we used to do so might be useful in other studies concerning geometric graphs.

\subsection{Related Work}

Since their introduction in 2010~\cite{Krioukov2010}, hyperbolic random graphs have been studied by various authors. Apart from the results already mentioned (power-law degree distribution, short graph distances, strong clustering, giant component, spectral gap and diameter), connectivity was investigated by Bode et al.~\cite{bode2016probability}. Further results exist on the number of $k$-cliques and the clique number~\cite{friedrich2015cliques}, the existence of perfect matchings and Hamilton cycles~~\cite{fountoulakis2021hamilton}, the tree-width~\cite{blasius2016hyperbolic} and subtree counts~\cite{owada2018sub}. Two models, commonly considered closely related to the hyperbolic random graphs, are scale-free percolation~\cite{deijfen2013scale} and geometric inhomogeneous random graphs~\cite{bringmann2019geometric}.

Few random processes on HRGs have been rigorously studied. Among the notable exceptions is the work by Linker et al.~\cite{Linker21} which studies the contact process on the HRG model,
bootstrap percolation has been studied by Candellero and Fountoulakis~\cite{candellero2016bootstrap} and Marshall et al.~\cite{marshall2017targeting}, and, for geometric inhomogeneous random graphs, by Koch and Lengler~\cite{koch2016bootstrap}. 
Komj\'athy and Lodewijks~\cite{KL2020} studied first passage percolation on scale-free spatial network models. Finally, Fernley and Ortgiese have studied the voter model on scale free random graphs \cite{FernOrt}. 

To the best of our knowledge, the only work that explicitly studies random walks that deal with 
(a more general model of) HRGs is the work by Cipriani and Salvi~\cite{CipSal} on mixing time of scale-free percolation. However, some aspects of simple random walks have been analyzed on infinite versions of HRGs. 
Specifically, Heydenreich et al.~study transience and recurrence of random walks in the scale-free percolation model~\cite{HHJ17} (also known as heterogeneous long-range percolation) which is a `lattice' version of the HRG model. For similar investigations, but for more general graphs on Poisson point processes, see~\cite{GHMM21}. Additionally, the first author, Linker, and Mitsche~\cite{Mobile} have studied a dynamic variant of the HRG generated by stationary Brownian motions.  

Cover and hitting time have been studied in other random graph models such as the binomial random graph~\cite{CooperSparse,CooperGiant,Jonasson98,Sylvester21}, random geometric graph~\cite{AvinE07,CF11}, preferential attachment model~\cite{CooperPref}, configuration model~\cite{CooperConfig}, random digraphs~\cite{CooperDigraph} and the binomial random intersection graph~\cite{bloznelis2020cover}. Recently a central limit theorem was proven for the target time of the binomial random graph \cite{LowTer}. 

\subsection{Organization}
The rest of this paper is organized as follows.
In Section~\ref{sec:prelim}, we introduce notation and terminology, as well as  collect known results and derive some direct corollaries that we rely on throughout the paper.
In Section~\ref{sec:effectiveResist}, we show that the effective resistance between typical vertices of the giant of the HRG is $\calO(1)$. This is a key step on which we build to establish our first main result, that is, determining the target time (Theorem~\ref{flow:thm:tstat}).
In Section~\ref{sec:cover}, we determine the cover and maximum hitting time for the giant of the HRG (Theorem~\ref{thm:maxhitcov}).
In Section~\ref{sec:commute}, we determine commute times between pairs of vertices added at points with a prescribed radial coordinate (Theorem~\ref{ap:thm:commut}).
We conclude, in Section~\ref{sec:conclussion}, with some final comments and discussion of potential future research directions.

\section{Preliminaries}\label{sec:prelim}
In this section we introduce notation, define some objects and terms we will be working with, and collect, for future reference, some known results concerning them. We recall a large deviation bound in Section~\ref{sec:devBnds}, we then formally define the HRG model in Section~\ref{sec:HRG}, we discuss random walks in Section~\ref{sec:rw} and electrical networks in Section~\ref{sec:electralnetworks}.

\subsection{Conventions:}\label{sec:conventions} Throughout, we use standard notions and notation concerning the asymptotic behavior of sequences.
If $(a_n)_{n\in\NN}, (b_n)_{n\in\NN}$ are two sequences of real numbers, we write $a_n=\calO(b_n)$
to denote that for some constant $C>0$ and $n_0\in\NN$ it holds that $|a_n|\leq C|b_n|$ for all $n\geq n_0$.
Also, we write $a_n=\Omega(b_n)$ if $b_n=\calO(a_n)$, and $a_n=\Theta(b_n)$ if $a_n=\calO(b_n)$ and $a_n=\Omega(b_n)$. Finally, we write $a_n=o(b_n)$ if $a_n/b_n\to 0$ as $n\to\infty$ and $a_n=\omega(b_n)$ if $b_n=o(a_n)$.

Unless stated otherwise, all asymptotics are as $n\to\infty$ and all other parameters are assumed fixed. Expressions
given in terms of other variables that depend on $n$, for example $\calO(R)$, are still asymptotics with respect to $n$. Most results in this paper assume that $1/2<\alpha<1$ and $\nu>0$ are fixed constants (independent of $n$), and the implied constants in $\mathcal{O}(\cdot)$ asymptotics will typically depend on $\alpha$ and $\nu$. As we are interested in asymptotics, we only claim and prove inequalities for $n$ sufficiently large. So, for simplicity, we always assume $n$ sufficiently large. For example, we may write 
$n^2> 5n$ without requiring $n>5$.

An event, more precisely a family of events parameterized by $n\in\NN$,
is said to hold \emph{with
high probability (w.h.p.)}, if for every $c > 0$ the event holds with probability at least $1-\calO(n^{-c})$. Recall that, in contrast, an event is said to hold \textit{asymptotically almost surely (a.a.s.)} if the event holds with probability at least $1-o(1)$.

\begin{remark}\label{rem:randomO}
    For convenience we will often make statements of the form `$X=\mathcal{O}(f(n))$ a.a.s.', for some function $f$. What is meant by this is that there exists a constant $C > 0$ such that
	$P(X < C\cdot f (n)) \rightarrow 1$ as $n \rightarrow \infty$. Likewise, we say $X=\Omega(f(n))$ a.a.s.\ if there exists a constant $c > 0$ such that $P(X > c\cdot f (n)) \rightarrow 1$ as $n \rightarrow \infty$.
\end{remark}

 We follow standard notation, such as denoting the vertex and edge sets of $G$ by $V(G)$ and $E(G)$, respectively. We use $d_G(u,v)$ to denote the graph distance between two vertices $u,v\in V(G)$, let $N(v):=\{u\in V\mid d_G(u,v)=1\}$ denote the neighbourhood of a vertex, and let $d(v)= |N(v)|$. 
 
We use $\ln$ to denote the natural logarithm (base $e$) and use $\log$ when the base is not important (we fix base $2$ for concreteness). 

\subsection{Poisson Random Variables.}\label{sec:devBnds}
We work with a Poissonized model, where the number of points within a given region is Poisson-distributed. Thus, we will need some elementary results for Poisson random variables. The first is a (Chernoff type) large deviation bound. 
\begin{lemma}\label{prelim:lem:devBnd}
  Let $P$ have a Poisson distribution with mean $\mu$. Then, the following holds:
  \begin{enumerate}[(i)]
  \item\label{prelim:lem:devBnd:itm1} If $\delta> 0$, then $\PP(P\leq(1-\delta)\mu)\leq e^{-\delta^2\cdot \frac12\mu}$. In particular, $\PP(P\leq\frac12\mu)\leq e^{-\frac18\mu}$.
  \item\label{prelim:lem:devBnd:itm2}
  If $\delta \geq e^{\frac32}$, then
  $\PP(P\geq \delta\mu) \leq e^{-\frac12\delta\mu}$.
  \end{enumerate}
\end{lemma}
\begin{proof}
The first part is a direct application of~\cite[Theorem A.1.15]{AlonSpencer}.
For the second part, by the cited result, note that if $\chi>0$, then
 \[
 \PP(P\geq (1+\chi)\mu) \leq \big(e^{\chi}(1+\chi)^{-(1+\chi)}\big)^{\mu}.
 \]
 Taking $\chi:=\delta-1$ above (also noticing that by hypothesis
 on $\delta$, we have $\chi>0$), we get 
 $e^{\chi}(1+\chi)^{-(1+\chi)}=\frac{1}{e}(e/\delta)^{\delta}\leq 
 e^{-\frac12\delta}$ and the desired conclusion follows.
\end{proof}

The following bound on the raw moments of Poisson random variables will be useful.

\begin{lemma}[{\cite[Theorem 2.5]{BerTam}}]\label{prelim:lem:poimoment} Let $X$ be a Poisson random variable, and $\kappa\geq 1$. Then, \[\EE(X^\kappa)\leq \left(\frac{\kappa}{\ln (1 + \kappa) } \right)^{\kappa} \cdot \max\{\EE(X) ,\; \EE(X)^\kappa \}.\]
\end{lemma}

\subsection{The HRG model.}\label{sec:HRG}
We represent the hyperbolic plane (of constant Gaussian curvature~$-1$), denoted $\HH^2$, by points in $\RR^2$. 
Elements of $\HH^2$ are referred to by the polar coordinates $(r,\theta)$ of their representation as points in $\RR^2$.
The point with coordinates $(0,0)$ will be called the \emph{origin} of $\HH^2$ and denoted $O$.
When alluding to a point $u\in\HH^2$ we denote its polar coordinates by $(r_u,\theta_u)$.
The hyperbolic distance $d_{\HH^2}(p,q)$ between two points $p,q\in\HH^2$ is determined via 
the Hyperbolic Law of Cosines as the unique solution of 
\begin{equation}\label{eq:hypdist}
\cosh d_{\HH^2}(p,q) = \cosh r_p\cosh r_q-\sinh r_p\sinh r_q\cos(\theta_p-\theta_q).
\end{equation}
In particular, the hyperbolic distance between the origin and a point $p$ equals $r_p$. The ball of radius $\rho>0$ centered at $p$
will be denoted $B_p(\rho)$, i.e., \[B_p(\rho):=\{q\in\HH^2\mid d_{\HH^2}(p,q)<\rho\}.\]

We will work in the Poissonized version of the HRG model which we describe next.
For a positive integer $n$ and positive constant $\nu$ we consider a Poisson point process on the hyperbolic disk centered at the origin $O$ and 
of radius $R:=2\ln(n/\nu)$.
The intensity function at polar coordinates $(r,\theta)$ for $0\leq r< R$ and $\theta\in\RR$ equals
\begin{equation}\label{intro:eqn:intensity}
\lambda(r,\theta) := \nu e^{\frac{R}{2}}f(r,\theta) = n f(r, \theta)
\end{equation}
where $f(r,\theta)$ is the joint density function of independent random variables $\theta$ and $r$, with $\theta$ chosen uniformly at random in $[0,2\pi)$
and $r$ chosen according to the following density function:
\[
f(r) := \frac{\alpha\sinh(\alpha r)}{\cosh(\alpha R)-1}\cdot \mathbf{1}_{[0,R)}(r) \qquad \text{where $\mathbf{1}_{[0,R)}(\cdot)$ is the indicator of
$[0,R)$.}
\]
The parameter $\alpha>0$ controls the distribution: For $\alpha=1$, the distribution is uniform in $\HH^{2}\cap B_O(R)$, for smaller values the vertices are distributed more towards the center of $B_O(R)$ and for bigger values more towards the border, see Figure~\ref{fig:hrgExamples} for an illustration of instances of $\HRG_{\alpha,\nu}(n)$ for two distinct values of $\alpha$.
\begin{figure}[ht]
  \begin{center}
  \begin{tikzpicture}[scale=0.45]
    \path (0,0) coordinate (O);
    \draw (O) circle (8.0);
    
    \begin{filecontents*}{testVertex1.csv}
vertLbl, radialCoord, angCoord
0,7.92663697,-1.16469259
1,3.85872628,2.35095921
2,5.72949808,0.08111560
3,7.82749932,3.05168947
4,3.30725175,-1.22992560
5,5.48620955,-2.74251727
6,7.54385759,-0.10241642
7,6.04099515,-1.09614899
8,7.69361531,0.36012342
9,7.88803734,-0.83204729
10,4.36405330,-2.03374465
11,4.49286490,0.98597991
12,5.73236831,-0.03839651
13,2.95971795,2.37475181
14,7.98220577,-1.68773914
15,7.42537170,0.80165993
16,7.70100457,-2.33280815
17,4.24180771,-1.37305564
18,4.46641644,-1.58971375
19,6.65675713,0.61214648
20,7.59252507,1.06768333
21,5.44702831,-2.51187942
22,4.81942253,-2.01342732
23,7.58540814,-0.34760131
24,5.14664915,-0.72916883
25,7.55149747,-0.77034792
26,6.73711985,2.28335237
27,4.61245826,0.69903554
28,4.44358030,-1.25643122
29,6.78690680,2.17948442
30,7.04839763,-1.54687767
31,7.82153399,1.77061434
32,7.73286038,1.73425390
33,6.34131205,-1.06820593
34,6.54176221,0.19728352
35,6.84200505,-1.53335039
36,7.92949186,-1.66514968
37,6.92882567,-2.59892958
38,2.76788745,1.65073877
39,7.96106372,-0.95455961
40,5.38473247,1.41400113
41,6.86138817,0.72150641
42,6.60627225,-2.76126530
43,4.52589769,-2.65119575
44,3.54237829,0.51175210
45,6.25326761,1.67211532
46,7.81882967,-2.03116324
47,7.65698396,-0.95747698
48,2.70473101,2.13053844
49,7.80504652,1.07994205
50,6.97005361,2.97070110
51,6.03567002,1.05861023
52,5.55462270,-2.62490004
53,7.75145230,-1.84748070
54,4.13074887,2.70479519
55,4.62794259,0.22872521
56,4.24582339,-1.16837413
57,6.45508815,-3.06179905
58,7.88414363,-2.57701488
59,4.22473092,3.06045665
60,4.26405246,-2.22926483
61,7.37309176,-2.92097791
62,2.59803353,-1.04665684
63,7.81417935,0.51612973
64,5.17984244,1.40571016
65,7.40821583,1.29948623
66,6.12124941,2.74737455
67,7.53605400,-0.81237088
68,7.23821882,1.80092451
69,6.42167076,-0.83522815
70,6.58671372,0.44563650
71,5.90612522,-2.57399217
72,4.26701928,-1.42013771
73,4.28891929,0.91365106
74,7.40327748,1.70398500
75,7.43195312,-0.41391451
76,7.22791826,2.74365044
77,7.96401758,2.65419636
78,2.82954567,-2.42622159
79,6.77908498,-1.66410181
80,7.76628632,2.31755326
81,5.90394939,2.62431366
82,5.81332375,1.25072266
83,6.30799729,-2.59251451
84,6.41685010,1.68377025
85,7.43091543,2.08383330
86,7.93965091,-0.40075882
87,7.51779641,0.50405116
88,4.84450284,-2.24735172
89,7.71510628,-1.56169956
90,7.14067206,3.06421962
91,7.85144494,-0.60826917
92,4.21293867,-0.67747961
\end{filecontents*}
\csvreader[head to column names]{testVertex1.csv}{1=\vertLbl, 2=\radialCoord, 3=\angCoord}{%
      \path (\angCoord*180/3.1416:\radialCoord) coordinate (v\vertLbl);
    }
    
    \begin{filecontents*}{testEdge1.csv}
vertLblA, vertLblB
0,4
0,7
0,17
0,28
0,56
0,62
1,4
1,10
1,11
1,13
1,26
1,27
1,29
1,38
1,40
1,43
1,44
1,45
1,48
1,54
1,59
1,60
1,62
1,64
1,66
1,73
1,76
1,78
1,80
1,81
1,85
2,12
2,27
2,34
2,38
2,44
2,55
2,62
2,92
3,59
3,90
4,7
4,9
4,10
4,11
4,12
4,13
4,17
4,18
4,21
4,22
4,24
4,25
4,27
4,28
4,30
4,33
4,35
4,38
4,39
4,43
4,44
4,47
4,48
4,52
4,54
4,55
4,56
4,59
4,60
4,62
4,67
4,69
4,72
4,73
4,78
4,79
4,88
4,89
4,92
5,10
5,13
5,21
5,37
5,38
5,42
5,43
5,48
5,52
5,54
5,59
5,60
5,62
5,71
5,78
5,83
5,88
6,12
7,17
7,18
7,24
7,28
7,33
7,56
7,62
7,72
7,78
7,92
8,44
8,55
8,70
9,24
9,62
9,67
9,69
9,92
10,13
10,17
10,18
10,21
10,22
10,28
10,38
10,43
10,44
10,46
10,48
10,52
10,53
10,54
10,56
10,59
10,60
10,62
10,71
10,72
10,78
10,79
10,88
10,92
11,13
11,15
11,19
11,20
11,27
11,38
11,40
11,41
11,44
11,48
11,49
11,51
11,55
11,62
11,64
11,73
11,78
11,82
12,34
12,38
12,44
12,55
12,62
12,92
13,17
13,18
13,21
13,22
13,26
13,27
13,28
13,29
13,38
13,40
13,43
13,44
13,45
13,48
13,50
13,52
13,54
13,55
13,56
13,57
13,59
13,60
13,62
13,64
13,66
13,68
13,71
13,72
13,73
13,76
13,77
13,78
13,80
13,81
13,82
13,84
13,85
13,88
13,90
13,92
14,18
14,36
14,79
15,27
15,41
15,44
15,73
16,60
16,78
16,88
17,18
17,22
17,24
17,28
17,30
17,33
17,35
17,38
17,43
17,44
17,48
17,56
17,60
17,62
17,72
17,78
17,79
17,88
17,89
17,92
18,22
18,24
18,28
18,30
18,35
18,36
18,38
18,43
18,44
18,48
18,56
18,60
18,62
18,72
18,78
18,79
18,88
18,89
18,92
19,27
19,41
19,44
19,55
19,73
20,38
20,49
20,51
20,73
21,22
21,37
21,38
21,42
21,43
21,48
21,52
21,58
21,59
21,60
21,62
21,71
21,78
21,83
21,88
22,28
22,38
22,43
22,46
22,48
22,52
22,53
22,56
22,59
22,60
22,62
22,72
22,78
22,88
24,25
24,28
24,33
24,38
24,44
24,48
24,56
24,62
24,67
24,69
24,72
24,78
24,91
24,92
25,62
25,67
25,69
25,92
26,29
26,38
26,48
26,54
26,80
27,38
27,40
27,41
27,44
27,48
27,51
27,55
27,62
27,63
27,64
27,70
27,73
27,78
27,82
27,87
27,92
28,30
28,33
28,35
28,38
28,44
28,48
28,56
28,60
28,62
28,69
28,72
28,78
28,88
28,92
29,38
29,48
30,35
30,56
30,62
30,72
30,89
31,32
31,38
31,48
31,68
31,84
32,38
32,45
32,48
32,74
32,84
33,56
33,62
33,72
33,92
34,44
34,55
35,56
35,62
35,72
35,89
36,79
37,43
37,52
37,58
37,60
37,71
37,78
37,83
38,40
38,43
38,44
38,45
38,48
38,51
38,52
38,54
38,55
38,56
38,59
38,60
38,62
38,64
38,65
38,66
38,68
38,72
38,73
38,74
38,78
38,81
38,82
38,84
38,85
38,88
38,92
39,47
39,56
39,62
40,44
40,45
40,48
40,51
40,62
40,64
40,65
40,73
40,82
40,84
41,44
41,73
42,43
42,52
42,59
42,71
42,78
42,83
43,48
43,52
43,54
43,56
43,57
43,58
43,59
43,60
43,61
43,62
43,71
43,72
43,78
43,83
43,88
44,48
44,51
44,54
44,55
44,56
44,59
44,60
44,62
44,63
44,64
44,70
44,72
44,73
44,78
44,82
44,87
44,92
45,48
45,64
45,74
45,84
46,60
46,78
47,56
47,62
47,92
48,50
48,51
48,52
48,54
48,55
48,56
48,57
48,59
48,60
48,62
48,64
48,66
48,68
48,71
48,72
48,73
48,74
48,76
48,78
48,80
48,81
48,82
48,84
48,85
48,88
48,92
49,51
49,73
50,54
50,59
50,90
51,64
51,73
51,82
52,58
52,59
52,60
52,62
52,71
52,78
52,83
52,88
54,57
54,59
54,60
54,62
54,66
54,73
54,76
54,77
54,78
54,81
54,90
55,56
55,62
55,70
55,73
55,78
55,92
56,60
56,62
56,69
56,72
56,78
56,88
56,92
57,59
57,78
58,71
58,78
58,83
59,60
59,61
59,62
59,66
59,71
59,76
59,78
59,81
59,88
59,90
60,62
60,71
60,72
60,78
60,83
60,88
60,92
61,78
62,64
62,67
62,69
62,71
62,72
62,73
62,75
62,78
62,79
62,88
62,89
62,91
62,92
63,70
63,87
64,65
64,73
64,78
64,82
64,84
65,82
66,76
66,77
66,78
66,81
67,69
67,92
68,84
69,92
70,73
70,87
71,78
71,83
71,88
72,78
72,79
72,88
72,89
72,92
73,78
73,82
73,92
74,84
75,86
75,92
76,81
77,81
78,79
78,81
78,83
78,88
78,92
83,88
91,92
\end{filecontents*}
\csvreader[head to column names]{testEdge1.csv}{1=\vertLblA,2=\vertLblB}{%
      \draw[gray!60] (v\vertLblA) -- (v\vertLblB);
    }
    \csvreader[head to column names]{testVertex1.csv}{1=\vertLbl, 2=\radialCoord, 3=\angCoord}{%
      \draw[fill=blue,blue] (v\vertLbl) circle (1.5pt);
    }
  \end{tikzpicture}\hspace{2em}  
  \begin{tikzpicture}[scale=0.45]
    \path (0,0) coordinate (O);
    \draw (O) circle (8.0);
    
    \begin{filecontents*}{testVertex2.csv}
vertLbl, radialCoord, angCoord
0,6.68015810,-2.48770287
1,7.62187743,2.75731231
2,4.64633591,2.52539127
3,7.57032619,2.42998201
4,7.48124562,-3.02149228
5,7.74576595,-0.10096626
6,7.91180266,-0.16502821
7,5.35542089,-1.43032057
8,6.02904907,1.36323763
9,5.88444366,2.56734302
10,7.07353579,-0.90369193
11,7.75313977,0.11057676
12,5.85804137,-2.22112692
13,6.65727199,2.82726715
14,6.65390646,-2.12918400
15,7.94051023,2.77843725
16,7.85602013,1.82674072
17,7.74656019,2.11998675
18,7.61520149,-1.76788267
19,5.15003314,0.55670406
20,7.18390562,0.15216195
21,5.29521861,-2.09092167
22,6.87605131,2.88801951
23,5.40668748,-0.59267627
24,6.74228544,1.48886827
25,6.98125925,-0.01871206
26,6.36572532,0.85496944
27,7.20626564,-2.15047230
28,2.78698223,-1.55665759
29,4.46076738,0.88308173
30,5.67530601,2.62811473
31,7.06776206,-0.45524903
32,7.69062791,0.76922205
33,6.26400650,2.92894972
34,7.25443066,-1.64463473
35,5.68600992,0.87717131
36,7.80674265,-0.08580283
37,6.82489303,-0.27271511
38,7.29762905,1.21638786
39,6.57351007,3.00533575
40,6.74106656,2.19358547
41,7.95758850,0.48453237
42,6.44223448,-0.96748201
43,6.11283623,-2.53451033
44,6.87677263,1.26126696
45,6.65931021,0.94128570
46,6.65471768,-0.43104120
47,7.14059901,-0.18195604
48,7.95303147,2.00480847
49,5.50032684,0.64872044
50,6.42735559,-0.03885439
51,7.82226820,-1.65348317
52,7.49128394,-0.79587989
53,6.82261825,1.46087758
54,7.84396330,0.48375131
55,7.16077256,-1.08575176
56,7.21255798,2.54723215
57,5.71538758,1.11850635
58,7.73218365,-3.14025160
59,6.70191759,1.30452955
60,6.90783759,-2.72105667
61,7.17911060,1.58643884
62,7.96511677,2.83734286
63,7.38957467,1.60286075
64,7.77388261,-2.61559338
65,5.04802262,2.29012464
66,7.63608441,-1.58497546
67,4.72294627,-0.50652102
68,5.69916120,-2.90209740
69,7.70886936,1.20616630
70,6.62633055,-0.20874851
71,7.85106880,2.09657521
72,7.60496998,0.29660443
73,7.79753977,1.45957663
74,7.56170182,-3.01708045
75,5.55206176,-2.70471900
76,6.95260183,-1.23654500
77,6.42397319,-1.50324314
78,7.50584176,-2.54475683
79,7.72630691,-0.52141897
80,7.36508303,-2.46821010
81,7.53656748,-2.42866429
82,5.51825764,1.82933916
83,7.34102612,1.77017621
84,5.78012091,-1.06773696
85,7.40017240,3.05819317
86,7.27034604,-1.06864074
87,7.19574488,-1.94409460
88,6.52697806,2.35487816
89,7.63353377,0.86334053
90,6.75731445,0.26387646
91,6.94413563,-2.29889144
92,7.26166227,3.05110882
93,6.60752765,-2.93524914
94,7.87240353,2.13738168
95,7.88887492,-1.91357098
96,5.27548727,0.94470359
97,6.74469349,-0.20255068
98,6.89350983,0.52027670
99,6.80545801,1.32341171
100,7.76229291,-2.35659974
101,7.54992648,3.06712694
102,7.46573412,2.84786768
103,6.65659162,-1.28504698
104,6.99197211,0.19797673
105,7.84917646,-3.10308061
106,6.99876839,-0.69691283
107,7.26632992,0.23772549
108,7.61271596,1.84363431
109,5.09267061,2.76313729
110,2.95562043,-0.10841837
111,7.74103124,0.56507460
112,7.94123123,-0.07472686
113,7.42316740,3.11870064
114,7.93383652,0.84893540
\end{filecontents*}
\csvreader[head to column names]{testVertex2.csv}{1=\vertLbl, 2=\radialCoord, 3=\angCoord}{%
      \path (\angCoord*180/3.1416:\radialCoord) coordinate (v\vertLbl);
    }
    
    \begin{filecontents*}{testEdge2.csv}
vertLblA, vertLblB
0,28
0,43
0,75
0,78
0,80
0,81
1,2
1,13
1,15
1,30
1,109
2,3
2,9
2,13
2,28
2,30
2,33
2,40
2,56
2,65
2,88
2,109
2,110
3,65
3,88
4,68
4,74
4,93
5,36
5,50
5,110
5,112
6,47
6,70
6,97
6,110
7,28
7,66
7,76
7,77
7,84
7,103
7,110
8,24
8,29
8,44
8,53
8,57
8,59
8,73
8,99
9,30
9,56
9,65
9,88
9,109
10,28
10,42
10,84
11,20
11,110
12,14
12,21
12,27
12,28
12,91
13,15
13,22
13,30
13,33
13,62
13,102
13,109
14,21
14,27
14,28
15,109
16,82
16,108
17,40
17,65
17,71
17,94
18,28
19,26
19,28
19,29
19,35
19,41
19,49
19,54
19,96
19,98
19,110
19,111
20,104
20,110
21,27
21,28
21,87
21,91
21,110
22,33
22,39
22,62
22,102
22,109
23,28
23,31
23,46
23,67
23,79
23,106
23,110
24,53
24,61
24,73
25,36
25,50
25,110
25,112
26,29
26,32
26,35
26,45
26,49
26,89
26,96
26,110
26,114
27,28
28,29
28,34
28,42
28,43
28,51
28,55
28,65
28,66
28,67
28,68
28,75
28,76
28,77
28,84
28,86
28,87
28,91
28,95
28,96
28,103
28,109
28,110
29,32
29,35
29,44
29,45
29,49
29,57
29,89
29,96
29,98
29,110
29,114
30,56
30,65
30,109
31,46
31,67
31,79
31,110
32,35
32,49
33,39
33,92
33,102
33,109
34,51
34,66
35,45
35,49
35,57
35,89
35,96
35,110
35,114
36,50
36,110
36,112
37,47
37,67
37,70
37,97
37,110
38,44
38,57
38,59
38,69
39,85
39,92
39,101
39,109
40,65
40,94
41,54
41,98
42,55
42,84
42,86
42,110
43,64
43,75
43,78
43,80
43,81
44,57
44,59
44,69
44,99
45,57
45,89
45,96
46,67
46,110
47,70
47,97
47,110
49,96
49,98
49,110
49,111
50,110
50,112
53,73
54,98
55,84
55,86
56,109
57,59
57,69
57,96
57,99
57,110
58,105
58,113
59,99
60,68
60,75
61,63
62,102
62,109
64,75
65,82
65,88
65,94
65,109
65,110
66,77
67,70
67,79
67,84
67,97
67,106
67,110
68,74
68,75
68,93
70,97
70,110
71,94
72,90
72,107
72,110
74,93
75,93
76,84
76,103
79,110
80,81
82,83
82,108
84,86
84,103
84,110
85,92
85,101
85,113
87,95
89,96
89,114
90,104
90,107
90,110
91,100
92,101
92,113
96,110
96,114
97,110
98,110
98,111
101,113
102,109
104,107
104,110
106,110
107,110
110,112
\end{filecontents*}
\csvreader[head to column names]{testEdge2.csv}{1=\vertLblA,2=\vertLblB}{%
      \draw[gray!60] (v\vertLblA) -- (v\vertLblB);
    }
    \csvreader[head to column names]{testVertex1.csv}{1=\vertLbl, 2=\radialCoord, 3=\angCoord}{%
      \draw[fill=blue,blue] (v\vertLbl) circle (1.5pt);
    }
  \end{tikzpicture}
  \end{center}
  \caption{ Instances of $\HRG_{\alpha,\nu}(n)$ for $n=100$, $\nu = 1.832$, $\alpha=0.6$ (left) and $\alpha=0.9$ (right).}\label{fig:hrgExamples}
\end{figure}
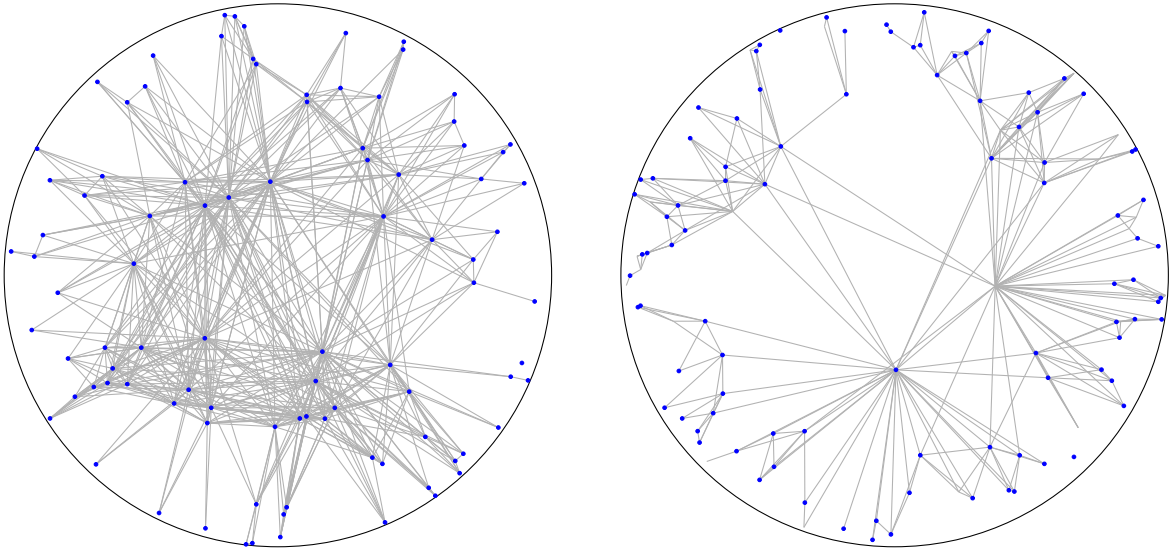  

We shall need the following useful approximation of the  density $f$. 
\begin{lemma}[{\cite[Equation (3)]{FriedrichK18}}] \label{lem:approxden}$f(r) = \alpha e^{-\alpha(R-r)}\cdot( 1 + \Theta(e^{-\alpha R} + e^{-2\alpha r} )) \cdot \mathbf{1}_{[0,R)}(r)$. 
\end{lemma}We denote the point set of the Poisson process by $V$, and we identify elements of $V$ with the vertices of a graph whose edge set $E$
is the collection of vertex pairs $uv$ such that $d_{\HH^2}(u,v)<R$.
The probability space over graphs $(V,E)$ thus generated is denoted by $\HRG_{\alpha,\nu}(n)$ and referred to as the HRG. 
Note in particular that $\EE|V|=n$ since  
\begin{equation}\label{eq:volumeofdisk}
\int_{B_O(R)}\lambda(r,\theta)\,\mathrm{d}\theta\mathrm{d}r = \nu e^{\frac{R}{2}}\int_{0}^{\infty}f(r)\,\mathrm{d}r = n.
\end{equation}
The parameter $\nu$ controls the average degree of $\HRG_{\alpha,\nu}(n)$
which, for $\alpha>\frac12$, is $(1+o(1))\frac{2\alpha^2\nu}{\pi (\alpha-1/2)^2}$ (see~\cite[Theorem 2.3]{GPP12}). 

\smallskip
The Hyperbolic Law of Cosines turns out to be inconvenient to work with when computing distances in hyperbolic space.
Instead, it is more convenient to consider the maximum angle $\theta_R(r_u,r_v)$ that two points $u,v\in B_O(R)$ can
form with the origin $O$ and still be within (hyperbolic) distance at most $R$ provided $u$ and $v$ are at distance
$r_u$ and $r_v$ from the origin, respectively.
\begin{remark}\label{prelim:rem:monotonicity}
Replacing in~\eqref{eq:hypdist} the terms $d_{\HH^2}(u,v)$ by $R$ and $\theta_u-\theta_v$
by $\theta_R(r_u,r_v)$, taking partial derivatives on both sides with respect to $r_u$ and 
some basic arithmetic gives
that the mapping $r\mapsto\theta_R(r,r_v)$ is continuous and
strictly decreasing in
the interval $[0,R)$.
Since $\theta_R(r_u,r_v)=\theta_R(r_v,r_u)$, the same is true of the
  mapping $r\mapsto\theta_R(r_u,r)$. (See~\cite[Remark 2.1]{KM19} for additional details.)
\end{remark}
The following estimate of $\theta_R(r,r')$, due to Gugelmann, Panagiotou \& Peter is very useful and accurate
(especially when $R-(r+r')$ goes to infinity with $n$).
\begin{lemma}[{{\cite[Lemma 6]{GPP12}}}]\label{prelim:lem:angles}
    If $0\leq r\leq R$ and $r+r'\geq R$, then
    $\theta_R(r,r')=2e^{\frac12(R-r-r')}(1+\Theta(e^{R-r-r'}))$.
\end{lemma}
Clearly, if $r+r'<R$, then $\theta_R(r,r')=\pi$.

\smallskip
We will need estimates for measures of regions of the hyperbolic plane, and specifically for the measure of balls.
We denote by $\mu(S)$ the measure of a set $S\subseteq\HH^2$, i.e.,
$\mu(S):=\int_{S}f(r,\theta)\,\mathrm{d}r\mathrm{d}\theta$.
The following approximation of the measures of the ball of radius $\rho$ centered at the origin and the ball of radius $R$ centered at $p\in B_O(R)$, both also due to Gugelmann et al., will be used frequently in our analysis.
\begin{lemma}[{\cite[Lemma 7]{GPP12}}]\label{prelim:lem:ballsMeasure}
    For $\alpha>\frac12$, $p\in B_O(R)$ and $0\leq r\leq R$ we have 
    \begin{align*}
       \mu(B_O(r)) & =e^{-\alpha (R-r)}(1+o(1)), \\
       \mu(B_p(R)\cap B_O(R)) & = \frac{2\alpha e^{-\frac12 r_p}}{\pi(\alpha-\tfrac12)}\big(1+\calO(e^{-(\alpha-\frac12)r_p}+e^{-r_p})\big).
    \end{align*}    
    Moreover, if $0\leq \rho\leq R$ and $r_p+\rho\geq R$, then
    \[
    \mu(B_p(R)\cap B_O(\rho)) = \Theta(e^{\frac12(R-r_p)-(\alpha-\frac12)(R-\rho)}/n).
    \]
\end{lemma}

Henceforth, the connected component that contains the vertex of $\calG_{\alpha,\nu}(n)$ that is closest to the origin will be referred to as the \emph{center component} of $\calG_{\alpha,\nu}(n)$ and denoted $\calC_{\alpha,\nu}(n)$.
 
Next, we state a result that is implicit in~\cite{BFM15} (later refined in~\cite{FM18}) concerning the size
and the `geometric location' of the giant component of $\HRG_{\alpha,\nu}(n)$.
In particular, it implies that a.a.s.~the center component and the giant component of $\calG_{\alpha,\nu}(n)$ are the same.
\begin{theorem}[{\cite[Theorem~1.4]{BFM15}, \cite[Theorem~1.1]{KM19}}]\label{prelim:thm:giantEqualCentralComp}
If $\frac12<\alpha<1$, then a.a.s.~the center component of $\calG_{\alpha,\nu}(n)$ contains every vertex in $B_O(\frac{R}{2})$, has size $\Theta(n)$, and all other connected components of  
$\HRG_{\alpha,\nu}(n)$ are of size polylogarithmic in $n$. 
\end{theorem}
The previous result partly explains why we focus our attention on simple random walks in the center component of $\HRG_{\alpha,\nu}(n)$. In the following remark we justify formally why we, henceforth, consider both the giant and the center component as being the same component, and consequently treat the latter as being the former.

\begin{remark}\label{rem:central4Giant} Let $\mathcal{S}_n$ be the event that the giant (the largest component) is equal to the center component $\calC:=\calC_{\alpha,\nu}(n)$ of $\HRG_{\alpha,\nu}(n)$, then $\PP(\mathcal{S}_n)=1-e^{-\Omega(n)}$ by~\cite{BFM15}. It follows immediately that all of our results holding a.a.s.~for the center component also hold for the giant component. For statements of the form $\EE(X(\mathcal{C}))$, where $X(G)$ is a function of a graph satisfying $1 \leq  X(G)  \leq |V(G)|^\kappa $ for some fixed $\kappa\geq 1$, for example (non-trivial) cover/hitting times, the results also carry over. That is, if $\mathcal{C}'$ is the giant of the HRG, then $\EE(X(\calC'))= (1+o(1))\EE(X(\calC))$. To see this, since $\EE(|V(\mathcal{G})|^\kappa)=\calO(n^\kappa)$ by  Lemma~\ref{prelim:lem:poimoment}, we have the following by Cauchy--Schwarz,  \[\EE(X(\calC')) \leq \EE(X(\calC')\mathbf{1}_{\mathcal{S}_n}) + \sqrt{\EE|V(\calC')|^{2\kappa} \cdot \PP(\mathcal{S}_n^c)} \leq \EE(X(\calC)) + \calO(n^{\kappa}e^{-\Omega(n)}) = (1+o(1))\EE(X(\calC)).\] 
 This also holds with the roles of $\calC'$ and $\calC$ reversed, giving the result. \end{remark}

The condition $\alpha>\frac12$ guarantees that $\HRG_{\alpha,\nu}(n)$ has constant average degree depending on $\alpha$ and $\nu$ only~\cite[Theorem 2.3]{GPP12}. 
If on the contrary $\alpha \leq \frac12$, then the average degree grows with~$n$.  
If $\alpha>1$, the largest component of $\HRG_{\alpha,\nu}(n)$ is sublinear in $n$~\cite[Theorem~1.4]{BFM15}. For $\alpha=1$ whether the largest component is of size linear in $n$ depends on $\nu$~\cite[Theorem~1.5]{BFM15}. Hence, the parameter range where $\frac12<\alpha<1$ is where HRGs are always sparse, exhibit power law degree distribution with exponent between $2$ and $3$ and the giant component is, a.a.s., unique and of linear size. All these are characteristics ascribed to so-called `social' or  `complex networks' which HRGs purport to model. Our main motivation is to contribute to understanding processes that take place in complex networks, many of which, as already discussed in the introduction, either involve or are related to simple random walks on such networks. Thus, we restrict our study exclusively to the case where $\frac12<\alpha<1$, but in order to avoid excessive repetition, we omit this condition from the statements we establish. 

The last known property of the HRG that we recall is that, w.h.p.~the center component has a number of edges that is linear in $n$.
\begin{lemma}[{\cite[Corollary~17]{KiwiMitscheSpectral}}]\label{prelim:lem:volumeCenterComp}For $\frac12<\alpha<1$, w.h.p.~ $|E(\calC_{\alpha,\nu}(n))|= \Theta(n)$.
\end{lemma}
 To conclude this section we point out that everything we do throughout this paper can be easily adapted to the case where $\HH^2$ has  Gaussian curvature $-\eta^2$ instead of $-1$ and all stated results hold provided $\alpha$ is in the range $(\eta/2,\eta)$ and is replaced by $\alpha/\eta$ everywhere. 

\subsection{Random Walks}\label{sec:rw} The simple random walk $(X_t)_{t\geq 0}$ on  a graph $G=(V,E)$ is a discrete time random process on $V$ which at each time moves to a neighbour of the current vertex~$u\in V$ uniformly with probability $1/d(u)$. We use $\mathbf{P}\left(\cdot \right):= \mathbf{P}^{G}\left(\cdot \right)$ to denote the law of the random walk on a graph $G$ (as opposed to $\mathbb{P}$ for the random graph). For a probability distribution $\mu$ on $V$ we let $\Pruh{\mu}{G}{\cdot } := \Prh{G}{\cdot  \mid X_0\sim \mu }$ be the conditional probability of the walk on $G$ started from a vertex sampled from~$\mu$. In the case where the walk starts from a single vertex $u\in V$ we write~$u$ instead of~$\mu$, for example $\Exuh{u}{G}{\cdot} := \Exh{G}{\cdot\mid X_0=u}$. We shall drop the superscript $G$ when the graph is clear from the context. We now define the \emph{cover time} $t_{\mathsf{cov}}(G)$ of  $G$ by \[t_{\mathsf{cov}}(G) := \max_{u\in V}\Exuh{u}{G}{\tau_{\mathsf{cov}}}, \qquad \text{where} \qquad \tau_{\mathsf{cov}} :=\inf\Big\{t\geq 0 : \bigcup_{i=0}^t\{X_i\} = V(G) \Big\}.\] 
Similarly, for $u,v\in V$ we let $\Exu{u}{\tau_v}$, where $\tau_{v}:=\inf\{t\geq 0\mid X_t =v\}$, be the \emph{hitting time} of~$v$ from~$u$. We shall use $\pi$ to denote the \emph{stationary distribution} of a simple random walk on a connected graph $G$, given by $\pi(v) := \frac{d(v)}{2|E|}$ for $v\in V$. Let\[ t_{\mathsf{hit}}(G):=\max_{u,v \in V} \; \Exuh{u}{G}{\tau_v},\quad \text{and} \quad   t_{\odot}(G):= \sum_{ u,v \in V(G)}\Exuh{u}{G}{\tau_v}\pi(u)\pi(v)  ,   
\] 
denote the \emph{maximum hitting time} and the \emph{target time}, respectively. We define each of these times to be $0$ if $G$ is either the empty graph or consists of just a single vertex. The target time $t_{\odot}(G)$, also given by  $\Exuh{\pi}{G}{\tau_\pi}$, is the expected time for a random walk to travel between two vertices sampled independently from the stationary distribution~\cite[Section 10.2]{peresmix2}. We will consider the random walk on the center component $\calC:=\calC_{\alpha,\nu}(n)$ of the HRG and so each of the expected stopping times $t_{\mathsf{cov}}(\calC), t_{\mathsf{hit}}(\calC)$ and $t_{\odot}(\calC)$ will in fact be random variables.

Formally, the commute time between vertices $u,v$ is $\Exu{u}{\tau_{u,v}}=\Exu{u}{\tau_{v}}+ \Exu{v}{\tau_{u}}$.

\subsection{Electrical Networks \& Effective Resistance}\label{sec:electralnetworks}
An electrical network, $N:=(G,C)$, is a graph $G$ and an assignment of conductances $C:E(G)\rightarrow \mathbb{R}^+$ to the edges of $G$. For an undirected graph $G$ we define $\vec E(G):=\left\{\vec{xy} \mid  xy\in E(G)\right\}$ as the set of all possible oriented edges for which there is an edge in $G$. For some $S,T \subseteq V(G)$, a \emph{flow} from $S$ to $T$ (or \emph{$(S,T)$-flow}, for short) on $N$ is a function $f:\vec E(G)\rightarrow \mathbb{R}$ satisfying $f(\vec{xy}) = -f(\vec{yx})$ for every $xy \in E(G)$ as well as Kirchhoff's node law for every vertex apart from those that belong to $S$ and $T$, i.e.
\begin{equation*}
\sum\limits_{u \in N(v)} f(\vec{uv}) = 0 \qquad \text{for each $v \in V\setminus (S\cup T)$.}
\end{equation*}
A flow from $S$ to $T$ is called a \emph{unit} flow if, in addition, its strength is  $1$, that is,
\[
\sum\limits_{s\in S}\sum\limits_{u \in N(s)} f(\vec{su}) =   
1. \]
We say that the $(S,T)$-flow is \emph{balanced} if 
\[
\sum_{u\in N(s)} f(\vec{su})=\!\!\sum_{u\in N(s')} f(\vec{s'u})
\ \text{ and }
\sum_{u\in N(t)} f(\vec{ut})= \!\!\sum_{u\in N(t')} f(\vec{ut'})
\qquad
\text{for all $s,s'\in S$ and $t,t'\in T$.}
\]
When $S=\{s\}$ and $T=\{t\}$ we write $(s,t)$-flow instead of $(\{s\},\{t\})$-flow. Note that $(s,t)$-flows are always balanced. For the network $N:=(G,C)$ and a flow $f$ on $N$ define the \emph{energy dissipated} by $f$,
denoted~$\mathcal{E}(f)$, by 
\begin{equation}\label{prelim:eqn:energyDef}
    \mathcal{E}(f) := \sum\limits_{e \in \vec E(G)}\frac{f(e)^2}{2C(e)},
\end{equation}
For future reference, we state the following easily verified fact: 
\begin{claim}\label{prelim:clm:concatFlows}
  Let $N:=(G,C)$ be an electrical network and $S,M,T\subseteq V(G)$.
  If $f$ and $g$ are balanced $(S,M)$- and $(M,T)$-flows on $N$, respectively,
  then
  $f+g$ is a balanced $(S,T)$-flow on $N$ and $\calE(f+g)\leq 2(\calE(f)+\calE(g))$.
  Moreover, if $f$ and $g$ are unit flows, so is $f+g$.
\end{claim} 

For $S,T\subseteq V(G)$, the \emph{effective resistance} between
$S$ and $T$, denoted $\Resu{C}{S}{T}$, is 
\begin{equation}\label{eq:resdef}
\Resu{C}{S}{T}:= \inf\left\{\mathcal{E}(f) \mid f \text{ is a unit flow from $S$ to $T$} \right\}. \end{equation}

The set of conductances $C$ defines a reversible Markov chain~\cite[Chapter 2]{LyonsPeres}. In this paper we shall fix $C(e)=1$ for all $e \in E(G)$ as this corresponds to a simple random walk. In this case, we write $\Resu{G}{S}{T}$ (or $\Res{S}{T}$ if the graph is clear) instead of $\Resu{C}{S}{T}$  and $\Resu{}{s}{t}$ instead of $\Resu{}{\{s\}}{\{t\}}$. The following is well known. 
\begin{proposition}[{\cite[Corollary 10.8]{peresmix2}}]\label{prop:resmetric}
 The effective resistances form a metric space on any graph, in particular they satisfy the triangle inequality.
\end{proposition} Choosing a single shortest path $P$ between any two vertices $s,t$ (if one exists) in a network (with $C(e)=1$ for each $e\in E$) and routing a unit flow down the edges of $P$ we obtain, straight from the definition~\eqref{eq:resdef} of $\Res{s}{t}$, the following basic but useful result.  
\begin{lemma}[{\cite[Lemma 3.2]{Chandra}}]\label{lem:basicresBdd}
For any graph $G=(V,E)$ and $s,t\in V$, we have $\Resu{}{s}{t} \leq d_G(s,t)$.
\end{lemma}

Another very useful tool is Rayleigh's monotonicity law (RML). 
\begin{theorem}[Rayleigh’s Monotonicity Law {\cite[Theorem 9.12]{peresmix2}}]\label{thm:RML} If $\{C(e)\}_{e\in E}$ and $\{C'(e)\}_{e\in E}$ are sets
of conductances on the edges of the same graph $G$ and if $C(e)\geq  C'(e)$ for all $e\in E$, then  
\[ \Resu{C}{S}{T} \leq \Resu{C'}{S}{T}\qquad \text{for all  $S,T\subseteq  V(G)$}.\]
\end{theorem}

The following result allows us to relate hitting times to effective resistance. 
\begin{lemma}[{\cite[Proposition 10.7]{peresmix2}}]\label{lem:commutetime} For any connected graph $G$ and any pair  $u,v\in V(G)$, 
\[  \Exu{u}{\tau_{u,v}} =2|E(G)|\cdot \Res{u}{v}. \tag{Commute time identity} \]
\end{lemma} 

The commute time identity gives us the following simple expression for the target time. 

\begin{lemma}\label{lem:targetresistance}  For any connected graph $G$ with $|V(G)|\geq 2$ we have \[t_{\odot}(G) = \frac{1}{4|E(G)|}\sum_{u,v\in V(\calC)} \Res{u}{v}\cdot d(u)d(v)\geq \frac{1}{2}|V(G)|. \] 
\end{lemma}
\begin{proof}Observe that, by the definition of $t_{\odot}(\cdot)$ and Lemma~\ref{lem:commutetime}, if $|V(G)|\geq 2$ then  
	\begin{equation*}
	t_{\odot}(G)
	= \frac{1}{2} \sum_{u,v\in V(G)} \Exu{u}{\tau_{u,v}}\cdot \pi(u)\pi(v)
	= \frac{1}{{4}|E(G)|}\sum_{u,v\in V(G)} \Res{u}{v}\cdot d(u)d(v).
	\end{equation*} 
 For any two vertices $u,v\in V$ we have $\Res{u}{v}\geq 1/(d(u)+1) + 1/(d(v)+1)$ by \cite[Lemma 1.4]{Jonasson98}. 
 Thus, {since $G$ is connected}, 
 \[
 \Res{u}{v}\cdot d(u)d(v)\geq {\big(\frac{1}{d(u)+1}+\frac{1}{d(v)+1}\big)}\cdot d(u)d(v) \geq {\frac12(d(u)+d(v))}
 \]
 and the result follows {since $\sum_{w\in V(G)}d(w)=2|E(G)|$}. \end{proof}

\section{The Effective Resistance Between Typical Vertices of the Giant}\label{sec:effectiveResist}
The main goal of this section is to show that, in expectation, the effective resistance
between two uniformly chosen
vertices in the giant component of $\HRG_{\alpha,\nu}(n)$
is $\calO(1)$.
By~Theorem~\ref{prelim:thm:giantEqualCentralComp} 
it suffices to establish the claim for the center component of 
$\HRG_{\alpha,\nu}(n)$. 

This section is organized in six subsections. In Section~\ref{ssec:tiling}, we partition hyperbolic space into tiles and associate a forest-like structure to the collection of tiles. In Section~\ref{ssec:flow}, under some conditions on vertices $s, t$ in the center component of $\calG_{\alpha,\nu}$ and a tiling, we construct a candidate unit $(s,t)$-flow $f_{s,t}$ whose dissipated energy will give us an upper bound on the effective resistance between vertices $s$ and $t$. Section~\ref{ssec:validflow} is devoted to establishing conditions under which one can guarantee the existence of the flow $f_{s,t}$, and to determining an upper bound on the energy it dissipates. Unfortunately, for most pairs of vertices $s,t$ the existence of the flow~$f_{s,t}$ is not guaranteed. Significant additional effort is spent, in Section~\ref{ssec:effectResist}, showing that for almost all pairs of vertices $s, t$ belonging to the center component  there are vertices $s', t'$ for which: (i) the flow $f_{s',t'}$ exists and its dissipated energy is relatively small, and (ii) the effective resistance between $s$ and~$s'$ is relatively small, and similarly for the resistance between $t'$ and $t$. In Section~\ref{sec:tergettime}, we use the flow constructed in the earlier subsections to determine the expectation of the average (among pairs of vertices) effective resistance, and the expectation of the target time of the giant of $\HRG_{\alpha,\nu}(n)$. Finally, in Section~\ref{ssec:conc}, we show that both the expectation of the average of the effective resistances and the target time is concentrated.

\subsection{Tiling}\label{ssec:tiling}
We start by constructing an infinite tiling of $\HH^2$. A region of $\HH^2$ between two rays emanating from the origin $O$ will be called a \emph{sector}.
A sector centered at the origin with bisector $p$ and (central) angle $\theta$
will be denoted by $\Upsilon_p(\theta)$.
Formally,
\[
\Upsilon_p(\theta):=\{ q\in\HH^2 \mid -\tfrac12\theta\leq\theta_p-\theta_q<\tfrac12\theta\}.
\]
(The reason for using closed-open angle intervals for defining sectors
is due to a minor technical convenience and mostly inconsequential for the
ensuing discussion.)

A \emph{tiling} of $\HH^2$ is a partition of $\HH^2$ into regions topologically equivalent to a disk. 
The tiling will depend on a parameter $c$ (a positive constant) and will be 
denoted by $\calF(c)$. There will be a distinguished collection of tiles, called \emph{root} tiles, corresponding
to the elements of the equipartition of $B_O(\frac{R}{2})$ into $N_0$ sectors, hence each sector with
central angle $2\pi/N_0$, for
\begin{equation}\label{flow:eqn:defN0}
N_0 := \lceil 2\pi/\theta_R(\tfrac12(R+c),\tfrac12(R+c))\rceil 
\qquad
\text{where $c>0$.}
\end{equation}
Also, let
\begin{equation}\label{flow:eqn:defNi}
N_i:=2^iN_0\ \text{ and } \ \theta_i:=2\pi/N_i \qquad \text{for $i\in\NN$}.
\end{equation}
The rest of the tiling is specified recursively as follows:
For $i\in\NN\setminus\{0\}$ and $j\in\{0,...,N_i-1\}$
  let $\{\Upsilon_{i,2j},\Upsilon_{i,2j+1}\}$
  be an equipartition of $\Upsilon_{i-1,j}$
  (thus each sector has central angle $\theta_{i}$ and there is a total of $N_{i}=2N_{i-1}$
  sectors for a given $i$).
Also, for the same constant $c>0$ as above, let
\[h_{-1}:= 0, \qquad h_0=\tfrac{1}{2} \cdot R, \qquad 
h_1:=\tfrac12(R+c), \] and then recursively define \[
h_{i} := \sup\{ h \mid \theta_R(h,h)\geq \tfrac12\theta_{R}(h_{i-1},h_{i-1})\}
\quad\text{for $i>1$.}
\]
The fact that $h_i$ is well-defined is a direct consequence of the
continuity and monotonicity of the mapping $h\mapsto\theta_R(h,h)$
(see Remark~\ref{prelim:rem:monotonicity}). 
In particular, observe that the sequence $(h_i)_{i\in\NN}$ is monotone increasing and unbounded. Moreover, since 
\[
\theta_i=2\pi/N_i=2\pi/(2^iN_0)\leq\theta_R(h_0,h_0)/2^i=\theta_R(h_i,h_i)
\]
we have $\theta_i\leq\theta_R(h_i,h_i)$ for all $i\in\NN$.

We can now specify the rest of the tiling. 
Let the \emph{$(i,j)$-tile} be defined by 
\begin{equation}
    \label{eq:tiling}
T_{i,j} :=
\Upsilon_{i,j}\cap (B_{O}(h_{i})\setminus B_{O}(h_{i-1}))
\qquad
\text{for $i\in\NN\setminus\{0\}$ and $j\in\{0,...,N_{i}-1\}$.}
\end{equation} 
Note that \eqref{eq:tiling} remains valid even for $i=0$.
See Figure~\ref{fig:tiling}(a) for an illustration of the tiling.

It will also be convenient to introduce various terms that will simplify the exposition.
First, a tile~$T_{i,j}$ for which $j\in\{0,...,N_{i}-1\}$ will be referred to as a \emph{level $i$}
tile.
We say that~$T_{i,j}$ is the \emph{parent} tile of both $T_{i+1,2j}$ and $T_{i+1,2j+1}$
and refer to the latter two tiles as \emph{children} of tile~$T_{i,j}$ (root tiles are assumed to be their own parent).
We define the mapping $T\mapsto \pi(T)$ that associates to tile $T$ its parent tile and denote
the composition of $\pi$ with itself $s$ times by $\pi^s$.
A tile that belongs to $\{\pi^s(T) \mid s\in\NN\}$ will be called \emph{ancestor} of $T$.
A tile $T$ will be said to be a \emph{descendant} of another tile $T'$ if the latter is an ancestor
of the former.
Note that by definition, a 
tile is always an ancestor and a descendant of itself.
From a geometrical perspective, the ancestors of  a given tile~$T$ are all those tiles that 
intersect the line segment between the origin and any point in $T$.
We say that a tile~$T'$ is a \emph{sibling} of a tile $T$ of level~$i$, if the former is obtained
from the latter by rotation around the origin in either $-\theta_i$ or $+\theta_i$.
If the rotation is in $-\theta_i$ (respectively, $+\theta_i$) we say that $T'$ is a
\emph{clockwise} (respectively, \emph{anti-clockwise}) sibling of $T$.

One last piece of terminology/notation that we will rely on is the following:
Let $\{\Upsilon^{0}_{i,j},\Upsilon^{1}_{i,j}\}$ be the equipartition 
of $\Upsilon_{i,j}$ (along its bisector)
into two isometric sectors.
We let $T_{i,j}^{b}:=T_{i,j}\cap\Upsilon^{b}_{i,j}$ for $b\in\{0,1\}$.
Note that $\{T^{0}_{i,j},T^{1}_{i,j}\}$ is the `natural' equipartition of
$T_{i,j}$. We refer to the~$T^{b}_{i,j}$'s as \emph{half-tiles}.
Given a tile $T$ we call
$H$ the \emph{parent half-tile} of~$T$ if
it is a half-tile of the parent of~$T$ and a line segment
from the origin to any point in the interior of~$T$ intersects~$H$.

We extend the concept of levels and of (clockwise/anti-clockwise) sibling to half-tiles (for the latter concept, the only difference now is that when dealing with level $i$ half-tiles rotations are in $\pm\frac12\theta_i$).
Clearly, a half-tile $H$ has two sibling half-tiles, one of which is included in the tile containing $H$, say tile~$T$, and the other one which is not contained in $T$. We refer to the former half-tile, i.e., to~$T\setminus H$
as the \emph{twin} half-tile of $H$ (or twin of $H$, for short).

\begin{remark}\label{flow:remark:key}
  There are two key facts about tiles that are worth stressing.
\begin{enumerate}[(i)]
    \item Two points in a given tile are within distance at most $R$ of each
other (which follows from the fact that points in a level $i$ tile have radial coordinate at most $h_{i}$ and
that any two such points span an
angle at the origin which is at most $\theta_i\leq\theta_R(h_{i},h_{i})$).

\item Any point in a tile is within distance at most $R$ from
any point in its parent half-tile 
(again, this follows from the fact that a level $i$ tile and its parent half-tile
are both contained in~$B_{O}(h_i)$ and any two points in their union
span an angle at the origin which is at most~$\theta_i\leq\theta_R(h_i,h_i)$).
\end{enumerate}
\end{remark}

  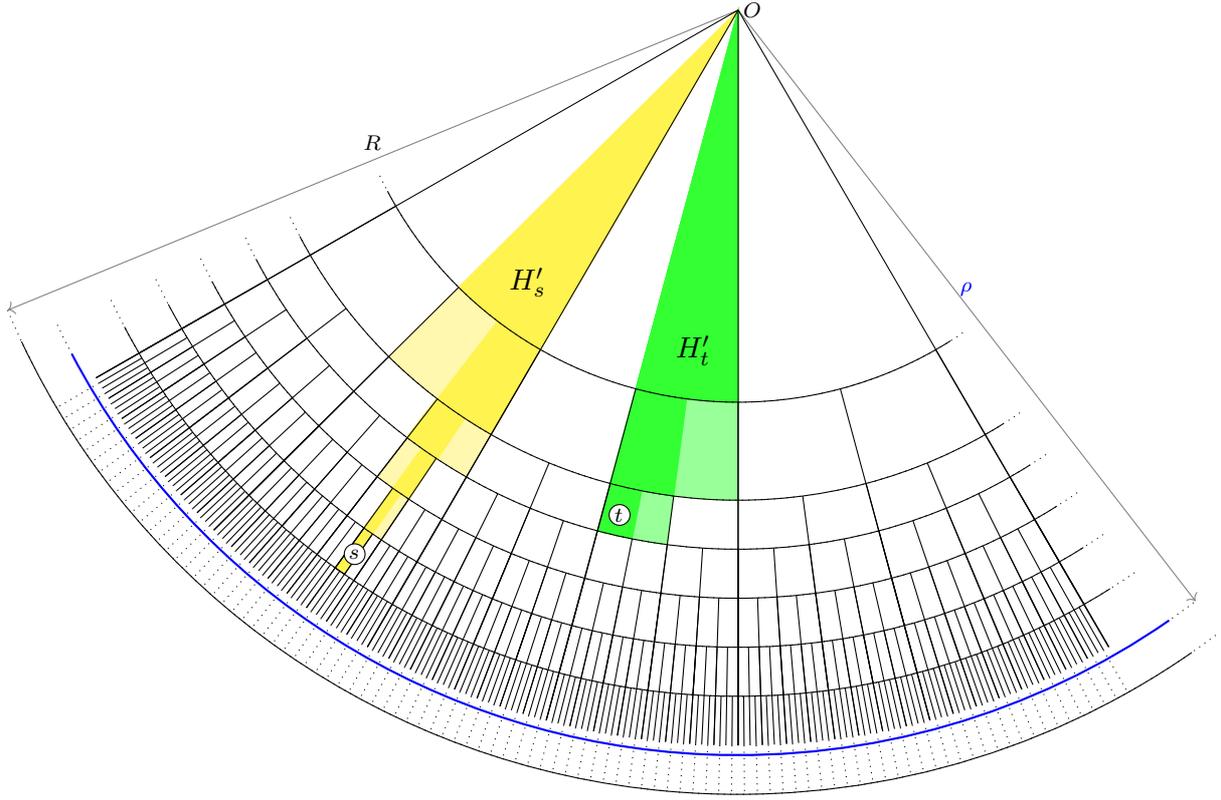
\begin{figure}
  \begin{tikzpicture}[scale=1.3, rotate=180]
      \def\c{8}
      \def\h{5}
      \def\n{5}
      \def\dlt{0.5}
      \node[black] (O) at (\c,\c) {$\quad _O$};
      \draw[fill=yellow!80] ({\c+(\h+4*\dlt)*cos(54.375)},{\c+(\h+4*\dlt)*sin(54.375)}) arc(54.375:55.3125:\h+4*\dlt) -- (\c,\c) -- cycle;
      \draw[fill=yellow!40, color=yellow!40] ({\c+(\h+3*\dlt)*cos(54.375)},{\c+(\h+3*\dlt)*sin(54.375)}) arc(54.375:56.25:\h+3*\dlt) -- (\c,\c) -- cycle;
      \draw[fill=yellow!80, color=yellow!80] ({\c+(\h+3*\dlt)*cos(54.375)},{\c+(\h+3*\dlt)*sin(54.375)}) arc(54.375:55.3125:\h+3*\dlt) -- (\c,\c) -- cycle;
      \draw[fill=yellow!40, color=yellow!40] ({\c+(\h+2*\dlt)*cos(52.5)},{\c+(\h+2*\dlt)*sin(52.5)}) arc(52.5:56.25:\h+2*\dlt) -- (\c,\c) -- cycle;
      \draw[fill=yellow!80, color=yellow!80] ({\c+(\h+2*\dlt)*cos(54.375)},{\c+(\h+2*\dlt)*sin(54.375)}) arc(54.375:56.25:\h+2*\dlt) -- (\c,\c) -- cycle;
      \draw[fill=yellow!40, color=yellow!40] ({\c+(\h+\dlt)*cos(52.5)},{\c+(\h+\dlt)*sin(52.5)}) arc(52.5:60:\h+\dlt) -- (\c,\c) -- cycle;
      \draw[fill=yellow!80, color=yellow!80] ({\c+(\h+\dlt)*cos(52.5)},{\c+(\h+\dlt)*sin(52.25)}) arc(52.25:56.25:\h+\dlt) -- (\c,\c) -- cycle;
      \draw[fill=yellow!40, color=yellow!40] ({\c+\h*cos(45)},{\c+\h*sin(45)}) arc(45:60:\h) -- (\c,\c) -- cycle;
      \draw[fill=yellow!80, color=yellow!80] ({\c+\h*cos(52.5)},{\c+\h*sin(52.5)}) arc(52.5:60:\h) -- (\c,\c) -- cycle;
      \draw[fill=yellow!80, color=yellow!80] ({\c+4*cos(45)},{\c+4*sin(45)}) arc(45:60:4) -- (\c,\c) -- cycle;
      \draw[fill=green!40, color=green!40] ({\c+(\h+\dlt)*cos(75)},{\c+(\h+\dlt)*sin(75)}) arc(75:82.5:\h+\dlt) -- (\c,\c) -- cycle;
      \draw[fill=green!80, color=green!80] ({\c+(\h+\dlt)*cos(75)},{\c+(\h+\dlt)*sin(75)}) arc(75:78.75:\h+\dlt) -- (\c,\c) -- cycle;
      \draw[fill=green!40, color=green!40] ({\c+\h*cos(75)},{\c+\h*sin(75)}) arc(75:90:\h) -- (\c,\c) -- cycle;
      \draw[fill=green!80, color=green!80] ({\c+\h*cos(75)},{\c+\h*sin(75)}) arc(75:82.5:\h) -- (\c,\c) -- cycle;
      \draw[fill=green!80, color=green!80] ({\c+4*cos(75)},{\c+4*sin(75)}) arc(75:90:4) -- (\c,\c) -- cycle;

      \draw[dotted] ({\c + (\c/2) * cos(25)},{\c + (\c/2) * sin(25)}) arc (25:125:\c/2);
      \draw ({\c + (\c/2) * cos(27.5)},{\c + (\c/2) * sin(27.5)}) arc (27.5:122.5:\c/2);
      \draw[dotted] ({\c + \c * cos(22.5)},{\c + \c * sin(22.5)}) arc (22.5:127.5:\c);
      \draw[black] ({\c + \c * cos(25.0)},{\c + \c * sin(25.0)}) arc (25.0:125:\c);
      \draw[dotted] ({\c + (\c-0.4) * cos(25)},{\c + (\c-0.4) * sin(25)}) arc (25:127.5:\c-0.4);
      \draw[blue, thick] ({\c + (\c-0.4) * cos(27.5)},{\c + (\c-0.4) * sin(27.5)}) arc (27.5:125:\c-0.4);
      \draw[->,gray] (\c,\c) -- ++(22.5:\c) node[midway,above,black] {$_R$};
      \draw[->,gray] (\c,\c) -- ++(127.5:\c-0.4) node[midway,above,blue] {$_\rho$};
    
      \foreach \i in {0,1,2,3,4} {
        \draw[dotted] ({\c + (\h+\i * \dlt) * cos(25)},{\c + (\h+\i * \dlt) * sin(25)}) arc (25:125:\h+\i * \dlt);
        \draw ({\c + (\h+\i * \dlt) * cos(27.5)},{\c + (\h+\i * \dlt) * sin(27.5)}) arc (27.5:122.5:\h+\i * \dlt);
      }
      \foreach \i in {1,2,...,4} {
        \draw (\c,\c) -- ++(360/12*\i:\h);
      }
      \foreach \i in {2,...,8} {
        \draw (O) ++(360/24*\i:4) -- ++(360/24*\i:4-\dlt);

      }
      \foreach \i in {4,5,...,16} {
        \draw (\c,\c) ++(360/48*\i:\h) -- ++(360/48*\i:3*\dlt);
      }
      \foreach \i in {8,...,32} {
        \draw (\c,\c) ++(360/96*\i:\h+\dlt) -- ++(360/96*\i:3*\dlt);
      }
      \foreach \i in {16,...,64} {
        \draw (\c,\c) ++(360/192*\i:\h+2*\dlt) -- ++(360/192*\i:3*\dlt);
      }
      \foreach \i in {32,...,128} {
       \draw (O) ++(360/384*\i:\h+3*\dlt) -- ++(360/384*\i:2*\dlt);
      }
      \foreach \i in {32,...,127} {
       \draw (O) ++(360/384*\i+360/768:\h+4*\dlt) -- ++(360/384*\i+360/768:\dlt);
       \draw[dotted] (O) ++(360/384*\i+360/768:\h+5*\dlt) -- ++(360/384*\i+360/768:\dlt);
      }

      \draw[fill=white] (\c+1.2,\c+\h+0.15) circle (3pt) node {$_t$};
      \draw[fill=white] (\c+3.88,\c+\h+0.55) circle (3pt) node {$_s$};
      
      \node at ({\c+3.5*cos(52.5)},{\c+3.5*sin(52.5)}) {$H'_s$};
      \node at ({\c+3.5*cos(82.5)},{\c+3.5*sin(82.5)}) {$H'_t$};
  \end{tikzpicture}
  \caption{(a) Partial illustration of a tiling of $B_O(R)$ (not at scale). (b) Illustration of flow between vertices $s$ and $t$ with no common ancestor tile. Coloured regions contain vertices through which flow from $s$ to $t$ is routed. Flow is pushed radially towards the origin $O$ from a yellow shaded tile to its parent half-tile. Flow is pushed in an angular direction from dark to light yellow shaded half-tiles. The direction of flow is reversed in the green shaded region.}\label{fig:tiling}
\end{figure}

Having described the tiling we will be working with, we can now compare it to the one introduced by Fountoulakis and M\"uller~\cite{FM18}.
In this article we work in the so-called \emph{Gans model}~\cite{gans66} or \emph{native model}~\cite{Krioukov2010}
of hyperbolic space, in contrast to~\cite{FM18} where the \emph{Poincar\'e half-plane model} is used. 
Although the two tilings are approximately equal, ours is a partition of $\RR^2$ instead of the 
half-plane, i.e., each tiling partitions the set used to represent hyperbolic space. Since both representations are alternative models of hyperbolic space, both tilings can be cast in one or the other. Doing so, one gets that the similarity of both tilings increases the further towards infinity their tiles are, i.e., further from the origin in the Gans model and closer to the half-plane boundary in the Poincar\'e half-plane model. For points close to the origin in the former or far from the boundary in the latter representation, both tilings differ significantly, although this is irrelevant for the analyses performed either here or in~\cite{FM18,MS19}. However, working with our tiling avoids having to perform, as in~\cite{FM18}, a coupling between the HRG and a point process in the half-plane  and also avoids some approximations incurred by working with the idealized model of M\"uller and Staps~\cite{MS19}. We believe this explains why we can guarantee that most of our results hold with probability $1-\calO(n^{-d})$ for all $d>0$ instead of the same probability but just for some $d>0$.

\subsection{Definition of Flows}\label{ssec:flow}
We now turn our attention to the construction of a unit $(s,t)$-flow, where $s$ and $t$ are any two distinct vertices in the center component $\calC_{\alpha,\nu}(n)$ of $\HRG_{\alpha,\nu}(n)$.
As already discussed, the energy dissipated by such flows yields bounds on the effective
resistance between~$s$ and $t$.  

Throughout, we assume $\frac12<\alpha<1$ and let $\calG:=\calG_{\alpha,\nu}(n)$, $V:=V(\calG)$ and $E:=E(\calG)$. 
In addition, we assume a constant $c>0$ is given which determines a tiling $\calF(c)$ as described
in Section~\ref{ssec:tiling}. All tiles and half-tiles mentioned should be understood with respect to $\calF(c)$. Also, given a vertex $w$ of $\calG$ we denote by $T(w)$ and $H(w)$ the (unique) tile and half-tile containing~$w$, respectively.

In order to make the $(s,t)$-flow construction simpler to grasp and analyze, we build it as a sum of three different flows. 
Specifically; a source flow $f_{s,H(s)}$, a  middle flow~$f_{H(s),H(t)}$ from the half-tiles containing $s$ to the one containing $t$, and a sink flow $f_{H(t),t}$, so that
\begin{equation}\label{flow:eqn:fst}
f_{s,t} := f_{s,H(s)} + f_{H(s),H(t)}+ f_{H(t),t}.
\end{equation}
The middle flow $f_{H(s),H(t)}$ will be obtained as a particular instance of a balanced flow $f_{H_s,H_t}$ that we define between pairs of half-tiles $H_s$ and $H_t$. The latter flow will also be useful later on when we consider local parameters of simple random walks in $\calG$.
 
The term~$f_{s,H(s)}$ will be a $(\{s\},V\cap H(s))$-flow that equally distributes the unit of flow that is injected in $s$ among all vertices in the half-tile $H(s)$ to which $s$ belongs. 
Similarly, the term~$f_{H(t),t}$ will be a balanced $(V\cap H(t),\{t\})$-flow that sends towards~$t$ all the flow at vertices in the half-tile~$H(t)$ to which~$t$ belongs.
Formally, we define $f_{s,H(s)}:\vec{E}(\calG)\to\RR$ and $f_{H(t),t}:\vec{E}(\calG)\to\RR$ as follows:
\begin{equation}\begin{aligned}\label{flow:eqn:startEnd}
f_{s,H(s)}(\vec{sv})& :=\frac{1}{|V\cap H(s)|}, \qquad \text{if $v\in H(s)$, $v\neq s$,} \\
f_{H(t),t}(\vec{ut})& :=\frac{1}{|V\cap H(t)|}, \qquad \text{if $u\in H(t)$, $u\neq t$.}
\end{aligned}\end{equation}
The mapping $f_{s,H(s)}$ is extended to all $\vec{E}(\mathcal{G})$ so it satisfies the antisymmetric property of flows, i.e., $f_{s,H(s)}(\vec{uv})=-f_{s,H(s)}(\vec{vu})$, and so it is null for all remaining oriented edges $\vec{uv}\in \vec{E}(\calG)$ where $f_{s,H(s)}(\vec{uv})$ remains undefined. 
The mapping $f_{t,H(t)}$ is similarly extended.

Note that the 2 terms $f_{s,H(s)}$ and $f_{H(t),t}$
end up being well-defined in the sense that no oriented edge is assigned multiple values.

Next, we define a flow $f_{H_s,H_t}$ from a half-tile $H_s$ to a half-tile $H_t$. The flow $f_{H(s),H(t)}$ will correspond to the particular case of $f_{H_s,H_t}$ where $H_s:=H(s)$ and $H_t:=H(t)$.
However, half-tiles $H_s, H_t$ can be arbitrary (the sub-indices are now unrelated to vertices $s$ and $t$, instead, they are used as labels of which half-tile is the source and which one the sink of the flow). We specify~$f_{H_s,H_t}$ as a sum of three terms $f'_{H_s,H_t}$, $f''_{H_s,H_t}$, and $f'''_{H_s,H_t}$.
There are a few special cases that need to be considered depending on whether $H_s$ and $H_t$ are ancestors of one another. We consider a specific case below (\emph{Case (i)} below).
In the other cases, among the 3 terms whose sum gives
$f_{H_s,H_t}$, some are null. We then discuss how to adapt the construction of~$f_{H_s,H_t}$ to the other cases.	\begin{labeling}{Case (iii):}

	\item[\emph{Case (i)}:] The half-tiles $H_s$ and $H_t$ do not share a common ancestor.
\end{labeling}  
For $H_w\in\{H_s,H_t\}$, let $H'_w$ be the half-tile contained in a root tile which intersects the ray between the origin and a point
in~$H_w$. Note that the choice of $H'_w$ is independent of the point in~$H_w$ that is picked.
Clearly, $H'_w$ is a half-tile of an ancestor of the tile 
that contains~$H_w$.
If~$H_w$ belongs to a root tile, then $H_w$ and $H'_w$ coincide.
The term $f'_{H_s,H_t}$ will be defined so that it determines a balanced flow from $H_s$ to $H'_s$.   
The underlying idea in specifying~$f'_{H_s,H_t}$ consists in pushing the flow from vertices in $H:=H_s$ to vertices in its twin half-tile $H'$, so that all vertices in the tile $T:=H\cup H'$ end up with the same amount of flow. Subsequently, flow is pushed through the oriented edges 
from vertices in $T$ to vertices in the parent half-tile  of $T$, so all edges involved carry exactly
the same amount of flow, and then repeating the process 
but taking $H$ as the parent half-tile of $T$, and so on and so forth until $H$ equals $H'_s$. 

Formally, we define $f'_{H_s,H_t}:\vec{E}(\calG)\to\RR$ so that if $H\neq H'_s$ and $T$ is the tile that contains~$H$ (thus, $T\setminus H$ is the twin of~$H$), then
\begin{equation}\label{flow:eqn:midAngular}
f'_{H_s,H_t}(\vec{uv}) := \frac{1}{|V\cap T|}\cdot\frac{1}{|V\cap H|}
\qquad \text{for all $u\in V\cap H$ and $v\in V\cap (T\setminus H)$.}
\end{equation}
If $T$ is a non-root tile which is an ancestor of the tile containing $H_s$ and $H$ is the parent half-tile of $T$, then
\begin{equation}\label{flow:eqn:midRadial}
f'_{H_s,H_t}(\vec{uv}) := \frac{1}{|V\cap T|}\cdot\frac{1}{|V\cap H|}
\qquad \text{for all $u\in V\cap T$ and $v\in V\cap H$.}
\end{equation}
The value of $f'_{H_s,H_t}$ is extended to all of $\vec{E}(\calG)$ in the same way as $f_{s,H(s)}$ was extended before.
Note that when~$H_s$ is contained in a root tile, we get that $f'_{H_s,H_t}$ is zero everywhere.

The term $f'''_{H_s,H_t}$ is defined as $f'_{H_s,H_t}$ but replacing $H_s$ and $H'_s$ by $H_t$ and $H'_t$, respectively, and reversing the direction of the flow in each oriented edge (i.e., multiplying the value of the flow of every oriented edge by $-1$).

It only remains to specify $f''_{H_s,H_t}$ which corresponds to the balanced $(V\cap H'_s,V\cap H'_t)$-flow obtained by
equally distributing the flow at vertex $v\in V\cap H'_s$ among the oriented edges leaving~$v$ and incident to vertices in $V\cap H'_t$.
Formally, $f''_{H_s,H_t}:\vec{E}(G)\to\RR$ is defined as follows:
\begin{equation}\label{flow:eqn:mid}
f''_{H_s,H_t}(\vec{uv}) := \frac{1}{|V\cap H'_s|}\cdot\frac{1}{|V\cap H'_t|}
\qquad \text{for all $u\in V\cap H'_s$ and $v\in V\cap H'_t$.}
\end{equation}

This concludes Case (i) when the tiles containing $H_s$ and $H_t$ have no common ancestors, and thus we need to route flow through vertices belonging to $B_O(\frac{R}{2})$.  
If the tiles containing $H_s$ and~$H_t$ share a common ancestor, then we distinguish several cases. 

	\begin{labeling}{Case (iii):}
	\item[\emph{Case (ii)}:] The half-tiles $H_s$ and $H_t$  belong to the same tile.  
\end{labeling}  
There are two subcases, one where $H_s$ and $H_t$ are twin half-tiles and the other when~$H_s$ and~$H_t$ are the same half-tile. In the former case, we set $f'_{H_s,H_t}$ and $f'''_{H_s,H_t}$ to be zero everywhere and set the flow $f''_{H_s,H_t}$ as in Case~(i). For the other subcase, we set $f_{H_s,H_t}=0$, that is, $f'_{H_s,H_t}$, $f''_{H_s,H_t}$ and $f'''_{H_s,H_t}$ are set to zero everywhere.  
\begin{labeling}{Case (iii):}
\item[\emph{Case (iii)}:] The half-tiles $H_s$ and $H_t$  belong to distinct tiles that share a common ancestor tile. 
\end{labeling} If $H_s$ and $H_t$ lie in a ray emanating from the origin, we can assume w.l.o.g.~that $H_t$ is closer to the origin than $H_s$. We let $H'_s:=H(t)$ and also $H'_t:=H(t)$, set
$f''_{H_s,H_t}$ and $f'''_{H_s,H_t}$ to be zero everywhere and let $f'_{H_s,H_t}$ be again as in Case~(i).
If $H_s$ and $H_t$ do not belong to half-tiles that lie along a ray emanating from the origin, then we set $H'_{s}$ and $H'_{t}$ as the two twin half-tiles whose union is the highest level tile among all common ancestor tiles of $H_s$ and~$H_t$ and define~$f_{H_s,H_t}$ as in Case~(i), that is, we $f'_{H_s,H_t}$, $f''_{H_s,H_t}$ and $f'''_{H_s,H_t}$ are fixed as in~Case~(i).

\smallskip
In all cases, the 3 terms $f'_{H_s,H_t}$, $f''_{H_s,H_t}$, $f'''_{H_s,H_t}$ are well-defined. 

\smallskip
Although our notation does not reflect it, the definition of $f_{s,t}$ depends on the tiling $\calF(c)$ and thus implicitly on $c$. If we need to stress the dependency, then we will say $f_{s,t}$ is $\calF(c)$-\emph{compatible}. Let $\calH_{s,t}$ be the collection of half-tiles $H$ of $\calF(c)$ such that $H\in\{H'_s,H'_t\}$ or
$H$ is a half-tile of an ancestor of either $T(s)$ or $T(t)$ but not of both.
We now show that $f_{s,t}$ as defined above is indeed a balanced unit $(s,t)$-flow. 
Although the proof is straightforward, we include it for completeness.
\begin{lemma}\label{flow:lem:fs0}
    Let $c>0$ be a constant that determines the tiling $\calF(c)$. If every half-tile in~$\calH_{s,t}$ contains a vertex of $\calG:=\HRG_{\alpha,\nu}(n)$, then each non-zero $f\in\{f_{s,H(s)},f_{H(s),H(t)},f_{H(t),t}\}$ is an~$\calF(c)$-compatible
    balanced unit flow in $\calG$, and so is
    $f_{s,t}:=f_{s,H(s)}+f_{H(s),H(t)}+f_{H(t),t}$.
\end{lemma}
\begin{proof}
The hypothesis that $V\cap H\neq\emptyset$ guarantees that none of the denominators that appear in~\eqref{flow:eqn:midAngular}-\eqref{flow:eqn:mid} is null.

The claim is obvious for $f_{s,H(s)}$ and $f_{H(t),t}$.
To show that it holds for $f_{H(s),H(t)}$ we argue that the claim holds for~$f'_{H(s),H(t)}$ (the case of $f'''_{H(s),H(t)}$ is analogous and the one of $f''_{H(s),H(t)}$ is obvious).
Assuming $f:=f'_{H(s),H(t)}$ is non-zero, the half-tiles
$H(s)$ and $H(t)$ cannot belong to the same tile.

Let $V:=V(\calG)$, $H_s:=H(s)$ and $H_t:=H(t)$.
We need to check that Kirchhoff's node law is satisfied at every vertex~$V\setminus(H_s\cup H'_s)$, that
the outgoing flow from every vertex $u$ in $H_s$ is $1/|V\cap H_s|$,
and that the flow into  every vertex $u$ in $H'_s$ is $1/|V\cap H'_s|$.
Throughout the rest of this proof, we let $f(u)$ denote the flow outgoing from $u$, 
i.e., $f(u):=\sum_{v\in N(u)}f(\vec{uv})$. 
	
	First, we verify that $f(u)=1/|V\cap H_s|$ for all $u\in H_s$. 
	Let $T_s$ be the tile containing $H_s$ and $H$ be the parent half-tile of $T_s$.
	Observe that by~\eqref{flow:eqn:midAngular}-\eqref{flow:eqn:midRadial}, since by construction of the tiling $\calF(c)$ we have $N(u)\cap (T_s\setminus H_s)=T_s\setminus H_s$ and $N(u)\cap H=H$,
	it follows that      
	\begin{align*}
	f(u) & = \sum_{v\in V\cap N(u)\cap (T_s\setminus H_s)} f(\vec{uv}) + \sum_{v\in V\cap N(u)\cap H} f(\vec{uv}) 
	\\ &
	= \sum_{v\in V\cap (T_s\setminus H_s)}\frac{1}{|V\cap T_s|}\cdot\frac{1}{|V\cap H_s|}
	+ \sum_{v\in V\cap H}\frac{1}{|V\cap T_s|}\cdot\frac{1}{|V\cap H|} \\
	& = \frac{1}{|V\cap T_s|}\cdot\frac{|V\cap (T_s\setminus H_s)|}{|V\cap H_s|}+\frac{1}{|V\cap T_s|}. 
	\end{align*}
	Since $|V\cap T_s|=|V\cap H_s|+|V\cap (T_s\setminus H_s)|$, it follows immediately that $f(u)=1/|V\cap H_s|$.
	
	Similarly, if $u\in H'_s$, one can verify that $f(u)=1/| V\cap H'_s|$.
	
	If $u\in V\setminus (H_s\cup H'_s)$, there are two subcases to consider.
	In the first one, $u$ belongs to a half-tile that intersects a ray starting at the origin $O$ and intersecting $H_s$. 
	Here, we let $H_a$ denote the parent half-tile of $T(u)$ and
	let $T_d$ be the tile whose parent half-tile is $H(u)$.
	Then, again by~\eqref{flow:eqn:midAngular}-\eqref{flow:eqn:mid},
	since by construction of the tiling $\calF(c)$ we have $N(u)\cap (T(u)\setminus H(u))=T(u)\setminus H(u)$,
	$N(u)\cap H_a=H_a$ and $N(u)\cap T_d=T_d$, it follows that
	\begin{align*}
	f(u) 
	&= \sum_{v\in V\cap N(u)\cap (T(u)\setminus H(u))} f(\vec{uv}) + \sum_{v\in V\cap N(u)\cap H_a} f(\vec{uv}) +
	\sum_{v\in V\cap N(u) \cap T_d}f(\vec{vu}) 
	\\ &	
	= \sum_{v\in V\cap (T(u)\setminus H(u))} \frac{1}{|V\cap T(u)|}\cdot\frac{1}{|V\cap H(u)|}+ 
	\sum_{v\in V\cap H_a}\frac{1}{|V\cap T(u)|}\cdot\frac{1}{|V\cap H_a|} 
	\\ & \qquad
	-
	\sum_{v\in V\cap T_d}\frac{1}{|V\cap T_d|}\cdot\frac{1}{|V\cap H(u)|}
	\\ & 
	= \frac{1}{|V\cap T(u)|}\cdot\frac{|V\cap (T(u)\setminus H(u))|}{|V\cap H(u)|}+ \frac{1}{|V\cap T(u)|}-\frac{1}{|V\cap H(u)|}
	\end{align*}
	Since $|V\cap (T(u)\setminus H(u))|=|V\cap T(u)|-|V\cap H(u)|$, 
	we get that $f(u)=0$ as sought.
	
	The other subcase, where $u$ does not belong to a half-tile that intersects a ray starting at the origin $O$ and intersecting $H_s$, is similar but simpler, and left to the reader. 

    Applying Claim~\ref{prelim:clm:concatFlows} establishes the stated result for $f_{s,t}$.
\end{proof}

\subsection{Existence of Flows and Energy Dissipated  
}\label{ssec:validflow}
Lemma~\ref{flow:lem:fs0} explains why we need to turn our attention to determining conditions under which $f_{s,t}$ exists (i.e., every half-tile of $\calF(c)$ contains a vertex), and more importantly is a good flow
in the sense that it  dissipates a small amount of energy.  
Clearly, the 
larger the number of vertices in each element of the half-tile collection $\calH_{s,t}$, the smaller the energy dissipated by the flow. To make this statement precise, we establish several
intermediate results whose aim is to quantify the expected number of vertices
per half-tile involved in the definition of $f_{s,t}$, and upper bound the probability 
that any of them contain far fewer than
the number of vertices one would expect.
\begin{claim}\label{flow:claim-h}
For any $\eps>0$, there exists a constant $c:=c(\epsilon)>0$ such that if $h_1:=\frac12(R+c)$,
\[
h_{i} := \sup\{ h \mid \theta_R(h,h)\geq \tfrac12\theta_{R}(h_{i-1},h_{i-1}) \}
\qquad \text{for all $i>1$,}
\]
then 
\[
  |h_{j}-h_{i}-(j-i)\ln 2|\leq \epsilon
  \quad \text{and}\quad
  \frac{\theta_i}{\sqrt{e^{\epsilon}}}\leq 2e^{\frac12(R-2h_i)}\leq \theta_i\sqrt{e^{\epsilon}} \qquad \text{for all $j\geq i\geq 1$.}
\] 
\end{claim}
\begin{proof}
  By Lemma~\ref{prelim:lem:angles},  there is a sufficiently large constant $c:=c(\epsilon)>0$
  such that if $a+b\geq R+c$, then for all sufficiently large $n$, 
  \[
  2e^{\frac12(R-a-b-\epsilon)} \leq \theta_R(a,b) \leq 2e^{\frac12(R-a-b+\epsilon)}.
  \]
  Hence, for $j\geq i\geq 1$, by definition of $h_s$, we have 
  $2^{j-i}e^{\frac12(R-2h_{j}+\epsilon)}\geq e^{\frac12(R-2h_{i}-\epsilon)}$,
  and moreover
  $2^{j-i}e^{\frac12(R-2h_{j}-\epsilon)}\leq e^{\frac12(R-2h_{i}+\epsilon)}$, so
  $h_{j}-(j-i)\ln(2)-\epsilon \leq h_{i} \leq h_{j}-(j-i)\ln(2)+\epsilon$.
  
  Finally, observe that from the above discussion, and given that
  $2h_i\geq R+c$ for all $i\in\NN\setminus\{0\}$, it follows
  that  
  $\theta_i\leq \theta_R(h_i,h_i)\leq 2e^{\frac12\epsilon}e^{\frac12(R-2h_i)}$.
  Moreover, since $N_0\leq 1+2\pi/\theta_R(h_0,h_0)$ and $\theta_R(h_i,h_i)=\theta_R(h_0,h_0)/2^i$, 
  \[
    \theta_i=2\pi/N_i=2\pi/(2^iN_0)
    \geq \theta_R(h_i,h_i)\cdot \big(1+\tfrac{1}{2\pi}\theta_R(h_0,h_0)\big)^{-1}.
  \]
  So, taking an even larger $c$ if needed, we can guarantee 
  that $\theta_R(h_i,h_i)\geq 2e^{-\frac14\epsilon}e^{\frac12(R-2h_i)}$
  and $1+\frac{1}{2\pi}\theta_R(h_0,h_0)\leq e^{\frac14\epsilon}$ to obtain the claimed lower bound on $\theta_i$.
\end{proof}
Next, we determine the expected number of vertices in each tile.
\begin{claim}\label{flow:claim:measure}
  Let $0<\epsilon<\frac{1}{2\alpha}\ln 2$ and $c:=c(\epsilon)$ be
  as in Claim~\ref{flow:claim-h}. Then, for 
  $V:=V(\HRG_{\alpha,\nu}(n))$ and any tile $T$ of level $i\geq 2$ in $\calF(c)$ such that $h_i<R$, we have
  \[
  \tfrac12\nu e^{(1-\alpha)(R-h_i)}\leq \EE|V\cap T| \leq 2\nu e^{(1-\alpha)(R-h_i)}.
  \]
  For $i=0$ and $i=1$, it holds that $\EE |V\cap T|=\Omega(e^{(1-\alpha)(R-h_i)})$.
\end{claim}
\begin{proof} Assume $T$ is a tile of level $i$,
  then $\mu(T)=\theta_i\mu(B_{O}(h_{i})\setminus B_{O}(h_{i-1}))$.
  By Lemma~\ref{prelim:lem:ballsMeasure}, we get 
  \begin{equation}\label{flow:claim:expect}
  \EE|V\cap T|=n\theta_i\mu(B_{O}(h_{i})\setminus B_{O}(h_{i-1}))
    = (1+o(1))n\theta_ie^{-\alpha (R-h_{i})}(1-e^{-\alpha(h_{i}-h_{i-1})}).
  \end{equation}
  Since $n=\nu e^{\frac{R}{2}}$ by definition of $R$, the fact that $\theta_0=2\theta_1=2\pi/N_0=\Theta(1)$, 
  $h_0-h_{-1}=\frac{R}{2}$ and $h_1-h_0=\frac{c}{2}$,
  both when $i=0$ and $i=1$ we get $\EE|V\cap T|=\Omega(e^{(1-\alpha)(R-h_i)})$ as claimed.
    
   For $i\geq 2$, Claim~\ref{flow:claim-h} and the fact that $n=\nu e^{\frac{R}{2}}$ imply  
    $e^{-\alpha(h_{i}-h_{i-1})}\geq (2e^{\epsilon})^{-\alpha}$
    and $n\theta_i\leq   2\nu\sqrt{e^{\epsilon}}e^{R-h_i}$. Thus,
    \begin{align*}
    \EE|V\cap T|
    & \leq 2\nu(1+o(1))\sqrt{e^{\epsilon}}(1-(2e^{\epsilon})^{-\alpha})e^{(1-\alpha)(R-h_i)} \\
    & \leq \nu(1+o(1))e^{-(\alpha-\frac12)\epsilon}(2e^{\alpha\epsilon}-1)e^{(1-\alpha)(R-h_i)}.
    \end{align*}
    Since $(e^{\alpha\epsilon}-1)^2>0$ and $e^{2\alpha\epsilon}\leq 2$ (by hypothesis regarding $\epsilon$), it follows that $2e^{\alpha\epsilon}-1\leq e^{2\alpha\epsilon}\leq 2$. 
    To conclude, take $n$ large enough so $1+o(1)\leq e^{(\alpha-\frac12)\epsilon}$
    above.
    The proof of the lower bound, for $i\geq 2$, is analogous.
\end{proof}
Henceforth, we say that a half-tile $H$ is \emph{sparse} if the number of
vertices it contains is fewer than half the ones expected, i.e.,
if $|V\cap H|<\frac12\EE|V\cap H|$.
We say that a tile $T$ is \emph{faulty} if either one of its two
  associated half-tiles is sparse.
  Since the number of points in a region $\Omega\subseteq B_{O}(R)$
  is distributed according to a
  Poisson distribution with mean $n\mu(\Omega)=\EE|V\cap\Omega|$, standard large deviation
  bounds for Poisson distributions yield the following result.

\begin{claim}\label{flow:claim:sparse}
  If $c>0$ and $V:=V(\HRG_{\alpha,\nu}(n))$, then
  the probability that a half-tile $H$ of $\calF(c)$ is sparse is at most
  $\exp(-\frac{1}{8}\EE|V\cap H|)$ and the probability that a tile $T$
  of $\calF(c)$ is faulty is at most~$2\exp(-\frac{1}{16}\cdot \EE|V\cap T|)$.
\end{claim}
\begin{proof}
  The first part is a direct application of Item~\eqref{prelim:lem:devBnd:itm1} of Lemma~\ref{prelim:lem:devBnd}.
  By definition of half-tile, the measure of $T$ is twice that of $H$, so we have 
  $\EE|T\cap V|=2\EE|H\cap V|$.
  Hence, the second part follows from the first part
  and a union bound.
\end{proof}
For $C>0$ a large constant to be determined, let
\begin{equation}\label{flow:eqn:rho}
\rho(C):=R-\frac{\ln(C\frac{R}{\nu})}{1-\alpha}.
\end{equation}
Using standard arguments concerning Poisson point processes we argue that the number of vertices in a half-tile~$H$ contained in a ball centered at the origin $O$ and of radius $\rho:=\rho(C)$ is concentrated close to its expected value, thus implying, w.h.p., that none of the tiles $T$ contained in $B_O(\rho)$ are faulty.  We state something slightly stronger as this is of use in Section~\ref{ssec:conc}.

\begin{lemma}\label{flow:lem:nonfaultyTiles}
Let $0<\epsilon\leq\frac12\ln 2$ and $c:=c(\epsilon)$ be as in Claim~\ref{flow:claim:measure}.
Then, for every $d>0$ there is a sufficiently large $C>0$ and 
$\rho:=\rho(C)$ such that, with
  probability $1-o(1/n^d)$, none of the half-tiles of $\calF(c)$ that intersect $B_{O}(\rho)$ contain less than a $3/4$ fraction of their expected points.
\end{lemma}
\begin{proof}
We begin by establishing that if the tile $T$ intersects $B_O(\rho)$, then 
$\EE|V\cap T| \geq (C/\sqrt{32})R$. Indeed, if $T$ is at level $i\geq 2$, then $T\cap B_{O}(\rho)\neq\emptyset$
implies that $h_{i-1}\leq\rho$.
By Claim~\ref{flow:claim-h} and our hypothesis regarding $\epsilon$, we have $h_{i-1}\geq h_{i}-\ln 2-\epsilon\geq h_{i}-\frac32\ln 2$.
Thus, by Claim~\ref{flow:claim:measure} and our choice of $\rho$, we get
\[
\mathbb{E}|V\cap T|\geq (\nu/2) e^{(1-\alpha)(R-h_i)}
\geq (\nu/2)e^{(1-\alpha)(R-\rho-\frac32\ln 2)}
\geq (C/\sqrt{32})R.
\]
If $T$ is a root tile, then 
\[
\EE|V\cap T|=\frac{\theta_0}{2\pi}\mu(B_O(\tfrac12R))=(1+o(1))(n/N_0)e^{-\alpha\frac{R}{2}}=(1+o(1))\nu e^{(1-\alpha)\frac{R}{2}}/N_0,
\] 
which for large $n$, since $N_0=\Theta(1)$, is at least $(C/\sqrt{32})R$ with room to spare. Similarly, if $T$ is a level $i=1$ tile, then $\EE|V\cap T|=(1+o(1))(n/N_1)e^{-\alpha\frac{R-c}{2}}=(1+o(1))\nu e^{\alpha\frac{c}{2}}e^{(1-\alpha)\frac{R}{2}}/N_1$ which, again for $n$ large, is at least $(C/\sqrt{32})R$ with plenty of slack. This completes the proof of this paragraph's opening claim.

  Let $H$ be a half-tile of $\calF(c)$ whose intersection with $B_O(\rho)$ is non-empty.
  By Claim~\ref{flow:claim:measure} and the assertion of the previous paragraph, 
  taking $C>0$ large enough, we can guarantee that
  $\EE|V\cap H|=\frac12\EE|V\cap T|\geq 16(d+1)R$.
  This implies, by Item~\eqref{prelim:lem:devBnd:itm1} of Lemma~\ref{prelim:lem:devBnd}, that $H$ contains less than $3/4$ of its expected points with probability at most $\calO(e^{-\frac{1}{4^2}\cdot \frac12\cdot  16(d+1)R})=\calO(1/n^{d+1})$.
  By a union bound, it is enough to argue that there are $o(n)$ tiles that intersect~$B_O(\rho)$.
  Indeed, the number of tiles at level $i$ 
  is $N_i$ as defined in~\eqref{flow:eqn:defN0} and~\eqref{flow:eqn:defNi}.
  Let~$\ell$
  be the largest integer such that $h_\ell\leq\rho$.
  Note that a tile intersecting $B_O(\rho)$ is of level at most~$\ell+1$. The number of tiles up to level $\ell+1$ is $\sum_{i=0}^{\ell+1}N_i=N_0\sum_{i=0}^{\ell+1}2^i<2^{\ell+2}N_0=\calO(2^\ell)$.
  Since by hypothesis $\epsilon\leq\frac12\ln 2$, Claim~\ref{flow:claim-h} 
  applies, so recalling the definition of $\rho$ from~\eqref{flow:eqn:rho} we get that
  \[
  R-\ln(C\tfrac{R}{\nu})\geq\rho\geq h_\ell
  \geq h_1+(\ell-1)\ln 2-\epsilon\geq\tfrac12R+(\ell-2)\ln 2.
  \]
  Since by definition $R=2\ln(n/\nu)$, we conclude that $\ell\leq \log_2(4n/(CR))$ and
  $2^{\ell}=o(n)$.
\end{proof}

\begin{remark}\label{rem:centerAndClique}
Since all vertices of $V(\calG_{\alpha,\nu}(n)\cap B_O(\frac{R}{2})$ are neighbours of each other, and vertices in a tile are neighbours of those in their parent half-tile (by Remark~\ref{flow:remark:key}), 
when none of the half-tiles of $\calF(c)$ that intersect~$B_O(\rho)$ is empty (in particular, if they contain a $3/4$ fraction of their expected points), then all vertices in $B_O(\rho)$ belong to the same connected component which must equal the center component (because by definition, the center component contains the vertex of $\calG_{\alpha,\nu}(n)$ that is closest to the origin).
\end{remark}

We say that \emph{access to $T$ is robust} (or \emph{$T$ is robust} for short) if none of its ancestors (including itself) is faulty. If access to $T$ is not robust we say that \emph{access to $T$ is frail}
(or \emph{$T$ is frail}).
Lemma~\ref{flow:lem:nonfaultyTiles} implies that w.h.p.~every
tile $T$ intersecting $B_{O}(\rho)$ is robust.
The condition $T\cap B_{O}(\rho)\neq\emptyset$ cannot be relaxed significantly,
so when dealing with tiles not intersecting $B_O(\rho)$ we have to settle for a weaker statement.
Indeed, we will show that, with non-negligible probability, if a tile is at least some
sufficiently large constant apart from the boundary of~$B_{O}(R)$, then it is robust.
Henceforth, for $C'>0$ let
\begin{equation}\label{flow:eqn:rhoPrime}
\rho'(C'):= R-\frac{\ln(\frac{2C'}{\nu})}{1-\alpha}.
\end{equation}
 For $\rho$ and $\rho'$ given by~\eqref{flow:eqn:rho} and~\eqref{flow:eqn:rhoPrime} we let  $\ell$ and $\ell'$ be the largest integers such that 
 \begin{equation}\label{validflow:eqn:ldef} h_{\ell}\leq\rho \qquad \text{ and } \qquad  h_{\ell'}\leq\rho'.\end{equation}
  
\begin{lemma}\label{flow:lem:outerTiles}
Let $\epsilon$ and $c:=c(\epsilon)$ be as in Lemma~\ref{flow:lem:nonfaultyTiles}.
Fix $C'\geq 32\ln 2$ and let $\rho':=\rho'(C')$, $\ell$, $\ell'$ be 
as defined in~\eqref{flow:eqn:rhoPrime}-\eqref{validflow:eqn:ldef}, respectively.  Then, for every level $i$ tile $T$ in the tiling $\calF(c)$ satisfying $\ell< i\leq\ell'$, the following holds:
  \begin{enumerate}[(i)]
  \item\label{flow:lem:outerTiles:itm1}
    $\EE|V\cap T|\geq C'e^{(1-\alpha)(h_{\ell'}-h_i)}$.
  \item\label{flow:lem:outerTiles:itm2}
  $\PP(\text{$T$ is frail} \mid \text{$\pi^{(i-\ell)}(T)$ is robust})=\calO(\exp(-\frac{1}{23}C'e^{(1-\alpha)(h_{\ell'}-h_i)}))$.
  \end{enumerate}
\end{lemma}
\begin{proof}
  For the first part, by Claim~\ref{flow:claim:measure}, the definition of $\ell'$, and as $h_i\leq h_{\ell'}\leq\rho'$, we have 
\[
\EE|V\cap T|
  \geq \tfrac12\nu e^{(1-\alpha)(R-h_{i})}
  = \tfrac12\nu e^{(1-\alpha)(R-h_{\ell'})}\cdot e^{(1-\alpha)(h_{\ell'}-h_i)}
  \geq C'\cdot e^{(1-\alpha)(h_{\ell'}-h_i)}.
  \]
  To establish the second part, observe that by Claim~\ref{flow:claim:sparse}, we have 
\begin{align*}
\PP(\text{$T$ is robust} \mid \text{$\pi^{(i-\ell)}$(T) is robust} ) & =
\prod_{p=0}^{i-\ell-1}\big(1-\PP(\text{$\pi^{p}(T)$ is faulty})\big) \\
& \geq 
\prod_{p=0}^{i-\ell-1}\big(1-2\exp\big(-\tfrac{1}{16}\EE|V\cap\pi^{p}(T)|\big)\big).
\end{align*}
By Item~\eqref{flow:lem:outerTiles:itm1}, if we take $C'\geq 32\ln 2$, then
$\exp(-\tfrac{1}{16}\EE|V\cap\pi^{p}(T)|)\leq \frac14$ for all
$p\in\{0,...,i-\ell-1\}$, so recalling that $\ln(1-x)\geq -\frac{x}{1-x_0}$ for $0\leq x\leq x_0<1$, 
\[
\ln\PP(\text{$T$ is robust} \mid \text{$\pi^{i-\ell}(T)$ is robust}) 
    \geq -4
    \sum_{p=0}^{i-\ell-1}\exp(-\tfrac{1}{16}\EE|V\cap\pi^{p}(T)|).
\]
Since by hypothesis $0<\epsilon\leq\frac12\ln 2$, Claim~\ref{flow:claim-h} holds, so we
have that 
\[
h_{\ell'}-h_{i-p}=(h_{\ell'}-h_i)+(h_i-h_{i-p})\geq (h_{\ell'}-h_i)+p\ln 2-\epsilon
\geq (h_{\ell'}-h_i)+(p-1)\ln 2.
\]
Thus, if we let $\delta:=(1-\alpha)\ln 2$ and
  $\gamma:=\frac{2^\alpha}{32}C'e^{(1-\alpha)(h_{\ell'}-h_i)}\geq\frac{1}{32}C'e^{(1-\alpha)(h_{\ell'}-h_i)}$, then applying Item~\eqref{flow:lem:outerTiles:itm1} of the present lemma and noting that $\pi^{p}(T)$ is a level $i-p$ tile, we get  
\begin{equation}\label{eq:TlGood}
\ln \PP(\text{$T$ is robust} \mid \text{$\pi^{i-\ell}(T)$ is robust} )
\geq -4\sum_{p=0}^{\ell-i-1} \exp\big(-\gamma e^{p\delta}\big).
\end{equation}
Next, since the mapping $x\mapsto e^{-\gamma e^{x\delta}}$ is non-increasing, applying the substitution $y:=\gamma e^{x\delta}$ gives 
\[
\sum_{p=1}^{i-\ell-1}\exp\big(-\gamma e^{p\delta}\big)
\leq \int_{0}^{\infty} e^{-\gamma e^{x\delta}}\,\mathrm{d}x
\leq \frac{1}{\delta}\int_{\gamma}^{\infty}y^{-1}e^{-y}\,\mathrm{d}y
\leq \frac{1}{\delta\gamma}e^{-\gamma},
\]
where the last inequality holds because
$\int_{a}^{\infty}y^{-1}e^{-y}\,\mathrm{d}y\leq\frac{1}{a}\int_{a}^{\infty}e^{-y}\,\mathrm{d}y$ for $a>0$.
Thus,
\begin{equation}\label{eq:xi}
\sum_{p=0}^{i-\ell-1}\exp\big(-\gamma e^{p\delta}\big)
\leq e^{-\gamma}
+
\frac{1}{\delta\gamma}e^{-\gamma}
=:\xi.
\end{equation}
Combining the above with~\eqref{eq:TlGood} we have $\PP(\text{$T$ is frail} \mid \text{$\pi^{i-\ell}(T)$ is robust}) \leq  1 -  e^{-4\xi} \leq 4\xi $. Since~$C'\geq 32\ln 2$,  the bound in~\eqref{eq:xi} gives
$4\xi\leq 4(1 + \frac{1}{\delta \gamma})e^{-\gamma} = \calO(e^{-\gamma})$, as we recall that $\gamma\geq\ln 2$  and $\delta:=(1-\alpha)\ln 2 >0$. The claimed result follows.
\end{proof}
Lemma~\ref{flow:lem:nonfaultyTiles} tells us that w.h.p.~every tile intersecting $B_O(\rho)$ is robust, which combined with Lemma~\ref{flow:lem:outerTiles} implies that a tile contained in $B_O(\rho')$ has a constant probability of being robust. The significance of this last implication
is given by our next result, which shows that between any two  vertices $s$ and $t$ belonging to either the same or distinct robust tiles, there is a flow whose dissipated energy is $\calO(1)$, and thence by definition of effective resistance
(see~\eqref{eq:resdef})  we deduce that $\Res{s}{t}=\calO(1)$.
\begin{proposition}\label{flow:prop:energy}
Let $\epsilon$, $c$, $C'$ and $\rho'$ be as in Lemma~\ref{flow:lem:outerTiles}.
  If vertices $s$ and $t$ of~$\calG:=\HRG_{\alpha,\nu}(n)$ belong to
  robust tiles of $\calF(c)$, both of which are contained in $B_{O}(\rho')$, then $f_{s,t}$ (as defined in Section~\ref{ssec:flow})  is a unit $(s,t)$-flow in $\calG$ satisfying
  \[
  \calE(f_{s,t})=\calO(1).
  \]
\end{proposition}
\begin{proof}
  The fact that $f_{s,t}$ is a unit $(s,t)$-flow in $\calG$ is a direct consequence of Lemma~\ref{flow:lem:fs0} (the necessary hypotheses are satisfied by definition of robust tile).

  Recall that $f_{s,t}$ equals the sum of flows $f_{s,H(s)}$, $f_{H(s),H(t)}$ and $f_{H(t),t}$ (as defined in Section~\ref{ssec:flow}).
  In order to establish the claimed bound for $\calE(f_{s,t})$ it suffices, by Lemma~\ref{prelim:clm:concatFlows}, to show that the same bound holds for each of $\calE(f_{s,H(s)})$, $\calE(f_{H(s),H(t)})$ and $\calE(f_{H(t),t})$.

  Assume $H(s)$ is at level $\ell_s$ of $\calF(c)$. Similarly, define $\ell_t$ but with respect to $H(t)$.

  Since $T(s)$ and $T(t)$ are robust tiles (by hypothesis), they
  are non-faulty. In particular, $|V\cap H(s)|$ and $|V\cap H(t)|$
  are at least half their expected value.
  Hence, by definition of $f_{s,H(s)}$ and $f_{H(t),t}$ (see~\eqref{flow:eqn:startEnd}), from Claim~\ref{flow:claim:measure}, we get that
  \[
  \calE(f_{s,H(s)}) =  \frac{1}{|V\cap H(s)|}=\calO\big(e^{-(1-\alpha)(R-h_{\ell_s})}\big), 
  \quad
  \calE(f_{H(t),t})    
  =\frac{1}{|V\cap H(t)|}= \calO\big(e^{-(1-\alpha)(R-h_{\ell_t})}\big).
  \]
  Since by hypothesis $H(s),H(t)\subseteq B_O(\rho')$ we have $h_{\ell_s}, h_{\ell_t}\leq \rho'=R-\calO(1)$ and the sought after bounds for $\calE(f_{s,H(s)})$ 
  and $\calE(f_{H(t),t})$ follow.

  To bound $\calE(f_{H(s),H(t)})$, recall that $f_{H(s),H(t)}$ equals the sum of flows $f'_{H(s),H(t)}$,
  $f''_{H(s),H(t)}$ and~$f'''_{H(s),H(t)}$ (as defined in
  Section~\ref{ssec:flow}).
  Thus, we again need only bound the energy dissipated by each of these three flows.
  We only discuss the subcase where $H(s)$ and $H(t)$ do not share 
  a common ancestor, the other cases being simpler are left to 
  the reader. Let~$H_s:=H(s)$, $H_t:=H(t)$ and 
  define $H'_s$ and $H'_t$ as in Section~\ref{ssec:flow}.
  For the subcase we are considering, both~$H'_s$ and $H'_t$ are half-tiles of level $0$ contained in an ancestor tile of $T(s)$ and $T(t)$, respectively. Since the latter are robust tiles (by hypothesis) the tiles containing $H'_s$ and $H'_t$ are non-faulty, hence $|V\cap H'_s|$ and $|V\cap H'_t|$ are at least
  half their expected value.  
  Directly from~\eqref{flow:eqn:mid} and 
  again using Claim~\ref{flow:claim:measure}, it follows that
  \[
  \calE(f''_{H(s),H(t)})    = \frac{1}{|V\cap H'_{s}|}\cdot\frac{1}{|V\cap H'_t|}
  = \calO\big(e^{-2(1-\alpha)(R-h_0)}\big) 
  = \calO\big(e^{-2(1-\alpha)\frac{R}{2}}\big).
  \]
 To upper bound $\calE(f'_{H(s),H(t)})$ note that the set of ancestors of $T(s)$ is $\{\pi^p(T(s)) \mid p\in\{0,...,\ell_s\}\}$.
  For $0\leq p<\ell_s$, let $H_{p+1}(s)\subseteq \pi^{p+1}(T(s))$ be the parent half-tile of $\pi^{p}(T(s))$.
  By~\eqref{flow:eqn:midAngular}-\eqref{flow:eqn:midRadial}, 
  \[
  \calE(f'_{H(s),H(t)}) = \sum_{p=0}^{\ell_s-1}\frac{1}{|V\cap H_p(s)|}\cdot\frac{1}{|V\cap (\pi^p(T(s))\setminus H_p(s))|}
  + \sum_{p=0}^{\ell_s-1}\frac{1}{|V\cap \pi^{p}(T(s))|}\cdot\frac{1}{|V\cap H_{p+1}(s)|}.
  \]
  Since $T(s)$ is by assumption a robust tile, none
  of its ancestors is faulty, so all 
  terms like $|V\cap \pi^{p}(T(s))|$, $|V\cap H_{p}(s)|$, $|V\cap (\pi^{p}(T(s))\setminus H_p(s))|$, are at least half
  their expected value, in particular none of them is $0$.
  By Claim~\ref{flow:claim-h}, we have $h_{\ell_s-p}\leq h_{\ell_s}-p\ln 2+\epsilon$ for every $0\leq p<\ell_s$,
  $h_1:=\frac{R}{2}+c$ and $h_0:=\frac{R}{2}$.
  Taking all the previous comments into account, using
  Claim~\ref{flow:claim:measure} and observing that $\pi^p(T(s))$ is a 
  level $\ell_s-p$ tile, we get, 
  \begin{align*}
  \calE(f'_{H(s),H(t)}) =
    \sum_{p=0}^{\ell_s}\calO\big(e^{-2(1-\alpha)(R-h_{\ell_s-p})}\big) 
    = \calO\Big(e^{-2(1-\alpha)(R-h_{\ell_s})}\sum_{p=0}^{\ell_s}\frac{1}{2^{2(1-\alpha)p}}\Big) 
 = \calO\big(e^{-2(1-\alpha)(R-h_{\ell_s})}).
  \end{align*}
  Similarly, we have
  $\calE(f'''_{H(s),H(t)})=\calO(e^{-2(1-\alpha)(R-h_{\ell_t})})$.
  Since $h_{\ell_s}, h_{\ell_t}\geq h_0=\frac{R}{2}$, we have
  \begin{align*}
  \calE(f_{H(s),H(t)})
  & =\calO(f'_{H(s),H(t)})+\calE(f''_{H(s),H(t)})+\calE(f'''_{H(s),H(t)}) \\
  & = \calO\big(e^{-(1-\alpha)(R-h_{\ell_s})}+e^{-(1-\alpha)(R-h_{\ell_t})}\big).
  \end{align*}
  The desired conclusion follows since, as already observed,
  both $H_s=H(s)$ and $H_t=H(t)$ are contained in $B_{O}(\rho')$, hence by hypothesis $h_{\ell_s}, h_{\ell_t}\leq \rho'=R-\calO(1)$.
\end{proof}

The following fact will be useful later on when we study commute times between given vertices of the HRG.
\begin{remark}\label{flow:remark:useful}
A careful inspection of the proofs of Lemma~\ref{flow:lem:fs0} and Proposition~\ref{flow:prop:energy} reveals that we have actually shown that given two half-tiles $H_s, H_t$ at levels $\ell_s, \ell_t$ of the tiling~$\calF(c)$, respectively, if they are both contained in robust tiles,
then $f_{H_s,H_t}$ (as defined in Section~\ref{ssec:flow}) is a balanced unit $(H_s,H_t)$-flow in $\calG_{\alpha,\nu}(n)$ such that $\calE(f_{H_s,H_t})=\calO\big(e^{-2(1-\alpha)(R-h_{\ell_s})}+ e^{-2(1-\alpha)(R-h_{\ell_t})}\big)$.
\end{remark}
By definition of effective resistance and the commute time identity, Proposition~\ref{flow:prop:energy} gives bounds on the hitting time between vertices belonging to robust tiles contained in~$B_O(\rho')$.
By Lemma~\ref{flow:lem:nonfaultyTiles} we know that w.h.p.~all vertices in $B_O(\rho)$ belong to robust tiles. However, most vertices of $\calG:=\HRG_{\alpha,\nu}(n)$ fall outside $B_O(\rho)$. In fact, a significant (but constant) fraction fall outside $B_O(\rho')$ even if we only take
into account the vertices belonging to the center component.
So, we are yet far from being able to bound either the average resistance or target time of the center component. A key step in this direction that we undertake in the next section is to study the average (graph) distance between vertices in the center component and robust tiles.

\subsection{Bounds on Distance to Robust Tiles}\label{ssec:effectResist}
In this section, we show that in expectation over $\calC:=\calC_{\alpha,\nu}(n)$, the graph distance between the vertices of $\calC$  located outside $B_{O}(\rho)$ and a vertex belonging to a robust tile is $\calO(1)$.
To achieve this goal, it will be useful to bound the probability that none of the level 
$\ell'$ descendants of a given tile $A$ is robust.
Unfortunately, the events corresponding to nearby tiles being robust are not
independent, which complicates our task.
Henceforth, we say a collection of tiles $\calT$ is \emph{frail} 
if each tile $T$ that belongs to $\calT$ is frail.
\begin{lemma}\label{flow:lem:frailDescendentOfAncestor}
Let $\epsilon$, $c:=c(\epsilon)$ be as in Lemma~\ref{flow:lem:outerTiles},  $C'>0$ sufficiently large,  and $\ell$ and $\ell'$ be given by~\eqref{validflow:eqn:ldef}.
If~$A$ is a level $i:=\ell'-p\geq\ell$ tile of the tiling~$\calF(c)$, then
\[
\PP(\text{$\pi^{-p}(\{A\})$ is frail} \mid \text{$A$ is robust})
\leq \exp\big({-}\tfrac{1}{16}C'e^{(1-\alpha)(h_{\ell'}-h_{i})}\big).
\]
\end{lemma}
\begin{proof}
Let $P_p(A)$ denote the probability in the statement. We claim that for $p\in\{0,...,\ell'-i\}$
\[
P_p(A)\leq \exp({-}\tfrac{1}{16}C'e^{(1-\alpha)(h_{\ell'}-h_{\ell'-p})}).
\]
To prove the claim, we proceed by induction on $p$. 
For~$p=0$ the result is obvious since then $\pi^{-p}(\{A\})=\{A\}$ so $P_p(A)=0$.
Assume now that $p>0$. Observe that $\pi^{-1}(\{A\})=\{A',A''\}$ where $A'$ and $A''$ are the two children of tile $A$. Note that $\pi^{-p}(\{A\})=\pi^{-(p-1)}(\{A'\})\cup\pi^{-(p-1)}(\{A''\})$.
Since by definition a non-robust tile cannot have robust descendant tiles, it follows that  
for $\pi^{-p}(\{A\})$ to be frail given that $A$ is robust it must happen that for every $T\in\{A',A''\}$ either $T$ is faulty or $\pi^{-(p-1)}(\{T\})$ is frail conditioned on $T$ being robust.
Formally, 
\[
P_p(A) \leq 
\prod_{T\in\pi^{-1}(\{A\})}\big(\PP(\text{$T$ is faulty})+P_{p-1}(T)\big).
\]
By the inductive hypothesis, Claim~\ref{flow:claim:sparse}, Item~\eqref{flow:lem:outerTiles:itm1} of Lemma~\ref{flow:lem:outerTiles} (taking $C'>32\ln 2$),  and since $h_{\ell'-p}\geq h_{\ell'-(p-1)}-(\ln 2+\epsilon)$, 
\[
P_p(A) 
\leq \big(3\exp\big({-}\tfrac{1}{16}C'e^{(1-\alpha)(h_{\ell'}-h_{\ell'-(p-1)})}\big)
\big)^2 \leq 9\big(\exp\big({-}\tfrac{1}{16}C'e^{(1-\alpha)(h_\ell'-h_{\ell'-p})}\big)\big)^{2^{\alpha}/e^{(1-\alpha)\epsilon}}.
\]
To conclude recall that
by hypothesis $\epsilon\leq\frac12\ln 2$, thus $2^\alpha/e^{(1-\alpha)\epsilon}\geq 2^{\frac12(3\alpha-1)}>1$,
so the induction claim follows provided $C'$ is large enough so
$\frac{1}{16}C'(2^{\frac12(3\alpha-1)}-1)>\ln 9$.
\end{proof}

Consider a vertex $w$ of $\calG:=\HRG_{\alpha,\nu}(n)$ which may or may not belong to the center component~$\calC$
of~$\calG$.
In what follows, we abuse terminology and say $w$ is robust if it belongs to a robust tile.
Let~$w^+$ (respectively, $w^-$) be any robust vertex in $V(\calG)\cap (B_O(h_{\ell'})\setminus B_{O}(h_{\ell'-1}))$ which is clockwise (respectively, anti-clockwise) from $w$ and closest in angular coordinate to $w$. It is not hard to see that both $w^-$ and $w^+$ exist and are distinct
w.h.p.
Observe that in our definition of $w^-$ and $w^+$ we do not require $w$ being a robust vertex. 
In fact, $w$ might not even belong to the center component of $\calG$.
In contrast, 
vertices $w^-$ and $w^+$ necessarily belong to $\calC$ (since by definition, robust vertices always belong to the center component). 
Henceforth, let $\Upsilon(w)$ be the smallest sector whose closure contains both $w^-$ and $w^+$.

\begin{lemma}\label{flow:lem:resistanceBnd}
Let $w$ be a vertex in the center component of $\calG:=\HRG_{\alpha,\nu}(n)$ and let $w'\in V(\calG)\cap B_O(\rho')$
be a robust vertex which is closest (in graph distance) to $w$.
Then,
\[
d(w,w') \leq 1+|V(G)\cap\Upsilon(w)|.
\]
\end{lemma}
\begin{proof}
Consider a shortest path $P_w$ between $w$ and $w'$.   
If the path~$P_w$ contains no vertex of $\HRG$ outside $\Upsilon(w)$, then 
  $d(w,w')\leq |V(\calG)\cap\Upsilon(w)|$ and we are done.
Assume then that $P_w$ contains vertices in $V(\calG)\setminus\Upsilon(w)$.
Let $u$ be the last (starting from $w$) vertex of~$P_w$ that belongs 
  to $\Upsilon(w)$. Also, let~$v$ be the first vertex of~$P_w$
  that is not in $\Upsilon(w)$.
  By definition of $\Upsilon(w)$, we can assume that the angular coordinate of $w^+$ is between those
  of $u$ and $v$ (the argument is similar if $w^-$ was the vertex
  `in between' $u$ and $v$, and thus we omit it).

  For the ensuing discussion it will be convenient, given two points $q,q'\in\HH^2$, to denote by $\theta(q,q')\in [0,\pi]$ the angle spanned by $q$ and $q'$ at the origin, i.e., $\theta(q,q')=\pi-|\pi-|\theta_q-\theta_{q'}||$.
  
  First, we consider the case where the radial coordinate of $v$ is at least the one of $w^+$, i.e., $r_v\geq r_{w^+}$.  Since $u$ and $v$ are neighbours in $P_w$, then $\theta(u,v)\leq\theta_R(r_u,r_v)$. On the other hand,
  by Remark~\ref{prelim:rem:monotonicity}, we have $\theta_R(r_u,r_v)\leq\theta_R(r_u,r_{w^+})$.
  Since $\theta(u,w^+)\leq\theta(u,v)$, it follows that $u$ and $w^+$ are neighbours in $\calG$, contradicting  the fact that $P_w$ is a shortest path between $u$ and a robust vertex in $B_O(\rho')$ (recall that by definition
  $w^+\in B_O(\rho')$ is a robust vertex).
  Assume then that $r_v<r_{w^+}$, but suppose that $r_u\geq r_{w^+}$. Arguing as before, we now
  deduce that~$v$ and $w^+$ are neighbours in $\calG$, so the only vertex that $P_w$
  has outside $\Upsilon(w)$ is $v$. Hence, 
  $d(w,w^+)\leq 1+|V(P)\cap\Upsilon(w)|$. Since $V(P)\subseteq V(\calG)$ and 
  recalling again that $w^+\in B_O(\rho')$
  is robust, by minimality we have $d(w,w')\leq d(w,w^+)$.
 
  We are left with the case where $r_u,r_v<r_{w^+}$. Let $i$ be the smallest integer such
  that $r_u,r_v< h_{i}$ (thus, $\max\{r_u,r_v\}\geq h_{i-1}$).
  By definition of $w^+$, the tile $T(w^+)$ to which it belongs is at level~$\ell'$, thus  $r_u,r_v<r_{w^+}<h_{\ell'}$ so $i\leq\ell'$.
  Since all tiles intersecting the
  line segment between the origin~$O$ and $w^+$ are ancestors of $T(w^+)$,
  they must be robust tiles (because $T(w^+)$ is a robust tile). Hence, all
  ancestor tiles of $T(w^+)$ are non-empty, in particular the one at level
  $i$, say $T$.  Let~$H$ be the parent half-tile of $T$
  (hence, $H$ belongs to a level $i-1$ tile). Since $T$ is
  a robust tile, its parent is not faulty and thus $V(\calG)\cap H$
  is non-empty, so there exists some vertex $\widehat{w}\in V(\calG)\cap H$.
  Since~$H$ is at level $i-1$, we have $r_{\widehat{w}}<h_{i-1}\leq \max\{r_u,r_v\}$.
  Note in particular that $\widehat{w}$ is robust (because it belongs to the robust tile $T$) and lies in $B_{O}(\rho')$ (because $r_{\widehat{w}}<h_{i-1}\leq h_{\ell'}\leq \rho'$).
  Observe that $\widehat{w}$ is in between $u$ and $v$, so
  $\theta(u,\widehat{w})+\theta(\widehat{w},v)=\theta(u,v)\leq\theta_R(r_u,r_v)$.
  To conclude, we consider two cases:
  \begin{enumerate}[(i)]
  \item\label{casei} Case $r_{\widehat{w}}<r_v$: By Remark~\ref{prelim:rem:monotonicity}
  we get $\theta(u,\widehat{w})\leq\theta_R(r_u,r_v)\leq\theta_R(r_u,r_{\widehat{w}})$ implying
  that $u\widehat{w}$ is an edge of $\HRG$ and contradicting the minimality
  of $P_w$.
  
  \item\label{caseii} Case $r_u>r_{\widehat{w}}\geq r_v$: Once more, by
  Remark~\ref{prelim:rem:monotonicity}, we get
  $\theta(\widehat{w},v)\leq\theta_R(r_u,r_v)\leq\theta_R(r_{\widehat{w}},r_v)$. It follows that $v\widehat{w}$ is an edge of $\HRG$, so by minimality, $P_w$ contains a single vertex outside $\Upsilon(w)$, and thus
  $d(w,w')\leq d(w,\widehat{w})\leq 1+|V(\calG)\cap\Upsilon(w)|$.
  \end{enumerate}
  The result follows.
\end{proof}

\begin{corollary}\label{flow:cor:resistBnd}
Let  $s, t \in V(\mathcal{C})$ be vertices of the center component of $\calG:=\HRG_{\alpha,\nu}(n)$.
If neither~$s$ nor  $t$ are located in $B_O(\rho)$, then w.h.p.,
\[
\Res{s}{t} \leq |V(\calG)\cap\Upsilon(s)|+|V(\calG)\cap\Upsilon(t)|+\calO(1).
\]
If $s$ but not $t$ (respectively, $t$ but not $s$) is located in $B_O(\rho)$, then the first (respectively, second) term in the right-hand side of the inequality above may be omitted, if both $s$ and $t$ are located in $B_O(\rho)$, then $\Res{s}{t} =\calO(1)$. 

Furthermore, the following two bounds hold w.h.p., 
\begin{align*}\sum_{u,v\in V(\calC)} \Res{u}{v} &= \mathcal{O}(n)\cdot \sum_{w \in V(\calC)\setminus B_O(\rho)}|V(\calG)\cap\Upsilon(w)| + \calO(n^2),\\
	t_{\odot}(\calC)&\leq    \sum_{w\in V(\calC)\setminus B_O(\rho)}|V(\calG)\cap\Upsilon(w)|\cdot d(w)  + \calO(n).
\end{align*}
\end{corollary}
\begin{proof}
Let $s'$ and $t'$ be robust vertices in $V(\calG)\cap B_O(\rho')$ which are closest (in graph distance) to~$s$ and $t$, respectively. 
Such vertices exist w.h.p.~because $s$ and $t$ belong to the center component of $\calG$, and since Lemma~\ref{flow:lem:nonfaultyTiles} guarantees that w.h.p.~$V(\calG)\cap B_O(\rho)$ is non-empty and all elements within are robust vertices.
Recalling that $\Res{s}{t}\leq \Res{s}{s'} + \Res{s'}{t'} + \Res{t'}{t}$ by Proposition~\ref{prop:resmetric} and observing that by
Proposition~\ref{flow:prop:energy} we know that
$ \Res{s'}{t'} =\calO(1)$, we get that 
$\Res{s}{t}\leq \Res{s}{s'}  + \Res{t'}{t}  +\calO(1)$.
The desired conclusion follows applying  Lemma~\ref{flow:lem:resistanceBnd} twice.
The last part of the statement holds because w.h.p., again by Lemma~\ref{flow:lem:nonfaultyTiles}, 
every vertex $w\in V(\calG)\cap B_O(\rho)$ is robust so $\Upsilon(w)=\{w\}$.

For the second claim let $S:= V(\calC)\setminus B_O(\rho)$ and thus, by the first claim,  w.h.p.\,\begin{equation*} \begin{aligned}\sum_{u,v\in V(\calC)} \Res{u}{v}&\leq \sum_{u,v\in V(\calC)}\big( |V(\calG)\cap\Upsilon(u)|\mathbf{1}_{u\in S }+|V(\calG)\cap\Upsilon(v)|\mathbf{1}_{v\in S }+ \calO(1) \big)\\ & = 2|V(\calC)|\cdot \sum_{w \in S}|V(\calG)\cap\Upsilon(w)| + \calO(|V(\calC)|^2).\end{aligned}\end{equation*} Thus, since $|V(\calC)|=\mathcal{O}(n)$ and  $|V(\calC)|^2=\mathcal{O}(n^2)$ w.h.p.~by Lemma~\ref{prelim:lem:devBnd}, the second claim holds.

For the third claim, recall $t_{\odot}(\calC)= \frac{1}{{4}|E(\calC)|}\sum_{u,v\in V(\calC)} \Res{u}{v}\cdot d(u)d(v)$ by Lemma~\ref{lem:targetresistance}. Thus, by the first claim in the statement, w.h.p.~we have  
	\begin{align*} 
	t_{\odot}(\calC)  & \leq  \frac{1}{{4}|E(\calC)|}\sum_{u,v\in V(\calC)}\left( |V(\calG)\cap\Upsilon(u)|\mathbf{1}_{u\in S }+|V(\calG)\cap\Upsilon(v)|\mathbf{1}_{v\in S }+\calO(1) \right)\cdot d(u)d(v)\notag \\
	&= \frac{1}{{2}|E(\calC)|}\sum_{v\in V(\mathcal{C})}d(v)\sum_{u\in S}|V(\calG)\cap\Upsilon(u)|\cdot d(u)  + \calO(|E(\calC)|)\notag  \\
		&=   \sum_{w\in S}|V(\calG)\cap\Upsilon(w)|\cdot d(w)  + \calO(|E(\calC)|). \end{align*}Thus the third claim in the statement holds as $|E(\calC)|=\mathcal{O}(n)$ w.h.p.\ by Lemma~\ref{prelim:lem:volumeCenterComp}.
 \end{proof}
Corollary~\ref{flow:cor:resistBnd} shows that we can bound the resistance between two poorly connected vertices in the peripheries of the center component of $\HRG_{\alpha,\nu}(n)$ by controlling the number of vertices in the smallest sectors between pairs of robust vertices that contain them. The next lemma helps us achieve this.

\begin{lemma}\label{lem:flow:sectortails} There exist constants $\kappa_1,\kappa_2 >0$ such that for $\calG:=\HRG_{\alpha,\nu}(n)$ and  any $w\in V(\calG)\setminus B_O(\rho)$ and $p\in \NN\setminus \{0\}$,
\[ \PP\left(|V(\calG)\cap\Upsilon(w)| \geq  \kappa_1\cdot 2^{p}\right)\leq \begin{cases}  
\exp\left(-\kappa_2\cdot 2^{p}\right) + \exp\left(-\kappa_2\cdot 2^{p(1-\alpha)}\right), & \text{ if $p< \ell'$,}
\\
\exp\left(-\kappa_2\cdot 2^{p}\right), & \text{ if $p \geq \ell'$.} \end{cases}.  
\] 
\end{lemma}
\begin{proof}
In what follows take $\epsilon$, $c:=c(\epsilon)$ and $C'$ to be as in the statement of Lemma~\ref{flow:lem:frailDescendentOfAncestor}. Let $\Upsilon^+(w)$ and $\Upsilon^-(w)$ be the elements of $\Upsilon(w)$ that are clockwise and 
	anti-clockwise from $w$, respectively. Since $\Upsilon(w)=\Upsilon^+(w)\cup\Upsilon^-(w)$, for all $w\in V(\calG)$, and the distribution of $|V(\calG)\cap\Upsilon^+(w)|$ and $|V(\calG)\cap\Upsilon^-(w)|$ are identical by symmetry, it suffices to prove the bound for $\PP\left(|V(\calG)\cap\Upsilon^+(w)| \geq (\kappa_1 /2)\cdot 2^{p}\right)$ for some constant $\kappa_1$.

    Recall that from~\eqref{flow:eqn:defNi} we have 
	$2^{\ell'}\theta_{\ell'}=2\pi/N_0$, where $N_0:=N_0(c)$ is a constant that is independent of $n$. Hence, as $2^{\ell'}N_0\theta_{\ell'}=2\pi>e^{\frac32}$, by Lemma~\ref{prelim:lem:devBnd} we have
	\begin{equation}\label{flow:eqn:bnd0}
	\PP(|V(\calG)\cap\Upsilon^+(w)|\geq 2^{p}N_0n\theta_{\ell'})
	\leq \PP(|V(\calG)|\geq 2^{p}N_0n\theta_{\ell'})
	\leq e^{-2^{p-1}N_0n\theta_{\ell'}},
	\qquad\text{if $p\geq \ell'$.}
	\end{equation}
	We claim that $\theta_{\ell'}=\Theta(1/n)$. Indeed, for large enough~$n$, by Claim~\ref{flow:claim-h},
	$h_{\ell'}+\ln 2+\epsilon>h_{\ell'+1}>\rho'\geq h_{\ell'}$,
	so by our choice of $\rho'$ in~\eqref{flow:eqn:rhoPrime}, we have that
	$\theta_{\ell'}=\theta_R(h_{\ell'},h_{\ell'})=\Theta(e^{-\frac12R})=\Theta(1/n)$. It follows that $N_0 n\theta_{\ell'} =\Theta(1)$ and so,
	by~\eqref{flow:eqn:bnd0}, the stated result holds for $p\geq \ell'$. 
     
	Henceforth, assume $p<\ell'$.
	Let $\Psi^+_w$ be the sector of central angle $2^{p}N_0\theta_{\ell'}$  whose elements are clockwise from $w$ and contains $w$ in its boundary, that is,
	\[
	\Psi^+_w:=\{ q\in\HH^2 \mid \theta_w\leq \theta_q<\theta_w+2^pN_0\theta_{\ell'}\}.
	\]
	Clearly, again by~\eqref{flow:eqn:defNi}, for $p <\ell'$, it holds that $2^pN_0\theta_{\ell'}< 2\pi$.
	Furthermore, note that if $\theta(w,w^+)\leq 2^pN_0\theta_{\ell'}$, 
	then $\Upsilon^+(w)$ is contained in $\Psi^+_w$, so 
	\begin{equation}\label{flow:eqn:bnd1}
	\PP(|V(\calG)\cap\Upsilon^+(w)|\geq 2^{p}N_0n\theta_{\ell'}) 
	\leq \PP(|V(\calG)\cap\Psi_w^+|\geq 2^{p}N_0n\theta_{\ell'}) 
	+\PP(\theta(w,w^+)>2^{p}N_0\theta_{\ell'}).
	\end{equation}
	By Lemma~\ref{prelim:lem:devBnd}, since $\EE|V(\calG)\cap\Psi^+_w|=\frac{2^pN_0\theta_{\ell'}}{2\pi}\cdot n$ and $2\pi>e^{\frac32}$, we have 
	\begin{equation}\label{flow:eqn:bnd2}
	\PP(|V(\calG)\cap\Psi^+_w|\geq 2^{p}N_0n\theta_{\ell'}) 
	\leq e^{-2^{p-1}N_0n\theta_{\ell'}},
	\qquad\text{for $p<\ell'$.}
	\end{equation}

    Let $T_w$ be the level $\ell'$ tile that intersects the
    ray starting at the origin which contains~$w$. Note that $w$ may not belong to $T_w$.
    Now, let $A^+_{p-1}(w)$ denote the clockwise sibling of the $(p-1)$-th ancestor of~$T_w$, i.e.,
	$A^+_{p-1}(w)$ is the clockwise sibling of $A_{p-1}(w):=\pi^{p-1}(T_w)$.
	Observe that every point which belongs to a tile in $\pi^{-(p-1)}(\{A^+_{p-1}(w)\})$ is at a clockwise angular distance at most $2^{p}\theta_{\ell'}$ from $w$. To see this, note that the factor $2^{p-1}$ is because $\pi^{-(p-1)}(\{A^+_{p-1}(w)\})$ contains that number of tiles, the factor $\theta_{\ell'}$ is because each such tile is contained within a sector of angle $\theta_{\ell'}$, and the 
	additional factor of $2$ accounts for the location of $w$ among the tiles in $\pi^{-(p-1)}(\{A_{p-1}(w)\})$.
	The relevance of the previous observation is that 
	we can now bound the probability that $\theta(w,w^+)$
	exceeds $2^{p}\theta_{\ell'}$ by the probability
	that $\pi^{-(p-1)}(\{A^+_{p-1}(w)\})$ is frail.
	Specifically, 
	\begin{equation}\label{flow:eqn:bnd3}
 \PP(\theta(w,w^+)>2^pN_0\theta_{\ell'})
	\leq \PP(\text{$\pi^{-(p-1)}(\{A^+_{p-1}(w)\})$ is frail}).
    \end{equation} 
     We consider two cases. First, we assume $p\leq\ell'-\ell$. Observe that $A^+_{p-1}(w)$ is an $(\ell' -(p-1)-\ell)$-th generation descendant of some tile at level $\ell$, contained in $B_{O}(\rho)$. By Lemma~\ref{flow:lem:nonfaultyTiles}, for any $d>0$ we can take $C>0$ sufficiently large so that, with probability $1-o(n^{-d})$, none of the tiles of~$\calF(c)$ contained in $B_{O}(\rho)$ is
     	faulty.  Thus, from~\eqref{flow:eqn:bnd3},  we get
    \begin{equation}\label{eq:thetabound}\begin{aligned}	\PP(\theta(w,w^+)\geq 2^{p}\theta_{\ell'}) 
   &\leq  \PP(\text{$\pi^{-(p-1)}(\{A^+_{p-1}(w)\})$ is frail} \mid \text{$A^+_{p-1}(w)$ is robust}) \\
   &\qquad + \PP(\text{$A^+_{p-1}(w)$ is frail} \mid \text{$\pi^{\ell' -(p-1)-\ell}(A^+_{p-1}(w))$ is robust})+o(n^{-d}) .\end{aligned}
    \end{equation}
    Note that $A_{p-1}^+(w)$ is a level $\ell'-p+1>\ell$ tile, so by Lemma~\ref{flow:lem:frailDescendentOfAncestor}  and Item~\eqref{flow:lem:outerTiles:itm2} of Lemma~\ref{flow:lem:outerTiles}
    we can bound the two probabilities in~\eqref{eq:thetabound}.
    It follows that 
    \begin{align} \label{flow:eqn:bnd4} 
    \PP(\theta(w,w^+)\geq 2^{p}\theta_{\ell'}) & \leq
    \exp\big({-}\tfrac{1}{16}C'e^{(1-\alpha)(h_{\ell'}-h_{\ell'-p+1})}\big) +
    \exp\big({-}\tfrac{1}{23}C'e^{(1-\alpha)(h_{\ell'}-h_{\ell'-p+1})}\big)+o(n^{-d}) \notag  \\
    &= \calO\big(\exp\big({-}\tfrac{1}{23}C'e^{(1-\alpha)(h_{\ell'}-h_{\ell'-p+1})}\big)\big)\notag \\ & =e^{-\Omega(2^{p(1-\alpha)})},
	\end{align}
	where the last equality holds since, for any $\eps>0$, if $n$ is large enough, by Claim~\ref{flow:claim-h} we have $h_{\ell'}-h_{\ell'-p+1}\geq (p-1)\ln 2-\epsilon$. Thus, since $N_0\geq 1$, we can bound the second term in~\eqref{flow:eqn:bnd1}. 
	Recalling that above we established that $\theta_{\ell'}=\Theta(1/n)$,
it follows that $N_0 n\theta_{\ell'} =\Theta(1)$ and so the stated result also holds for $p\leq \ell'-\ell$. 

  We still need to consider the case where $\ell'-\ell<p<\ell'$. 
    Now, $A_{p-1}^+(w)$ is a level $\ell'-p+1\leq\ell$ tile, so 
    among its descendants there is at least one level $\ell$ tile. In fact,
    among the descendants of $A_{p-1}^+(w)$ there are precisely 
    $m:=2^{\ell+p-\ell'-1}$ tiles at level $\ell$. Let $\calT$ be the set of such 
    tiles. Let $\mathcal{E}_T$ denote the event that none of the level~$\ell'$ descendants of $T$ is robust, i.e., $\calE_T:=\{\text{$\pi^{-(\ell'-\ell)}(\{T\})$ is frail}\}$.
    Since distinct $\calE_T$'s are determined by what happens for disjoint collections of tiles, they are independent. 
    Thus, from~\eqref{flow:eqn:bnd3} we deduce that
    \[
    \PP(\theta(w,w^+)\geq 2^{p}\theta_{\ell'}) 
    \leq \prod_{T\in\calT}\PP(\text{$\pi^{-(\ell'-\ell)}(T)$ is frail}).
    \]
    Observing that w.h.p., by Lemma~\ref{flow:lem:nonfaultyTiles}, every tile $T$ in $\calT$ is robust, from Lemma~\ref{flow:lem:frailDescendentOfAncestor} we get
    \[
    \PP(\theta(w,w^+)\geq 2^{p}\theta_{\ell'}) 
    \leq \exp\big({-}\tfrac{|\calT|}{16}C'e^{(1-\alpha)(h_{\ell'}-h_{\ell})}\big).
    \]
    Since $|\calT|=2^{\ell+p-\ell'-1}\geq 2^{(1-\alpha)(\ell+p-\ell'-1)}$, 
    arguing as was done above, we have $e^{(1-\alpha)(h_{\ell'}-h_{\ell})}=\Theta(2^{(1-\alpha)(\ell'-\ell)})$, so we again obtain that
    \begin{equation}\label{flow:eqn:bnd5}
    \PP(\theta(w,w^+)\geq 2^{p}\theta_{\ell'}) \leq 
    \exp\big(-\Omega(2^{-p(1-\alpha)})\big),
    \end{equation}
as claimed. 
\end{proof}
To conclude this section, we establish a bound on the raw moments of $|V(\HRG_{\alpha,\nu}(n))\cap\Upsilon(w)|$. The derived estimate will be useful for bounding, in the following section, the expectation of the target time. 

\begin{lemma}\label{lem:flow:averagingnew}  For any fixed  $ \kappa\geq 1$ and $\calG:=\HRG_{\alpha,\nu}(n)$  we have 
	\[
	\EE\big(\sum_{w\in V(\calG)\setminus B_O(\rho)}|V(\calG)\cap\Upsilon(w)|^\kappa \big) =\calO(n).
	\]
\end{lemma}
\begin{proof}An application of the Campbell-Mecke formula~\cite[Theorem 4.4]{Last2018}, and a recollection of the intensity function~$\lambda(\cdot,\cdot)$   from~\eqref{intro:eqn:intensity}, yields the following
	\[
	\EE\big(\sum_{w\in V(\calG)\setminus B_O(\rho)}|V(\calG)\cap\Upsilon^+(w)|^\kappa\big)
	= \int_{0}^{2\pi}\int_{\rho}^R \EE |V(\calG)\cap\Upsilon^+(w)|^\kappa\,
	\lambda(r_w,\theta_w)\,\mathrm{d}r_w\mathrm{d}\theta_w.
	\]
	It suffices to show that for a fixed $w\in  V(\calG)\setminus B_O(\rho)$, and any $\kappa \geq 1$, there exists a constant $\eta:= \eta_{\alpha, \nu,\kappa}>0$ such that $\EE |V(\calG)\cap\Upsilon^+(w)|^\kappa\leq \eta$. 
	
	Observe that for a non-negative random variable $X$ and any $x_0>0$ we have  
	\[\EE X^\kappa = \sum_{x\in \NN} \PP(X^\kappa>x) \leq x_0^\kappa + \sum_{p\in \NN} 2^{p\kappa}x_0^\kappa\cdot\PP(X^\kappa\geq 2^{p\kappa}x_0^\kappa) = x_0^\kappa\cdot (1+\sum_{p\in\NN} 2^{p\kappa}\cdot \PP(X\geq 2^px_0)),\] taking $X:=\theta(w,w^+)$ and $x_0:=\kappa_1$ to be the constant from Lemma~\ref{lem:flow:sectortails}, the bound above gives
	\begin{align*}
	\EE |V(\calG)\cap\Upsilon(w)|^\kappa
	&\leq (\kappa_1)^\kappa\cdot\Big(1+\sum_{p\in\NN} 2^{p\kappa}\cdot \,\PP(|V(\calG)\cap\Upsilon(w)|\geq \kappa_1 \cdot 2^{p})\Big)\\
	&\leq   (\kappa_1)^\kappa\cdot \Big(2+\sum_{p=1}^{\ell'-1}2^{p\kappa}\cdot e^{-\kappa_2\cdot 2^{p(1-\alpha)}}+\sum_{p\in\NN} 2^{p\kappa} \cdot e^{-\kappa_2\cdot 2^{p-1}}\Big) \leq \eta, 
	\end{align*}for some $\eta:= \eta_{\alpha, \nu,\kappa}$ by the bound in Lemma~\ref{lem:flow:sectortails}. 
\end{proof}

\subsection{Average Resistance \& Target Time}\label{sec:tergettime}
In this section we establish the convergence in expectation claims of two of our main contributions (Theorems \ref{flow:thm:tstat}   and \ref{thm:resistance}). First, we determine the expectation of the average (over pairs of vertices) effective resistance and then the expectation of the target time of the center component. 

Recall that the Kirchhoff index $\mathcal{K}(G)$ of a graph $G$ is the sum of resistances in the graph, that is \[\mathcal{K}(G)= \sum_{u,v\in V(G)}\Res{u}{v}.\] 
\begin{proposition}\label{expect:prop:thm4}
If $\calC:=\calC_{\alpha,\nu}(n)$, then
\[
\EE(\mathcal{K}(\calC)) = \Theta(n^2) 
\qquad \text{and} \qquad
\EE\Big(\frac{1}{|V(\calC)^2|}\sum_{u,v\in V(\calC)}\Res{u}{v}\Big) = \Theta(1).
\]
\end{proposition}
\begin{proof}
For the lower bound, observe that in any connected graph $G$ with at most $\kappa |V(G)|$ edges there are at least $|V(G)|/2$ vertices with degree at most $4\kappa $. Thus, there are $\Omega(|V(G)|^2)$ pairs $u,v$ with degree at most $4\kappa $ that do not share an edge. It follows that $\Res{u}{v} \geq 1/d(u) + 1/d(v) \geq 1/(2\kappa)$ for each such pair, and so $\sum_{u, v \in V(G)} \Res{u}{v} =  \Omega(|V(G)|^2/(2\kappa))$.  Recall that $|V(\calC)|=\Theta(n)$ and $|E(\calC)|=\Theta(n)$ a.a.s.~by Lemmas~\ref{prelim:thm:giantEqualCentralComp} and~\ref{prelim:lem:volumeCenterComp} respectively, giving the lower bound for the expectation of $\mathcal{K}(\mathcal{C})$ and the average resistance. 
	
For the upper bound, let $S:= V(\calC)\setminus B_O(\rho)$ and $\mathcal{E}_n$ be the event that \begin{equation}\label{eq:resbdd} \sum_{u,v\in V(\calC)} \Res{u}{v} = \mathcal{O}(n)\cdot \sum_{w \in S}|V(\calG)\cap\Upsilon(w)| + \calO(n^2), \end{equation} and thus  $\mathcal{E}_n$ holds w.h.p.~by Corollary~\ref{flow:cor:resistBnd}.  Then, by~\eqref{eq:resbdd} and Lemma~\ref{lem:flow:averagingnew}, we have
	\begin{equation*}\EE\Big(\sum_{u,v\in V(\calC)} \Res{u}{v}\mathbf{1}_{\mathcal{E}_n}\Big) =\mathcal{O}(n)\cdot  \EE\Big(\sum_{w \in S}|V(\calG)\cap\Upsilon(w)|\Big) + \calO(n^2)=\calO(n^2). \end{equation*}
	Since $\mathcal{E}_n$ holds w.h.p., we can assume that $\PP(\mathcal{E}_n^c)=o(n^{-6})$. Thus, by Cauchy–Schwarz
	\[
	\EE\Big(\sum_{u,v\in V(\calC)} \Res{u}{v}\mathbf{1}_{\mathcal{E}_n^c}\Big) \leq \EE\Big(|V(\calC)|^3\cdot \mathbf{1}_{\mathcal{E}_n^c}\Big) \leq \sqrt{\EE(|V(\calC)|^6)\cdot \PP(\mathcal{E}_n^c)} = o(1), 
	\]
	as $\Res{u}{v}<|V(\calC)|$, since the resistance  between two vertices is bounded by their graph distance, and $\EE(|V(\calC)|^6)=\calO(n^6)$ by Lemma~\ref{prelim:lem:poimoment}. Consequently, $\EE(\mathcal{K}(\calC)) = \calO(n^2)$. 
	
For the average resistance, let $\mathcal{V}_n := \{|V(\calC)|\geq c n \}\cap \mathcal{E}_n$, where $c>0$ is the constant implied by the lower bound in Lemma~\ref{prelim:lem:volumeCenterComp}. Then, by~\eqref{eq:resbdd} and Lemma~\ref{lem:flow:averagingnew}  with $\kappa=1$,
\begin{equation*}\EE\Big(\frac{1}{|V(\calC)|^2}\sum_{u,v\in V(\calC)} \Res{u}{v}\mathbf{1}_{\mathcal{V}_n}\Big) \leq \calO(1/n)\cdot \EE\Big(\sum_{w \in S}|V(\calG)\cap\Upsilon(w)|\Big) + \calO(1)=\calO(1). \end{equation*}
 As $\mathcal{E}_n$ holds w.h.p., we can assume  $\PP(\mathcal{V}_n^c) =o(n^{-2})$ by Lemma~\ref{prelim:lem:volumeCenterComp}.
 	Arguing as for the sum of resistances, we have 
  $\Res{u}{v}<|V(\calC)|$ and $\EE(|V(\calC)|^2)=\calO(n^2)$.
  Thus, 
  \begin{equation*} 
	\EE\Big(\frac{1}{|V(\calC)|^2}\sum_{u,v\in V(\calC)} \Res{u}{v}\mathbf{1}_{\mathcal{V}_n^c}\Big) \leq \EE\Big(|V(\calC)|\cdot \mathbf{1}_{\mathcal{V}_n^c}\Big) \leq \sqrt{\EE(|V(\calC)|^2)\cdot \PP(\mathcal{V}_n^c)} = o(1), 
	\end{equation*} 
	and so combining the two bounds above gives us the result for average resistance. 
 \end{proof}

Proposition \ref{expect:prop:thm4} gives the claim on the expectations in Theorems \ref{thm:resistance}. We will also use our bounds on the effective resistance to determine the target time, however first we must control sums of moments of degrees. 
\begin{lemma}\label{lem:edges} For any fixed  $1\leq \kappa < 2\alpha$ and $\HRG:=\HRG_{\alpha,\nu}(n)$, we have \[ \EE\Big(\sum_{v\in V(\mathcal{G})}d(v)^{\kappa}\Big) = \calO(n).\]
\end{lemma}
\begin{proof}
The number of neighbours $d(v)$ of a vertex $v$, where $v$ is at radial distance $r_v$ from the origin, is Poisson-distributed with mean $\mu_v$ which, by Lemma~\ref{prelim:lem:ballsMeasure}, is given by \begin{equation}\label{eq:mu}\mu_v:= n\cdot  \mu(B_v(R)\cap B_O(R)) = nc_\alpha e^{-\frac12 r_v}\big(1+\calO(e^{-(\alpha-\frac12)r_v}+e^{-r_v})\big), \end{equation} where $c_\alpha :=2\alpha/ \pi(\alpha-1/2)$. Recall $R=2\ln(n/\nu)$. For any $v\in V$ we have $r_v\leq R$ and thus $\mu_v\geq \nu c_\alpha (1-o(1))$ which, for any fixed $\nu>0$, is bounded from below by a fixed positive constant when $n$ is sufficiently large. By Lemma~\ref{prelim:lem:poimoment} and \eqref{eq:mu}, for any   $\kappa\geq 1 $ ,   we have 
\begin{equation}\label{eq:degmoment}\EE (d(v)^\kappa)\leq   \left(\frac{\kappa}{\ln (1 + \kappa) } \right)^{\kappa} \cdot \max\{\mu_v ,\; \mu_v^\kappa \}  \leq  c'n^\kappa\cdot e^{-\frac12 r_v \kappa}, \end{equation}
  for some constant $c':=c_{\alpha,\kappa}'$ as $\alpha >1/2$. Recall the function $f(r)$ in our definition of the intensity function $\lambda(r,\theta)$ from~\eqref{intro:eqn:intensity}, and observe that $f(r) = \calO(e^{-\alpha(R-r)})$ by Lemma~\ref{lem:approxden}. Thus, by the Campbell-Mecke formula~\cite[Theorem 4.4]{Last2018} and~\eqref{eq:degmoment} we deduce that  
\begin{align*}
\EE\big(\sum_{w\in V(\calG)}d(v)^\kappa\big)
&= \int_{0}^{2\pi}\int_0^R \EE (d(v)^\kappa)\,
\lambda(r_v,\theta_v)\,\mathrm{d}r_v\mathrm{d}\theta_v  = \calO(n^{\kappa+1}\cdot e^{-\alpha R}) \cdot \int_0^R e^{(\alpha -\frac12 \kappa)r_v } \,\mathrm{d}r_v. 
\end{align*}
Now observe that, since $R=2\ln(n/\nu)$ and $\alpha -\kappa/2 >0$, the desired result follows since 
\[\int_0^R e^{(\alpha -\frac12 \kappa)r_v } \,\mathrm{d}r_v =   \left[ \frac{e^{(\alpha -\frac12 \kappa)r_v }}{\alpha -\frac12 \kappa} \right]_0^R  = \calO(e^{(\alpha -\frac12\kappa) R}) = \calO(e^{\alpha R}n^{-\kappa}).\qedhere
\]
\end{proof}

We finally have all necessary ingredients to determine the expectation, over the choice of~$\HRG:=\HRG_{\alpha,\nu}(n)$, of the target time of the center component $\calC$, thus establishing a claim from Theorem~\ref{flow:thm:tstat}.
\begin{proposition}\label{flow:prop:thm1}
    If $\calC:=\calC_{\alpha,\nu}(n)$, then
    $\EE(t_{\odot}(\calC))=\Theta(n)$.
\end{proposition}
\begin{proof}
	For the lower bound we have $ \EE\left(t_{\odot}(\calC)\right)\geq \EE(|V(\mathcal{C})|)/2 =\Omega(n)$ Lemmas~\ref{lem:targetresistance} and~\ref{prelim:thm:giantEqualCentralComp}. 
 
 For the upper bound,  let $S:= V(\calC)\setminus B_O(\rho)$ and $\mathcal{E}_n$ be the event that 
	\begin{equation*}
	t_{\odot}(\calC)
	\leq    \sum_{w\in S}|V(\calG)\cap\Upsilon(w)|\cdot d(w)  + \calO(n). \end{equation*}
		Recall that $1/2<\alpha<1$ is fixed, so we can set $\kappa := 3/4 + \alpha /2$; which satisfies $1<\kappa < 2\alpha$, and has H\"older conjugate  $1<\kappa/(\kappa-1) < \infty$. Thus, applying H\"older's inequality gives
		\[t_{\odot}(\calC)  \leq   \Big(\sum_{w\in S}|V(\calG)\cap\Upsilon(w)|^{\frac{\kappa}{\kappa-1}}\Big)^{\frac{\kappa-1}{\kappa}} \cdot\Big(\sum_{w\in S} d(u)^\kappa\Big)^{\frac{1}{\kappa}} + \calO(n).\] 
  Applying the version of H\"older's inequality for expectations (with the same H\"older pair) yields 
	\[ 
	\EE\big(t_{\odot}(\calC) \cdot \mathbf{1}_{\mathcal{E}_n}\big) \leq \EE\Big(\sum_{w\in S}|V(\calG)\cap\Upsilon(w)|^{\frac{\kappa}{\kappa-1}} \Big)^{\frac{\kappa-1}{\kappa}}\cdot \EE\Big(\sum_{v\in S}d(v)^\kappa \Big)^{\frac{1}{\kappa}} + \calO(n) =\calO(n),
	\] by the bounds on moments from  Lemmas~\ref{lem:flow:averagingnew} and~\ref{lem:edges}.  We now need to control what happens outside the `desirable' event $\mathcal{E}_n$.	By Lemma~\ref{lem:commutetime}, since the center component $\mathcal{C}$ is connected,  \[t_{\odot}(\calC)\leq t_{\mathsf{hit}} \leq  2|E(\mathcal{C})|\cdot \max_{u,v \in V(\mathcal{C})}\Res{u}{v} \leq |V(\calG)|^3.\] 
 Recall that $\PP(\mathcal{E}_n^c)=o(1/n^{6})$ by Corollary~\ref{flow:cor:resistBnd}. Thus, by the Cauchy--Schwarz inequality, 
	\[\EE \big( t_{\odot}(\calC) \cdot \mathbf{1}_{\mathcal{E}_n^c}\big)\leq \EE \big( |V(\mathcal{\calC})|^3 \cdot \mathbf{1}_{\mathcal{E}_n^c}\big)\leq \sqrt{\EE(|V(\mathcal{\calC})|^6)\cdot\PP(\mathcal{E}_n^c)}= o(1),  
 \] as $\EE(|V(\calC)|^6)= \calO(n^6)$ by Lemma~\ref{prelim:lem:poimoment}, giving the result. 
\end{proof}

\subsection{Concentration of Average Resistance \& Hitting Time}\label{ssec:conc}
In this section we prove concentration for the Kirchhoff constant $\mathcal{K}(\mathcal{C})$, the average resistance, and the target time~$t_{\odot}(\calC)$ of the center component of the hyperbolic random graph (Theorems \ref{flow:thm:tstat}   and \ref{thm:resistance}). This will be achieved by proving concentration for sums of powers of degrees and numbers of points in sectors between robust tiles, since we can bound our quantities of interest in terms of these random variables by Corollary~\ref{flow:cor:resistBnd}. We control the former using the second moment method, for the latter we use the following extension of McDiarmid's inequality proved by  Warnke~\cite{Warnke16}. 

\begin{theorem}[{\cite[Theorem 2]{Warnke16}}]\label{thm:Warnke}
	Let $X=(X_j)_{j\in [N]}$ be a family of independent random 
	variables with $X_j$ taking values in a set $\Lambda_j$. 
	Let $\Gamma \subseteq \Lambda:=\prod_{j \in [N]}\Lambda_j$ be an event and 
	assume that the function $f:\Lambda \to \RR$ 
	satisfies the following \emph{typical Lipschitz condition}: 
	\begin{itemize}
		\item[(TL)] 
             There are numbers $(c_j)_{j \in [N]}$ and 
		$(d_j)_{j \in [N]}$ with $c_j \le d_j$ such that whenever 
		$x,\widetilde{x} \in \Lambda$ differ only in the 
		$k$-th coordinate we have 
		\begin{equation*}
			|f(x)-f(\widetilde{x})| \; \le \; \begin{cases}
				c_k, & \;\text{if $x \in \Gamma$,}\\ 
				d_k, & \;\text{otherwise.} 
			\end{cases}
		\end{equation*}
	\end{itemize} 
	For any numbers $(\gamma_j)_{j \in [N]}$ with $\gamma_j \in (0,1]$ 
	there is an event $\calB:=\calB(\Gamma,(\gamma_j)_{j \in [N]})$ 
	satisfying 
	\begin{equation*}\label{eq:WPrB}
		\PP(\calB) \leq \PP(X \notin \Gamma) \cdot \sum_{j \in [N]} \gamma_j^{-1}  
		\quad \text{and} \quad \Lambda\setminus\calB \subseteq \Gamma, 
	\end{equation*}
	such that for $e_j:=\gamma_j (d_j-c_j)$ and any 
	$t \ge 0$ we have 
	\begin{equation*}\label{eq:WPr}
		\PP(\{|f(X) - \EE f(X)| \geq t\} \cap \calB^c) \leq 2\cdot \exp\Big(-\frac{t^2}{2\sum_{j \in [N]}(c_j+e_j)^2}\Big) . 
	\end{equation*}
\end{theorem} 
 To apply Theorem~\ref{thm:Warnke} it will be convenient to work in the binomial model for the HRG, as the total number of points is then fixed. In this model exactly~$n$ vertices are placed independently in $B_O(R)$ according to the density $f(r, \theta)$ given by~\eqref{intro:eqn:intensity}, the same density as the Poisson model, and then two vertices are adjacent if they have distance at most $R$. Let $\widehat{\mathbb{P}}$ denote probability of an event in the binomial model of the HRG, and recall that $\PP$ denotes probability in the Poisson model.  Observe that if we condition on the number of points in the Poission model being $n$ the distribution is identical to that of the binomial process, that is $\widehat{\PP}(\cdot)= \PP(\cdot \mid |V(\mathcal{G})|=n)$. Thus, since  $\PP(|V(\mathcal{G})| =n) = \frac{1}{n!}(n/e)^{n} =\Theta(n^{-1/2})$,   for any event $\mathcal{E}$, we have
	\begin{equation}\label{eq:poitobin}\widehat{\mathbb{P}}(\mathcal{E})\leq \PP(\mathcal{E})/ \PP(|V(\mathcal{G})| =n) = \mathcal{O}(\sqrt{n})\cdot \PP(\mathcal{E}).\end{equation}
Conversely, for any event $\mathcal{E}_n$ where $n$ is the number of points in the binomial point process and the expected number of points in the corresponding Poisson point process, we have
  \begin{equation}\label{eq:bintopoi}\PP(\mathcal{E}_n)\leq n^{-\omega(1)} +  \sum_{k=n/2}^{3n/2} \PP(\mathcal{E}_k \; \big| \; |V(\mathcal{G})| =k  )\cdot \PP(|V(\mathcal{G})| =k  )  \leq n^{-\omega(1)} + \max_{n/2\leq  k\leq 3n/2}\widehat{\PP}(\mathcal{E}_k)  , \end{equation} since $\PP(||V(\mathcal{G})|-n|> n/2)= n^{-\omega(1)}$ by Lemma~\ref{prelim:lem:devBnd}.

\begin{lemma}\label{lem:sharedancestors}
Let $T$ and $T'$ be any two tiles in $B_O(R)\setminus B_{O}(h_{\ell}) $. If the minimum angle between any two rays from $O$ intersecting $T$ and $T'$ respectively is strictly greater than  $2\theta_\ell$, then $T$ and $T'$ do not share any common ancestor tiles in $B_O(R)\setminus B_{O}(h_{\ell}) $.  
\end{lemma}
\begin{proof}
Let $A$ and $A'$ be the unique ancestors of $T$ and $T'$ respectively that intersect $B_{O}(h_{\ell+1})\setminus B_{O}(h_{\ell})$. Observe that the internal angle spanned within either tile $A$ or $A'$ is $\theta_\ell$. Thus, if $\phi$ is the minimum angle between any two rays from $O$ intersecting $T$ and $T'$, then since these rays must pass through $A$ and $A'$ respectively, it follows that if $\phi - 2\theta_\ell >0 $, then $A$ and $A'$ do not intersect. Since all ancestor tiles of $T$ and $T'$ in $B_O(R)\setminus B_{O}(h_{\ell})$ are descendants of $A$ and $A'$ the result follows. 
\end{proof}

Using the variant of McDiarmid's inequality from Theorem \ref{thm:Warnke} and Lemma~\ref{lem:flow:sectortails},  we can derive the following result. 
\begin{lemma}\label{lem:sumsofsectors} For $\HRG:=\HRG_{\alpha,\nu}(n)$ and any fixed $ \kappa \geq 1 $,  one can take $C>0$ suitable large so that the following holds w.h.p.~for $\rho:=\rho(C)$ as in~\eqref{flow:eqn:rho}, 
	\[
	 \sum_{w\in V(\calG)\setminus B_O(\rho)}|V(\calG)\cap\Upsilon(w)|^\kappa   =\calO(n).
	\]
\end{lemma}
\begin{proof} 
 We apply Theorem~\ref{thm:Warnke} to the binomial model for the HRG, where exactly~$n$ vertices are placed independently in $B_O(R)$ with the same density as the Poisson model. Let $S:=V(\calG)\setminus B_O(\rho)$.  Observe that if $X_j\in B_O(R)$ is the position of vertex $j\in [n]$, and we let 
 $X=(X_j)_{j\in [n]}$, then 
 \[f(X):=\sum_{w \in S}|V(\calG)\cap\Upsilon(w)|^\kappa\geq 0,\]
 is a function of the $n$ independent random variables $X_1, \dots, X_n$. 
	
	We begin by showing $f$ satisfies condition (TL) for some event $\Gamma$ and sequences $(c_j)_{j\in [n]}$ and $(d_j)_{j\in [n]}$. Observe that since the graph has at most $n$ vertices, for any $X_1, \dots, X_n$  we have, 
 \[f(X)\leq \sum_{w \in V(\calG) }|V(\calG)|^\kappa \leq n^{1+\kappa},\] thus we can take $d_j := n^{1+\kappa}$ for each $j\in [n]$. Fix $s:= (\log n)^{\frac{2}{1-\alpha}}$ and define the event $\Gamma := \Gamma_1\cap \Gamma_2 \cap \Gamma_3\subseteq (B_O(R))^n$ as the intersection of the following three events:
	\begin{align*}\Gamma_1 &:= \big\{ |V(\calG)\cap \Upsilon(w) | \leq s \text{ for each }w\in S \big\},\\
	\Gamma_2 &:= \{\text{each half tile $H$ intersecting $B_{O}(\rho)$ contains at least $\tfrac{3}{4}\cdot \EE|V(\calG)\cap H|$ many points}\},\\
	\Gamma_3 &:= \{\text{any sector of central angle at most $4\theta_\ell$ contains at most $s$ points}\}.\end{align*}

	\begin{claim}\label{clm:Lipshitz}For any $x\in \Gamma$, and $\widetilde{x}\in (B_O(R))^n$ that differs from $x$ at a single coordinate,  we have $|f(x) - f(\widetilde{x})|\leq 2\cdot (3s)^{1+\kappa}$ . 
	\end{claim}
	\begin{poc}To bound  $|f(x) - f(\widetilde{x})|$ we consider the intermediate point configuration $x'\in  (B_O(R))^{n-1}$ which contains the $n-1$ points that are common to both $x$ and~$\widetilde{x}$. For a vertex $w \in V(\calG)$,  recall that $w^+$ (respectively, $w^-$) is the robust vertex in $V(\calG)\cap (B_O(h_{\ell'})\setminus B_{O}(h_{\ell'-1}))$ which is clockwise (respectively, anti-clockwise) from $w$ and closest in angular coordinate to $w$. Recall that $S=V(\calG)\setminus B_O(\rho)$.  Consider the collection of sets $\{V(\calG)\cap\Upsilon(w)\}_{w \in S}$. Define the set $U:=\{w^+,w^- \mid w\in S\}$ and give its elements an ordering $v_0, \dots, v_{\tau-1}$ by fixing an arbitrary ray from $O$ and ordering the points of $U$ by clockwise angle from the ray, breaking ties arbitrarily. Then, each $V(\calG)\cap\Upsilon(w)$ is equal to $S_i:=\{ w\in S \mid w^- = v_{i}, w^+ = v_{i+1} \} $ for some $i\in [\tau]$, where index addition is always modulo $\tau$. Thus, if $\Gamma_1$ holds, then $|S_i|\leq s $  for each $i\in [\tau]$.

	\textbf{Bound $1$ [$|f(x) - f(x')|$]:}	
 Note  that $|f(x) - f(x')|$ is the change to $f$ caused by removing a point from a configuration $x\in \Gamma$. There are two cases as either removing the point causes a tile containing some vertex $v_i$ to become faulty, or it does not.  
  
  For the first subcase, observe that as $x\in \Gamma_2$, removing any single point from $B_O(\rho)$ cannot cause the half-tile containing it to become sparse. Thus, if the removal of a point caused a tile~$T(v_i)$ containing some $v_i\in U$ to become frail then it must be by removing a point in one of the ancestor tiles of $T(v_i)$ which intersects $S$. 
  By Lemma~\ref{lem:sharedancestors} any two tiles in $S$ at a minimum angle greater than $2\theta_\ell$ do not share any ancestor tiles contained in $S$, and so the removal of a single point can only cause vertices of $U$ in a sector~$\Upsilon$ of central angle at most $4\theta_\ell$ to no longer belong  to robust tiles. Since any vertices of $U$ outside of sector~$\Upsilon$ are unaffected, this change only affects sets $S_{j_1}, \dots, S_{j_k}$ which overlap with the sector~$\Upsilon$. 
  Since the contribution to $f$ from $S_{j_1}, \dots, S_{j_k}$ is given by $\sum_{i=1}^k \sum_{w\in S_{j_i}} |S_{j_i}|^\kappa = \sum_{i=1}^k|S_{j_i}|^{1+\kappa}$ it follows by convexity that in the worst case they all merge. However, observe that as $x\in \Gamma_3$ the sector $\Upsilon$ contains at most $s$ points, and so we must have $|S_{j_1}\cup \cdots\cup S_{j_k}| \leq 3s$. 
  Thus,     
  \begin{equation}\label{eq:sub1}|f(x) - f(x')| \leq f(x') - f(x)\leq  |S_{j_1}\cup \cdots\cup S_{j_k}|^{1+\kappa} - \sum_{i=1}^k|S_{j_i}|^{1+\kappa} \leq (3s)^{1+\kappa}.   \end{equation}

  We now treat the other subcase, where the removal of the point does not make any point in~$U$ faulty. Observe that the number of vertices in the set $S_i$, for each $i \in [\tau]$, cannot increase by removing this point, and at most two  sets $S_{j}$ and $S_{j+1}$ can lose at most one vertex by removing this point. No other sets are affected and thus, in this case, we have 
		\begin{equation}\label{eq:sub2}|f(x) - f(x')| \leq |S_{j}|^{1+\kappa} + |S_{j+1}|^{1+\kappa}  - \left[ (|S_{j}|-1)^{1+\kappa} + (|S_{j+1}|-1)^{1+\kappa} \right]  \leq 2s^{1+\kappa} .   \end{equation}

	\textbf{Bound $2$ [$|f(x')-f(\widetilde{x})|$]:}		We now bound $|f(x')-f(\widetilde{x})|$, the change to $f(x')$ by adding a point. Firstly, suppose that the tile which receives the point is already robust, or it is not robust and adding the point does not make it robust. In this case $U$ does not change, and the point is contained in at most two sets $S_j,S_{j+1}$. These sets gain at most one point and no other sets $S_i$ are changed, thus we have   
	\begin{equation}\label{eq:add1}|f(x') - f(\widetilde{x})| \leq   (|S_{j}|+1)^{1+\kappa}+ (|S_{j+1}|+1)^{1+\kappa}- \left[ |S_{j}|^{1+\kappa}+ |S_{j+1}|^{1+\kappa} \right]  \leq 2(s+1)^{1+\kappa} .   \end{equation}

 Otherwise, the point added increases the size of $U$ by one and either the new point added has the same angle as an existing point in $U$, or it is in-between two existing members of $U$. In the former case, one point is just added to at most two sets $S_j, S_{j+1}$ and, similarly to~\eqref{eq:add1},  we have $|f(x') - f(\widetilde{x})|\leq2(s+1)^{1+\kappa} $. Otherwise, a single set $S_j$ is affected, by being split in two sets $S_{j}',S_{j}''$, thus
	    \begin{equation}\label{eq:add2}|f(x') - f(\widetilde{x})| \leq   |S_{j}|^{1+\kappa}- \left[ |S_{j}'|^{1+\kappa}+ |S_{j}''|^{1+\kappa} \right]  \leq s^{1+\kappa}.   \end{equation}
 
		\textbf{Conclusion:} Finally, by the triangle inequality and~\eqref{eq:sub1}, \eqref{eq:sub2}, \eqref{eq:add1} and~\eqref{eq:add2}, we have \[|f(x)-f(\widetilde{x})|\leq |f(x) - f(x')| + |f(x') - f(\widetilde{x})| \leq \max\{ (3s)^{1+\kappa},2s^{1+\kappa}\} + \max\{2(s+1)^{1+\kappa},s^{1+\kappa}\}, \] which is at most $2\cdot (3s)^{1+\kappa}$ for large $n$ since $s\rightarrow \infty$, as claimed. 	\end{poc}

	By Claim~\ref{clm:Lipshitz} we can take $c_j:=2\cdot (3s)^{1+\kappa}$, for each $j\in [n]$, to satisfy the typical Lipschitz condition (TL) of Theorem~\ref{thm:Warnke}.  Since we fixed $s= (\log n )^{\frac{2}{1-\alpha}}\geq (\log n )^2$ we have $\PP(\Gamma_1^c) = n\cdot e^{- \Omega(s) } = n^{-\omega(1)}$ by Lemma~\ref{lem:flow:sectortails}. Now, for any fixed $d\geq 0$ and $\kappa\geq 1$, by Lemma~\ref{flow:lem:nonfaultyTiles} we can choose $C>0$ sufficiently large such that we have $\PP(\Gamma_2^c) \leq  n^{- d-4-\kappa }$. To bound $\PP(\Gamma_3^c)$,  observe that by Claim~\ref{flow:claim-h} and~\eqref{flow:eqn:rho} the expected number of points in a sector of central angle $4\theta_\ell $ is 
 \[\frac{4\theta_\ell}{2\pi}\cdot n = \Theta( e^{\frac12(R-2h_\ell)}\cdot n  )  =\Theta( \left(C R/\nu\right)^\frac{1}{1-\alpha}) = \Theta(\sqrt{s}).\] 
 Thus, by Lemma~\ref{prelim:lem:devBnd}, the probability that a given sector of central angle $4\theta_\ell$ contains more than~$s/2$ points is at most $e^{-s/4}$. If we divide the disk evenly into consecutive sectors, of which less than~$n$ fit, then, by the union bound, the probability any of them contain more than $s/2$ points is at most $ne^{-s/4}$. Since any sector of central angle $\theta_\ell$ is contained in the union of at most two of these fixed consecutive sectors, it follows $\PP(\Gamma_3^c) \leq  ne^{-s/4} = n^{-\omega(1)}$. 
 
 Hence, by~\eqref{eq:poitobin}, for any $d\geq 0$ there exists a~$C$ such that 
 \[\widehat{\mathbb{P}}(X\notin \Gamma)\leq \mathcal{O}(\sqrt{n})\cdot\PP(X\notin \Gamma ) = \mathcal{O}(\sqrt{n})\cdot  n^{- d-4-\kappa } =o( n^{- d-3-\kappa }) .\]  Thus, if we let $\gamma_j := 1/n^{1+\kappa}$, then we have $e_j:= \gamma_j(d_j-c_j) \leq 1$ for each~$j\in [n]$. Consequently, by taking $t:=  n$ in Theorem~\ref{thm:Warnke} and setting $\widehat{\mu}:=\widehat{\EE}(f(X))$, there exists an event $\mathcal{B}$ satisfying \[\widehat{\PP}(\mathcal{B}) \leq \sum_{j\in [n]}\gamma_j^{-1}\widehat{\PP}(X\not\in\Gamma)\leq \sum_{j\in [n]}o(n^{-d-2})  = o(n^{-d-1}),\] such that 
	\begin{equation}\label{eq:preconc} 
\widehat{\PP}(\{|f(X) - \widehat{\mu}| \geq n\} \cap \calB^c)	 \leq2\cdot  \exp\Big(-\frac{  n^2}{2\sum_{k \in [n]}\big(2\cdot 3^{1+\kappa}(\log n)^{\frac{2(1+\kappa)}{1-\alpha}}+1\big)^2}\Big)  = n^{-\omega(1)}, \end{equation} 
since $\kappa$ is a fixed constant. It remains to bound $\widehat{\mu}$. Recall that $\mathbb{E}(f(X)) \leq \eta \cdot n$ for some constant~$\eta$ by Lemma~\ref{lem:flow:averagingnew}. Now by~\eqref{eq:preconc} and~\eqref{eq:bintopoi} we have $\PP(f(X) \geq \widehat{\mu}-n) = 1-o(1)$, thus 
\[\eta \cdot n\geq \mathbb{E}(f(X)) \geq \EE(f(X)\mid f(X) \geq \widehat{\mu}-n )\cdot \PP(f(X) \geq \widehat{\mu}-n) \geq (\widehat{\mu}-n)\cdot (1-o(1)),   \] and thus  
$\widehat{\mu} \leq \eta' \cdot n$ for some constant $\eta'>0$. It follows from~\eqref{eq:preconc} that $\widehat{\PP}(f(X) \ge (\eta'+1)n ) = o(n^{-d-1})$ and thus  ${\PP}(f(X) \ge 3(\eta'+1)n/2 ) = o(n^{-d})$ by~\eqref{eq:bintopoi}, as claimed.
\end{proof}

The following lemma will help us deal with the intersection of balls of radius $R$ around vertices when bounding the variance of sums of squares of degrees. 

\begin{lemma}\label{lem:deginteract}
If $p, q\in B_O(R)$ are such that $0\leq\theta_q-\theta_p\leq \pi$ and $r_p+r_q\geq R$, then 
\[ 
B_p(R)\cap B_q(R) \subseteq B_O(r^*)
\]
where $r^*=r^*(p,q):=R-\min\{r_p,r_q\}-2\ln(\frac{1}{2\pi}(\theta_q-\theta_p))$.
Moreover, $r^*\leq R$
if and only $\theta_q-\theta_p\geq 2\pi e^{-\frac12\min\{r_p,r_q\}}$.
\end{lemma}
\begin{proof}
Recall that $\theta_R(r,r')$ denotes the largest possible angle between two vertices at distances $r$ and $r'$ from the origin, which are at hyperbolic distance at most $R$ from each other.
Observe that, for $s\in B_p(R)\cap B_q(R)$, it must hold that 
$\theta_R(r_p,r_s)+\theta_R(r_q,r_s)\geq \theta_q-\theta_p$
or else $s$ could not be at distance less than $R$ from both $p$ and $q$ simultaneously.
Assuming, without loss of generality, that $r_p\leq r_q$,
by Remark~\ref{prelim:rem:monotonicity}, 
we get $2\theta_R(r_p,r_s)\geq \theta_q-\theta_p$.
Observe that if $r_p + r_s<R$, then $s\in B_O(r^*)$. Thus, we can assume $r_p+r_s\geq R$, and apply \cite[Lemma~3.4]{katzThesis23} to give $\theta_R(r_p,r_s)\leq\pi e^{\frac12(R-r_p-r_s)}$. Thence 
$R-r_p-r_s\geq 2\ln(\frac{1}{2\pi}(\theta_q-\theta_p))$.
The `moreover' part of the claim is obvious.
\end{proof}

We now establish a key result used when proving concentration of the target time. \begin{lemma}\label{lem:comsumdegs} 
For $\calC:=\calC_{\alpha,\nu}(n)$, $C>0$ and $\rho:=\rho(C)$ as in~\eqref{flow:eqn:rho},
\[
\operatorname{Var}\bigg(\sum_{v\in V(\calC)\setminus B_O(\rho)}d(v)^2\bigg) = \mathcal{O}\big(n^{1+\frac{1}{2\alpha}} \cdot (\ln n)^{\frac{{3}}{1-\alpha}} \big).
\]
\end{lemma}
\begin{proof}Let $X:=\sum_{v\in V(\calC)\setminus B_O(\rho)}d(v)^{2}$ and recall that  $\operatorname{Var}(X) = \EE(X^2) - \EE(X)^2$. 
Using  the Campbell-Mecke formula~\cite[Theorem 4.4]{Last2018}, we have
\begin{equation}\label{conc:eqn:firstMoment}
\EE(X) = \int_{0}^{2\pi}\int_\rho^R\EE (d(p)^2)\,
	\lambda(r_p,\theta_p)\,\mathrm{d}r_p\mathrm{d}\theta_p,
 \end{equation}
where $\lambda(r,\theta)$ is from~\eqref{intro:eqn:intensity}. Also, by the Campbell-Mecke formula (again) and then by symmetry,
\begin{align}
\EE(X^2) & = \int_{0}^{2\pi}\int_\rho^R \int_{0}^{2\pi}\int_\rho^R \EE (d(p)^2d(q)^2)\,
	\lambda(r_p,\theta_p)\lambda(r_q,\theta_q)\,\mathrm{d}r_q\mathrm{d}\theta_q\,\mathrm{d}r_p\mathrm{d}\theta_p \notag
 \\
 & = 4\int_{0}^{2\pi}\int_\rho^R \int_{\theta_p}^{\theta_p+\pi}\int_{r_p}^R \EE (d(p)^2d(q)^2)\,
	\lambda(r_p,\theta_p)\lambda(r_q,\theta_q)\,\mathrm{d}r_q\mathrm{d}\theta_q\,\mathrm{d}r_p\mathrm{d}\theta_p, \label{eq:fullint}
 \end{align}

By partitioning  the range of integration of $\theta_q$ (in the last quadruple integral in~\eqref{eq:fullint} above)
between $[\theta_p+\eta,\theta_p+\pi)$ and $[\theta_p,\theta_p+\eta)$, for $\eta:=1/n^{1-\frac{1}{2\alpha}}$ 
chosen with hindsight,
we obtain two integrals, say $I$ and~$J$ respectively, such that $\EE(X^2)=4(I+J)$.  Thus,
\begin{equation}\label{eq:varBnd}
\operatorname{Var}(X) = 4(I+J)-\EE(X)^2.
\end{equation}
In what follows we will bound~$I$ and $J$, starting with a crude but sufficient bound for the latter.

Since, for $a,b\geq 0$, it holds that $ab\leq 2ab\leq a^2+b^2$, we have $d(p)^2d(q)^2\leq d(p)^4+d(q)^4$. 
Taking expectation and applying Lemma~\ref{prelim:lem:poimoment}, we get
$\EE(d(p)^2d(q)^2) = \mathcal{O}(\EE(d(p))^4) + \mathcal{O}(\EE(d(q))^4)$.
Thus, as $\EE(d(p)) = \mathcal{O}(n\cdot e^{-\frac12 r_p})=\calO(e^{\frac12(R-r_p)})$ by Lemma~\ref{prelim:lem:ballsMeasure}, we have  
\begin{align*}
	J	&:=\int_{0}^{2\pi}\int_\rho^R \int_{\theta_p}^{\theta_p+\eta}\int_{r_p}^R \EE(d(p)^2d(q)^2)\,
	\lambda(r_p,\theta_p)\lambda(r_q,\theta_q)\,\mathrm{d}r_q\mathrm{d}\theta_q\,\mathrm{d}r_p\mathrm{d}\theta_p \\
 & \phantom{:}=\calO(n\cdot\eta)\cdot\int_{0}^{\pi}\int_\rho^R \EE(d(p))^4\,
	\lambda(r_p,\theta_p)\,\mathrm{d}r_p\,\mathrm{d}\theta_p
 \\ &\phantom{:} =\calO(n^2\cdot\eta)\cdot\int_{0}^{\pi}\int_\rho^R e^{(2-\alpha)(R-r_p)}\,\mathrm{d}r_p\,\mathrm{d}\theta_p.
 \end{align*}
 Hence, by our choice of $\eta=1/n^{1-\frac{1}{2\alpha}}$, since $2-\alpha>0$, and $\rho=R-\frac{\ln(C\frac{R}{\nu})}{1-\alpha}$, we have
 \begin{equation}\label{eq:J}
  J = \calO\big(n^2\cdot\eta\cdot e^{(2-\alpha)(R-\rho)}\big) 
  \\
  = \calO\big(n^2\cdot\eta\cdot (\ln n)^{\frac{2-\alpha}{1-\alpha}}\big)
  = \calO\big(n^{1+\frac{1}{2\alpha}}\cdot (\ln n)^{\frac{2-\alpha}{1-\alpha}}\big).
\end{equation}

In order to bound $I$, let  $r^*:=r^{*}(p,q)$ be as in Lemma~\ref{lem:deginteract}, and for $z\in \{p,q\}$ define  \[x(z) :=|V(\calG)\cap (B_z(R)\setminus B_O(r^*))| \qquad \text{and}\qquad y(z) :=|V(\calG)\cap B_z(R)\cap B_O(r^*)|.\]  Then, $d(p)^2d(q)^2  = (x(p) + {y(p)})^2 \cdot ({x(q)} + {y(q)})^2$.
Hence,
\begin{equation}\label{eq:expanddegrees}
d(p)^2d(q)^2 =\sum_{i=0}^{2}\sum_{j=0}^2\binom{2}{i}\binom{2}{j}{x(p)}^{i}{x(q)}^{j}{{y(p)}}^{2-i}{{y(q)}}^{2-j}.
\end{equation}
By Lemma~\ref{lem:deginteract}, as $r_p+r_q\geq 2\rho=2R-\Omega(\ln R){\geq R}$, we have that $B_p(R)\setminus B_O(r^*)$, $B_q(R)\setminus B_O(r^*)$, and $B_O(r^*)$ are pairwise disjoint, implying that ${x(p)},{x(q)}$ and ${y(p)}{y(q)}$ are {pairwise} independent. 
So, by taking the expectation of~\eqref{eq:expanddegrees}, the expression $\EE(d(p)^2d(q)^2)$ can be bounded by a sum of terms with the form $\EE({x(p)}^{i})\EE({x(q)}^{j})\EE({{y(p)}}^{2-i}{{y(q)}}^{2-j})$, for $i,j\in \{0,1,2\}$. Thus, if we define the integrals
\begin{equation*}
	I_{i,j}	:=\int_{0}^{2\pi}\int_\rho^R \int_{\theta_p+\eta}^{\theta_p+\pi}\int_{r_p}^R \EE({x(p)}^{i})\EE({x(q)}^{j})\EE({{y(p)}}^{2-i}{{y(q)}}^{2-j})\,
	\lambda(r_p,\theta_p)\lambda(r_q,\theta_q)\,\mathrm{d}r_p\mathrm{d}\theta_p\,\mathrm{d}r_q\mathrm{d}\theta_q, 
\end{equation*}
for $i,j\in \{0,1,2\}$, then since $\binom{2}{i}\leq 2$ we have 
\begin{equation}\label{eq:partsofI} I\leq 
	I_{2,2} + 4\cdot \sum_{(i,j)\neq (2,2)}  I_{i,j}.
\end{equation}  

Observe now that by~\eqref{conc:eqn:firstMoment}, we have
\begin{equation}\label{eq:220}
	I_{2,2}\leq\int_{0}^{2\pi}\int_\rho^R \int_{\theta_p}^{\theta_p+\pi}\!\!\int_{r_p}^R \EE(d(p)^2)\EE(d(q)^2)\,
	\lambda(r_p,\theta_p)\lambda(r_q,\theta_q)\,\mathrm{d}r_q\,\mathrm{d}\theta_q\,\mathrm{d}r_p\,\mathrm{d}\theta_p = \tfrac14\EE(X)^2,
\end{equation}
so this term will cancel out. We thus restrict our focus to the case $(i,j)\neq (2,2)$. Let $s\in B_O(R)$. Since $x(s)\leq d(s)$ 
 and as already observed $\EE(d(s))=\calO(e^{\frac12(R-r_s)})$, for any fixed $k\geq 0$,			
\[
\EE(x(s)^k)  =\mathcal{O}(\EE(d(s))^k) =\mathcal{O}\big(n^k\cdot\mu(B_s(R)\cap B_O(R))^k \big)
= \calO\big(e^{\frac{k}{2}(R-r_s)}\big).
\]
Note that ${y(p)},{y(q)}\leq Z$,
where $ Z:=|V(\calG)\cap B_{O}(r^*)|$ is Poisson distributed. Provided $\eta\leq\theta_q-\theta_p\leq\pi$, since $2\ln\frac{2\pi}{\eta}=R(1-\frac{1}{2\alpha})+\Theta(1)\leq \min\{r_p,r_q\}$,  by Lemma~\ref{lem:deginteract}, we get that $r^*:=r^*(p,q)\leq R$. Recall also from Lemma~\ref{lem:deginteract} that, since the order of integration fixes $r_p\leq r_q $, we have $r^*=R-r_p-2\ln(\frac{1}{2\pi}(\theta_q-\theta_p))$. Thus, by Lemma~\ref{prelim:lem:ballsMeasure}, we obtain
\[
\EE(Z)=n\cdot\mu(B_O(r^*))=\calO(ne^{-\alpha(R-r^*)})=\calO\Big(n\cdot \Big(\frac{e^{-\frac12 r_p}}{\theta_q-\theta_p}\Big)^{2\alpha}\Big). 
\]
Let $k:=4-i-j$.
By Lemma~\ref{prelim:lem:poimoment}, we have $\EE({{y(p)}}^{2-i}{{y(q)}}^{2-j})\leq \EE(Z^{k}) = \calO(\EE(Z)+\EE(Z)^{k})$. Thus, 
\begin{align}\label{eq:momentsofdegparts}
& 
\EE({x(p)}^{i})\EE({x(q)}^{j})\EE({{y(p)}}^{2-i}{{y(q)}}^{2-j}) 
\notag \\
& \qquad = \calO\big(\EE(d(p))^{i}\EE(d(q))^{j}\big[\EE(Z) + \EE(Z)^{k}\big]\big)\notag \\ 
& \qquad = \calO\Big(e^{\frac{i}{2}(R-r_p)}e^{\frac{j}{2}(R-r_q)}\Big[ n\Big(\frac{e^{-\frac12 r_p}}{\theta_q-\theta_p}\Big)^{2\alpha}+ n^{k}\Big(\frac{e^{-\frac12 r_p}}{\theta_q-\theta_p}\Big)^{2\alpha k}\Big]\Big) .
\end{align}
For $h\in \{1,k\}$, we define 
	\begin{equation}\label{eq:Ihij}I_{i,j}^h : =n^{2+h}\int_{0}^{2\pi}\int_{\rho}^R\int_{\theta_p+\eta}^{\theta_p+\pi}\int_{r_p}^R 
	e^{-(\alpha-\frac{i}{2})(R-r_p)}
	e^{-(\alpha-\frac{j}{2})(R-r_q)}
	\Big(\frac{e^{-\frac12 r_p}}{\theta_q-\theta_p}\Big)^{2\alpha h}\,\mathrm{d}r_q\,\mathrm{d}\theta_q\,\mathrm{d}r_p\,\mathrm{d}\theta_p
 \end{equation}  
 and, from~\eqref{eq:momentsofdegparts}, we see that $I_{i,j}=\mathcal{O}\big( I_{i,j}^1\big)+\mathcal{O}\big( I_{i,j}^{k}\big)$. 

Next, observe that for any $r_p \in [\rho,R]$ and $\theta_p\in[0,2\pi)$, and $j\in \{0,1,2\}$ we have 
\begin{align*}\int_{\theta_p+\eta}^{\theta_p+\pi}\int_{r_p}^R 
\frac{e^{-(\alpha-\frac{j}{2})(R-r_q)}}{
\big( \theta_q-\theta_p\big)^{2\alpha h}}\,\mathrm{d}r_q\,\mathrm{d}\theta_q &= \frac{1-e^{-(\alpha-\frac{j}{2})(R-r_p)}}{\alpha-\frac{j}{2}}\cdot \frac{\eta^{1-2\alpha h}-\pi^{1-2\alpha h}}{2\alpha h -1}\\ &= \mathcal{O}(e^{(1-\alpha)(R-\rho)}\cdot \eta^{1-2\alpha h}).
\end{align*}
This bound above on the inner integral in \eqref{eq:Ihij} is independent of $r_p$, so it remains to bound
\begin{equation*}
	\int_{0}^{2\pi}\int_{\rho}^R e^{-(\alpha-\frac{i}{2})(R-r_p)}
	\cdot 
	 e^{-\alpha h  r_p} \,\mathrm{d}r_p\,\mathrm{d}\theta_p = \mathcal{O}\big(e^{-(\alpha-\frac{i}{2})R  -(\alpha(h-1)+\frac{i}{2})\rho   }   \big)
	 =\mathcal{O}\big(e^{ (\alpha(h-1) +\frac{i}{2})(R-\rho)-\alpha h R }   \big).
\end{equation*}
Thus, as $e^{R-\rho} =(\ln n)^{\frac{1}{1-\alpha}}$ and $1-\alpha + \alpha(h-1) +\frac{i}{2} = 1 + 2\alpha -\alpha j +  (\frac{1}{2} -\alpha)i \leq 1 +2\alpha $, we have
\[I_{i,j}^h =  \mathcal{O}\big(n^{2+h}\cdot  e^{(1-\alpha)(R-\rho)}  \eta^{1-2\alpha h}\cdot e^{(\alpha(h-1) +\frac{i}{2})(R-\rho) -\alpha h R }   \big) = \mathcal{O}\big(n^{2 + (1-2\alpha)h} \cdot \eta^{1-2\alpha h}\cdot (\ln n)^{\frac{1+2\alpha}{1-\alpha }} \big).   \]
Our choice of $\eta=1/n^{1-\frac{1}{2\alpha}}$ yields $n^{2+(1-2\alpha)h}\eta^{1-2\alpha h}=n^{1+\frac{1}{2\alpha}}$ for any $h$. Thus, by~\eqref{eq:momentsofdegparts}, for all $(i,j)\neq (2,2)$, we have
\begin{equation}\label{eq:maxboundI}
I_{i,j}=\mathcal{O}\big(I_{i,j}^1\big)+ \mathcal{O}\big(I_{i,j}^{k}\big) =  \calO\big(n^{1+\frac{1}{2\alpha}}\cdot(\ln n)^{\frac{1+2\alpha}{1-\alpha}}\big). 
\end{equation} 
Combining the bounds~\eqref{eq:varBnd} through~\eqref{eq:maxboundI}, we obtain
\[\operatorname{Var}(X) = \EE(X^2) - \EE(X)^2  =\mathcal{O}(n^{1+\frac{1}{2\alpha}}\cdot (\ln n)^{\frac{1+2\alpha}{1-\alpha}})+ \mathcal{O}(n^{1+\frac{1}{2\alpha}}\cdot (\ln n)^{\frac{2-\alpha}{1-\alpha}})=\mathcal{O}(n^{1+\frac{1}{2\alpha}}\cdot (\ln n)^{\frac{3}{1-\alpha}}),  \] since $\alpha<1$. 
\end{proof}

The result above gave us the variance of sums of squared degrees, to prove concentration for the target time we will instead need to control a sum of degrees raised to some non-integer power less than two. The following lemma lets us transfer the variance bound from Lemma~\ref{lem:comsumdegs}. 
\begin{lemma}\label{lem:vardec}Let $A>0$ and $X$ be any non-negative integer valued random variable satisfying $\EE(X^{2A}\ln X)<\infty $. Then, $\operatorname{Var}(X^a)$ is non-decreasing in $0\leq a \leq A$. 
\end{lemma}
\begin{proof}
	Observe that for $a\leq A$, each element of the sequence $(f_n(a))$, where $f_n(a):=\sum_{x=1}^n x^{2a}\cdot \PP(X=x)$, is differentiable and $f_n$ converges to $\EE(X^{2a})<\infty $. Furthermore, the sequence $(\frac{\mathrm{d}}{\mathrm{d}a}f_n(a))_n$ converges uniformly to $\sum_{x \geq 1} (2a\ln x ) x^{2a}\cdot\PP(X=x) = 2a\cdot  \EE(X^{2a}\ln X) <\infty$. Thus, it follows from a standard convergence result~\cite[Theorem 7.17]{Rudin} that  	\begin{align*}\frac{\mathrm{d}}{\mathrm{d}a} \operatorname{Var}(X^a) &= \frac{\mathrm{d}}{\mathrm{d}a}  \sum_{x\geq 1}x^{2a}\PP(X=x) -  \frac{\mathrm{d}}{\mathrm{d}a}\Big( \sum_{x\geq 1}x^{a}\PP(X=x)\Big)^2\\
		&=  \sum_{x\geq 1}(2a\ln x ) x^{2a}\cdot\PP(X=x) -  2\Big(\sum_{x\geq 1} (a\ln x) x^a \cdot\PP(X=x)\Big)\cdot\Big( \sum_{x\geq 1}x^{a}\PP(X=x)\Big)\\
		&= 2a \cdot\big(\EE(X^{2a}\ln X)  - \EE(X^{a}\ln X)\EE(X^{a})\big) \\&\geq 0,   
\end{align*}  where the last inequality follows from the fact that $\operatorname{Cov}(f(X),g(X))\geq 0$ for increasing functions~$f,g$, see for example \cite{Schmidt}. Thus, since $\operatorname{Var}(X^0)= 0$, $\operatorname{Var}(X^a)$ is non-decreasing in $0\leq a\leq A$. \end{proof}

We can now prove the main result of this section,
that is the a.a.s.~convergence claims in Theorem~\ref{flow:thm:tstat} and~\ref{thm:resistance}, which together with Propositions~\ref{expect:prop:thm4} and~\ref{flow:prop:thm1} establish the said theorems.
\begin{proposition}\label{conc:prop:thm1}
If $\calC:=\calC_{\alpha,\nu}(n)$, then a.a.s.~$t_{\odot}(\calC) = \Theta(n)$, and w.h.p.\  
	\[\mathcal{K}(\calC) = \Theta(n^2), \qquad \text{and} \qquad   \frac{1}{|V(\calC)|^2}\sum_{u,v\in V(\calC)} \Res{u}{v}= \Theta(1).  \]
\end{proposition}
\begin{proof}We begin with $\mathcal{K}(\calC)$. The lower bound follows since the giant component has $\Omega(n)$ vertices of degree~$\mathcal{O}(1)$ by Lemmas~\ref{prelim:thm:giantEqualCentralComp} and~\ref{prelim:lem:volumeCenterComp}. If $S:= V(\calC)\setminus B_O(\rho)$ then by Corollary~\ref{flow:cor:resistBnd} the following holds w.h.p., 
	\[\sum_{u,v\in V(\calC)} \Res{u}{v}\leq   \mathcal{O}(n)\cdot \sum_{w \in S}|V(\calG)\cap\Upsilon(w)| + \mathcal{O}(n^2),\]
	so the result follows from Lemma~\ref{lem:sumsofsectors} with $\kappa=1$.

For the average resistance the upper and lower bound follow immediately from the bounds on $\mathcal{K}(\mathcal{C})$ and the fact that the center component has $\Omega(n)$ many vertices w.h.p.\ by Lemma~\ref{prelim:thm:giantEqualCentralComp}.

For the target time $t_{\odot}(\calC)$, the lower bound follows directly from Lemmas~\ref{lem:targetresistance} and~\ref{prelim:thm:giantEqualCentralComp} as a.a.s~$t_{\odot}(\calC)\geq |V(\mathcal{C})|/2 =\Omega(n)$. For the upper bound, let $S:= V(\calC)\setminus B_O(\rho)$ and recall that by Corollary~\ref{flow:cor:resistBnd}, w.h.p.,
	\begin{equation*}
	t_{\odot}(\calC)
	\leq    \sum_{w\in S}|V(\calG)\cap\Upsilon(w)|\cdot d(w)  + \calO(n), \end{equation*}
	 As $1/2<\alpha<1$ is fixed, we set $\kappa = 3/4 + \alpha /2$; which satisfies $1<\kappa < 2\alpha<2$, and has H\"older conjugate  $1<\kappa/(\kappa-1) < \infty$. Thus, by H\"older's inequality, w.h.p.\ 		\begin{equation}\label{eq:holder}t_{\odot}(\calC)  \leq \Big(\sum_{w\in S}|V(\calG)\cap\Upsilon(w)|^{\frac{\kappa}{\kappa-1}}\Big)^{\frac{\kappa-1}{\kappa}} \cdot\Big(\sum_{w\in S} d(u)^\kappa\Big)^{\frac{1}{\kappa}} + \calO(n).\end{equation}Observe that the first term above is $\mathcal{O}(n^{\frac{\kappa-1}{\kappa}})$ w.h.p.~by Lemma~\ref{lem:sumsofsectors}.  To control the second, observe that by Lemmas~\ref{lem:vardec} and~\ref{lem:comsumdegs}, since $\kappa <2$, we have \[\operatorname{Var}\Big( \sum_{v\in S}d(v)^\kappa \Big) \leq \operatorname{Var}\Big( \sum_{v\in S}d(v)^2\Big) = \mathcal{O}\Big(n^{1+\frac{1}{2\alpha}} \cdot (\ln n)^{\frac{{3}}{1-\alpha}} \Big)=o(n^2).\]Additionally $\EE(\sum_{v\in V(\mathcal{G})}d(v)^{\kappa}) = \calO(n)$ by Lemma~\ref{lem:edges}. Thus, $\sum_{v\in V(\mathcal{G})}d(v)^{\kappa} =\mathcal{O}(n)$ a.a.s.\ by Chebychev's inequality, and the second non-trivial term in~\eqref{eq:holder} is $\mathcal{O}(n^{\frac{1}{\kappa}})$ giving the result.   
 \end{proof}

\section{Cover Time and Maximum Hitting Time}\label{sec:cover}
In this section, we first determine the asymptotic behavior
of the cover time of the giant component of the HRG.
To do so, we rely on several intermediate structural results concerning HRGs that were established 
in the previous section. However, the arguments developed here are not flow based. 
In particular, the ones concerning commute times rely on
extensions of Matthew's Bound, a classical and useful method for bounding cover times.
This result is then complemented, by the determination of the maximum hitting time also for the giant.

As in the previous sections, we restrict our discussion exclusively to the parameter range where $\frac12<\alpha<1$.

\subsection{Upper Bounds on Hitting and Cover times}
As mentioned above, to determine the cover time of the center component of the HRG we rely on Matthew's bound (stated below) and one of its extensions (mentioned later). The bound and its extensions relate the cover time of a graph with its  minimum and maximum hitting times.
\begin{theorem}{\cite[Theorem 11.2]{peresmix}}\label{prelim:thm:matthew}
  For any finite irreducible Markov
chain on $n>1$ states we have \[\tcov \leq \thit \cdot \left(1 +\frac{1}{2}+ \cdots + \frac{1}{n} \right).  \] 
\end{theorem}
A well known relation between hitting times and diameter of graphs together with established results 
concerning the diameter of the HRG allows us to easily derive upper bounds on  
hitting and, via Matthew's bound, also on the cover time of simple random walks on the center component of $\calG_{\alpha,\nu}(n)$.
\begin{proposition}\label{cover:prop:maxHit} If $\calC:=\calC_{\alpha,\nu}(n)$, then
 a.a.s.~and in expectation we have \[\thit (\calC)=\calO(n\log n) \qquad \text{and } \qquad \tcov(\calC)=\calO(n\log^2 n).\]
\end{proposition}
\begin{proof}
  Recall that $d_G(u,v)$ denotes the length of the shortest path 
  between a pair of vertices $u,v\in V(G)$, of a graph $G$. Then, by Lemmas~\ref{lem:commutetime} and~\ref{lem:basicresBdd}, we have $\Exu{u}{\tau_v}\leq 2|E(G)|d_G(u,v)$. Let $\mathcal{D}_1 := \{\max_{u,v\in V(\calC)}d_{\mathcal{C}}(u,v)\leq \kappa \ln n \}$ for a suitably large constant $\kappa$, then we have $\PP(\mathcal{D}_1^c) =n^{-\kappa'}$ for some fixed $\kappa'>0$ by~\cite{MS19},
  as $1/2<\alpha <1$ is fixed. We note that M\"uller and Staps~\cite{MS19} only state that their $\calO(\log n)$ diameter bound holds a.a.s., however if one checks the proofs in their paper they hold with probability $1-n^{-1000}$, and they additionally use a coupling~\cite[Lemmas 27 \& 30]{FM18} that holds with probability $1-n^{-\kappa'}$, where $\kappa'>0$ depends on $\alpha$. Recall also that $|E(\mathcal{C})|=\calO(n)$ a.a.s. by Lemma~\ref{prelim:lem:volumeCenterComp}. Thus, we get that $\thit =\calO(n\log n)$ holds a.a.s..  We also define the event $\mathcal{L} := \{\sum_{k=1}^{|V(\mathcal{C})|} k^{-1} \leq 2\ln n \}$ and observe that $\PP(\mathcal{L}^c) \leq \PP(|V(\calC )|> 100 n)\leq e^{-\Omega(n)}  $ by Lemma~\ref{prelim:lem:devBnd}. The desired a.a.s.~upper bound for the cover time then follows from Matthew's bound (see Theorem~\ref{prelim:thm:matthew}). 
   
   We now show that these bounds also hold in expectation. Before we begin, we must define an additional event $\mathcal{D}_2 := \{ \max_{u,v\in V(\calC)}d_{\mathcal{C}}(u,v) \leq \kappa (\ln n)^{\frac{1}{1-2\alpha}} \} $ and note that $\PP(\mathcal{D}_2^c) =\calO(n^{-100})$ by~\cite[Theorem 1]{FriedrichK18} (also see the remark on~\cite[Page 1316]{FriedrichK18} after Theorem 3). 
   
   We begin with $\EE(t_{\mathsf{hit}})$, conditioning on  the event $\mathcal{D}_1 = \{\max_{u,v\in V(\calC)}d_{\mathcal{C}}(u,v)\leq \kappa \ln n \}$ gives  
  \begin{equation*}
  \EE\left(t_{\mathsf{hit}}\cdot \mathbf{1}_{\mathcal{D}_1}\right)\leq (\kappa\ln n)\cdot 2\EE\left(|E(\mathcal{C})|\right) = \calO(n\log n), \end{equation*} as $\EE\left(|E(G)|\right) =\calO(n)$ by Lemma~\ref{lem:edges}. By the Cauchy–Schwarz inequality we have \[\EE\left(t_{\mathsf{hit}}\cdot\mathbf{1}_{\mathcal{D}_1^c\cap \mathcal{D}_2}\right) \leq \kappa(\ln n)^{\frac{1}{1-2\alpha}}\cdot 2\EE\left(|E(\mathcal{C})|\mathbf{1}_{\mathcal{D}_1^c} \right) \leq  2\kappa(\ln n)^{\frac{1}{1-2\alpha}}\cdot\sqrt{\EE\left(|E(\calC)|^2\right)\PP(\mathcal{D}_1^c)}, \]
  then as   $\EE(|E(\mathcal{G})|^2) = \mathcal{O}(n^2)$ by~\cite[Claim 5.2]{chellig2021modularity}
  and by the bound on $\PP(\mathcal{D}_1^c)$ above we have 
 \begin{equation*}\label{eq:exhit2}\EE\left(t_{\mathsf{hit}}\cdot\mathbf{1}_{\mathcal{D}_1^c\cap \mathcal{D}_2}\right) \leq 2\kappa(\ln n)^{\frac{1}{1-2\alpha}}\cdot\sqrt{\calO(n^2) \cdot n^{-\kappa'}} = o(n).  \end{equation*}
  Finally, since $d_{\mathcal{C}}(u,v)\cdot |E(\calC)|\leq |V(\calC)|^3$, and  $\EE|V(\calC)|^k = \calO(n^k)$ for any $k\geq 1$ by Lemma~\ref{prelim:lem:poimoment}, \begin{equation*}\label{eq:exhit3}\EE\left(t_{\mathsf{hit}}\cdot\mathbf{1}_{\mathcal{D}_1^c\cap \mathcal{D}_2^c }\right) \leq\EE\left(|V(\calC)|^3\cdot \mathbf{1}_{\mathcal{D}_2^c}\right)\leq  \sqrt{\EE\left(|V(\calC)|^6\right)\PP(\mathcal{D}_2^c)} =o(1), \end{equation*} 
  The result follows from the last 3 displayed equations.

    We now bound $\EE(t_{\mathsf{cov}})$. Recall the event $\mathcal{L} = \{\sum_{k=1}^{|V(\mathcal{C})|} k^{-1} \leq 2\ln n \}$ and observe that  
 \begin{equation*}
 \EE\left(t_{\mathsf{cov}}\cdot \mathbf{1}_{\mathcal{D}_1\cap \mathcal{L}}\right)\leq (\kappa\ln n)\cdot 2\EE\left(|E(\mathcal{C})|\right)\cdot (2\ln n) = \calO(n\log^2 n). 
 \end{equation*}
 Now, since $ d_{\calC}(u,v) <|V(\calC)|$, $\sum_{k=1}^{|V(\mathcal{C})|} k^{-1}<|V(\calC)| $, and $|E(\calC)|\leq |V(\calC)|^2$, we have 
 \begin{equation*}
 \EE\left(t_{\mathsf{cov}}\cdot \mathbf{1}_{\mathcal{L}^c}\right)\leq \EE\left(|V(\mathcal{C})|^4\mathbf{1}_{\mathcal{L}^c}\right) \leq \sqrt{\EE(|V(\mathcal{C})|^8)\PP(\mathcal{L}^c)} = o(1), \end{equation*}
  as $\PP(\mathcal{L}^c) = e^{-\Omega(n)} $ and $\EE|V(\calC)|^k = \calO(n^k)$ for any fixed $k\geq 1$ by Lemma~\ref{prelim:lem:poimoment}. Similarly, 
  \begin{equation*}
  \EE\left(t_{\mathsf{cov}}\cdot\mathbf{1}_{\mathcal{D}_1^c\cap\mathcal{L}\cap  \mathcal{D}_2}\right) \leq \kappa(\ln n)^{\frac{1}{1-2\alpha}} \cdot 2\sqrt{\EE\left(|E(\calC)|^2\right)\PP(\mathcal{D}_1^c)}\cdot  (2\ln n)  =o(n).   \end{equation*} 
  Finally, again similarly to before, we have 
  \begin{equation*}
  \EE\left(t_{\mathsf{cov}}\cdot\mathbf{1}_{\mathcal{D}_1^c\cap \mathcal{L}\cap \mathcal{D}_2^c }\right) \leq\EE\left(|V(\calC)|^4\cdot \mathbf{1}_{\mathcal{D}_2^c}\right)\leq  \sqrt{\EE\left(|V(\calC)|^8\right)\PP(\mathcal{D}_2^c)} =o(1). \end{equation*} The bound on $\EE(t_{\mathsf{cov}})$ follows from combining the last 4 displayed equations.\end{proof}
  
\subsection{Lower Bounds on Hitting and Cover times}The lower bounds in the classical version of Matthew's bound~\cite[Section 11.2]{peresmix} do not directly (or at least
not obviously) imply a 
lower bound for the cover time of $\calC_{\alpha,\nu}(n)$ matching the upper bound obtained in  Proposition~\ref{cover:prop:maxHit}.
Fortunately, several extensions of Matthew's bound have been established. Among these, 
we will take advantage of the following result by Kahn, Kim, Lov\'asz and Vu.
\begin{theorem}{(\cite[Theorem 1.3]{KKLV00})}\label{cover:thm:kklv}
If $G=(V,E)$ is a connected graph, then
\begin{equation*} 
t_{\mathsf{cov}}(G)
  \geq \frac12 \cdot \max_{U\subseteq V}\big(\kappa_U\cdot\ln |U|\big)
  \qquad \text{where $\kappa_{U} := \min_{u,v\in U}[\Exu{u}{\tau_{v}}+ \Exu{v}{\tau_{u}}]$.}
\end{equation*}
\end{theorem}
The main difficulty in applying the preceding result, in order to lower bound the cover time of the center component of the HRG, is finding an adequate set $U$ of vertices of the center component of $\calG_{\alpha,\nu}(n)$ which is both sufficiently large and such that for every distinct pair of elements $u,v\in U$
the commute time between $u$ and $v$ is large enough, say $|U|=n^{\Omega(1)}$ and~$\kappa_U=\Omega(\ln n)$.
In order to obtain the desired lower bound, we show that such a set of vertices of $\calC_{\alpha,\nu}(n)$ is likely to exist.
The set's existence and the fact that it satisfies the necessary conditions 
relies on a structural result about induced paths of length~$\Theta(\log n)$ in the HRG.  This result
implies that a.a.s.~there are many (specifically, $n^{\Omega(1)}$) vertices of the center component of 
$\HRG_{\alpha,\nu}(n)$ whose
removal gives rise to a path component of order $\Omega(\log n)$.
The result is implicit in~\cite[Theorem 4.1]{KM15} and explicitly known (in an equivalent form) to the first author and Mitsche 
since their characterization of the size of the second-largest component of the HRG~\cite[Theorem 1.1]{KM19}.
We state and prove it below, since it is essential for our study of the commute time of the giant component, might also be useful in other settings,
and now seems of interest in its own right. The proof argument is somewhat different 
(we believe also simpler) from the one implicit in~\cite[Theorem 4.1]{KM15} and relies on the tiling structure constructed in the previous section.
Before going into its proof, we give an overview of it. 
To start, $B_O(R)$ is partitioned into several regions. 
The first region will be a ball centered at the origin with radius
as large as possible, while still remaining unlikely to contain vertices of $\calG_{\alpha,\nu}(n)$.
This region will be denoted $\calR_O$ and correspond to $B_O(\rho_{O})$ for a $\rho_{O}$ appropriately chosen.
Next, for a positive integer $m$ to be determined, we consider $m$ sectors of central angle $2\pi/m$ and 
intersect each one of them with $B_O(R)\setminus\calR_O$ thus
obtaining a partition $\{\calR_1,...,\calR_m\}$ of $B_O(R)\setminus\calR_O$.
We then show that we can pick $m=\Omega(n^{\zeta+\chi})$ for some $0<\zeta+\chi<1-\frac{1}{2\alpha}$ so that in each region~$\calR_a$, $a\in [m]$, with a small but non-negligible probability (specifically, with probability $\Omega(n^{-\widetilde{\chi}})$ for a small constant $0<\widetilde{\chi}<\chi$), and independent of what 
happens in the other regions except for $\calR_O$, there is a 
vertex $v_a$ of the center component of $\HRG_{\alpha,\nu}(n)$ whose removal disconnects~$\HRG_{\alpha,\nu}(n)$ 
into two connected components, one of which is a path $\calP_a$ of length $\Theta(\log n)$.
This suffices for our purpose, since it implies that the expected number of induced 
paths of length $\Theta(\log n)$ in~$\calC_{\alpha,\nu}(n)$ is $\Omega(n^{\zeta})$. So,
conditioned on $\calR_O$ being empty, the independence of events within different regions together with
standard concentration bounds would yield the result we seek to establish. 
This completes the high-level overview of our proof argument.

We first show that, for an adequate choice of $\rho_O$, the region
$\calR_O$ is unlikely to contain vertices of a HRG.
\begin{claim}\label{cover:cl:empty}
Let $\HRG:=\HRG_{\alpha,\nu}(n)$. Set $\rho_{O}:=(1-\frac{1+\eta}{2\alpha})R$ where $0<\eta<2\alpha-1$ is a constant.
 
Then, for $\calR_O:= B_O(\rho_{O})$, it holds that
\[
\PP(\text{$V(\HRG)\cap\calR_O=\emptyset$}) =  1-\calO(n^{-\eta}).
\]
\end{claim}
\begin{proof}
By Lemma~\ref{prelim:lem:ballsMeasure}, we have $\mu(\calR_O)=(1+o(1))e^{-(1+\eta)\frac{R}{2}}$.
Thus, since $R=2\ln(\frac{n}{\nu})$,
\[
\PP(\text{$V(\HRG)\cap\calR_O=\emptyset$}) = e^{-n\mu(\calR_O)} = e^{-\calO(n^{-\eta})} = 1-\calO(n^{-\eta}).\qedhere
\]
\end{proof}
Next, consider the equipartition of $\HH^2$ into $m$ sectors of central angle $2\pi/m$, say $\Upsilon_1,...,\Upsilon_m$, and let $\calR_a:=(\Upsilon_a\cap B_O(R))\setminus\calR_O$.
Each sector $\Upsilon_a$ will contain three contiguous sub-sectors. Moving clockwise around the
origin, the order in which these sub-sectors are encountered are; first a buffer sub-sector $\Upsilon^{\mathrm{bf}}_a$,
next a path sub-sector $\Upsilon_a^{\mathrm{pt}}$, and
finally a connection sub-sector~$\Upsilon_a^{\mathrm{ct}}$.
Their intersections with $\calR_a$ will be denoted $\calR_a^{\mathrm{bf}}$, $\calR_a^{\mathrm{pt}}$ 
and $\calR_a^{\mathrm{ct}}$, respectively. (See Figure~\ref{fig:upsilonA}~(a).)

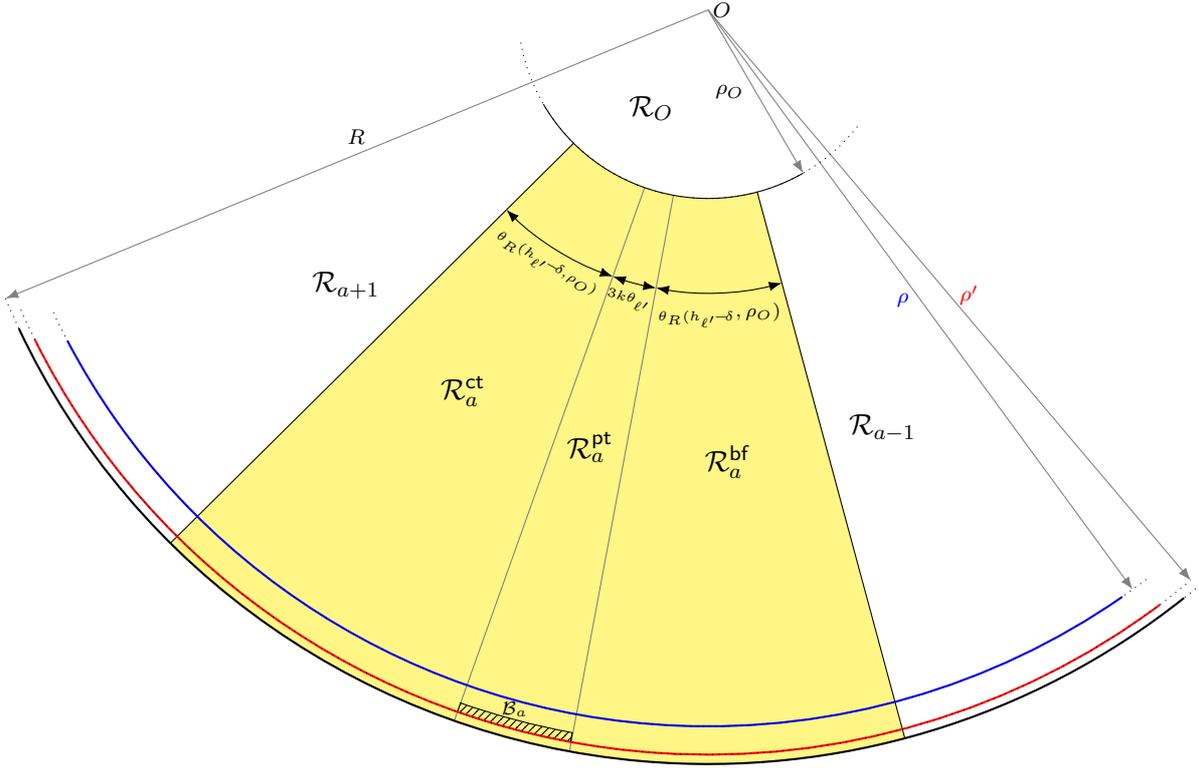
\begin{figure}
\begin{tikzpicture}[scale=1.25, rotate=180]
      \def\c{8}
      \def\rzero{2}
      \def\h{5}
      \def\n{5}
      \def\dlt{0.5}
      \node (O) at (\c,\c) {$\quad _O$};

      \draw[yellow!60,fill] (\c,\c) ++(360/12*1.5:\rzero) -- ++(360/12*1.5:\c-\rzero) arc (360/12*1.5:360/12*3.5:\c) -- ++(360/12*3.5:-\c+\rzero) arc(360/12*3.5:360/12*1.5:\rzero) -- cycle;

      \draw[Latex-Latex] (\c,\c) ++(360/12*1.5:3) arc(360/12*1.5:360/12*2.35:3) node[midway,below,rotate=-30] {\tiny $_{\theta_R(h_{\ell'}\!{-}\!\delta,\rho_O)}$};
      \draw[Latex-Latex] (\c,\c) ++(360/12*2.35:3) arc(360/12*2.35:360/12*2.65:3) node[midway,below,rotate=-15] {\tiny $_{3k\theta_{\ell'}}$};
      \draw[Latex-Latex] (\c,\c) ++(360/12*2.65:3) arc(360/12*2.65:360/12*3.5:3) node[midway,below] {\tiny $_{\theta_R(h_{\ell'}\!{-}\!\delta},\rho_O)$};

      \draw (\c,\c) ++(360/12*1.5:\rzero) -- ++(360/12*1.5:\c-\rzero);
      \draw (\c,\c) ++(360/12*3.5:\rzero) -- ++(360/12*3.5:\c-\rzero);
      \draw[gray] (\c,\c) ++(360/12*2.35:\rzero) -- ++(360/12*2.35:\c-\rzero);
      \draw[gray] (\c,\c) ++(360/12*2.65:\rzero) -- ++(360/12*2.65:\c-\rzero);

      \node at ({\c + \c*0.6 * cos(37.5)},{\c + \c*0.6 * sin(37.5)}) {$\calR_{a+1}$};
      \node at ({\c + \c*0.6 * cos(57.5)},{\c + \c*0.6 * sin(57.5)}) {$\calR^{\mathsf{ct}}_{a}$};
      \node at ({\c + \c*0.6 * cos(75)},{\c + \c*0.6 * sin(75)}) {$\calR^{\mathsf{pt}}_{a}$};
      \node at ({\c + \c*0.6 * cos(92.5)},{\c + \c*0.6 * sin(92.5)}) {$\calR^{\mathsf{bf}}_{a}$};
      \node at ({\c + \c*0.6*cos(112.5)},{\c + \c*0.6* sin(112.5)}) {$\calR_{a-1}$};

      \draw[pattern=north east lines] (\c,\c) ++(360/12*2.35:\c-0.2) -- ++(360/12*2.35:0.1) arc (360/12*2.35:360/12*2.65:\c-0.15) -- ++(360/12*2.65:-0.1) arc(360/12*2.65:360/12*2.35:\c-0.2) node[midway, yshift=4pt] {\small $_{\calB_a}$};

      \draw[dotted] ({\c + \c * cos(22.5)},{\c + \c * sin(22.5)}) arc (22.5:130:\c);
      \draw[-Latex,gray] (\c,\c) -- ++(22.5:\c) node[midway,above,black] {$_R$};
      \draw[black,thick] ({\c + \c * cos(25.0)},{\c + \c * sin(25.0)}) arc (25.0:128.75:\c);

      \draw[blue, thick] ({\c + (\c-0.4) * cos(27.5)},{\c + (\c-0.4) * sin(27.5)}) arc (27.5:125:\c-0.4);
      \draw[red, thick] ({\c + (\c-0.1) * cos(26.25)},{\c + (\c-0.1) * sin(26.25)}) arc (26.25:127:\c-0.1);

      \draw[dotted] ({\c + (\c-0.4) * cos(25)},{\c + (\c-0.4) * sin(25)}) arc (25:127.5:\c-0.4);
      \draw[dotted] ({\c + (\c-0.1) * cos(23.75)},{\c + (\c-0.1) * sin(23.75)}) arc (23.75:130:\c-0.1);
      \draw[gray,-Latex] (\c,\c) -- ++(120:\rzero) node[midway,left,black] {$_{\rho_O}$};
      \draw[-Latex,gray] (\c,\c) -- ++(126:\c-0.4) node[midway,left,blue] {$_\rho$};
      \draw[-Latex,gray] (\c,\c) -- ++(130:\c-0.1) node[midway,right,red] {$_{\rho'}$};
    
      \draw[dotted] ({\c + \rzero * cos(10)},{\c + \rzero * sin(10)}) arc (10:142.5:\rzero);
      \draw[black] ({\c + \rzero * cos(30)},{\c + \rzero * sin(30)}) arc (30:120:\rzero);
      \node at ({\c + \c*0.15*cos(60)},{\c + \c*0.15* sin(60)}) {$\calR_{O}$};
\end{tikzpicture}
\caption{(a) Shaded region corresponds to $\calR_a:=(\calR_a^{\mathrm{ct}}\cup\calR_a^{\mathrm{pt}}\cup\calR_a^{\mathrm{bf}})\setminus\calR_O$, and (b) hatched region corresponds to band $\calB_a$.
(Picture not to scale.)}\label{fig:upsilonA}
\end{figure}

Given a non-negative integer parameter $k$,
  our immediate aim is to achieve the following three objectives: 
\begin{itemize}
    \item \textbf{Goal 1:} Show that, with non-negligible probability (as a function of $k$), in a narrow band $\calB_a$ of $\calR_a^{\mathrm{\mathrm{pt}}}$ close
    to the boundary of $B_O(R)$, there is a collection of vertices that induce a path $\calP'_a$ of length~$\Theta(k)$. 
    \item \textbf{Goal 2:} Establish that with constant probability there is a path 
(all of whose vertices belong to~$\calR_a^{\mathrm{ct}}$) that guarantees one of the end vertices of $\calP'_a$
(and thus also all vertices of $\calP'_a$) belongs to the center component of $\HRG_{\alpha,\nu}(n)$.
    \item \textbf{Goal 3:} Prove that points outside band $\calB_a$ at distance at most $R$ from vertices of $\calP'_a$ (except for $\calO(1)$ vertices close to one of $\calP'_a$'s ends) would have to belong to either $\calR_O$ or some subregion of $\calR_a^{\mathrm{pt}}\cup\calR_a^{\mathrm{bf}}$, none of which, with non-negligible probability (as a function of $k$), 
contain vertices of~$\calG_{\alpha,\nu}(n)$ (we already showed this for $\calR_O$ in Claim~\ref{cover:cl:empty}). 
\end{itemize}
We say that a subgraph $P$ of a graph $G$ is a \emph{dangling path}
if~$P$ is a path in $G$ and there is a single edge of $G$ whose removal disconnects the graph into two connected components, one of which is~$P$.
Furthermore, we say that $P$ is a \emph{maximal dangling path} if it is not a proper subgraph of some other dangling path of $G$.
Achieving all three goals described above insures that, with non-negligible probability, there is a maximal dangling path of length $\Omega(\log n)$ in~$\calC_{\alpha,\nu}(n)$ contained in region $\calR_a$. Thus, showing that we can fix $m:=m(n)$ sufficiently large 
 yields that $\calC_{\alpha,\nu}(n)$ typically has many dangling paths. 

\medskip
Henceforth, consider the tiling $\calF(c)$ of Section~\ref{ssec:tiling} where $\eps:=\frac12\ln 2$ and $c:=c(\eps)$ is given in Claim~\ref{flow:claim-h}.
Also, let $\rho$ and $\rho'$ be defined as usual, i.e., as in~\eqref{flow:eqn:rho} and~\eqref{flow:eqn:rhoPrime}, respectively,
and fix $\ell$ and $\ell'$ as always, i.e., let them be the largest integers such that $h_{\ell}\leq\rho$ and $h_{\ell'}\leq\rho'$, respectively.
Also, for concreteness' sake, let $\phi_a\in [0,2\pi)$ be such that
\[
\Upsilon_a := \{(r,\theta) \mid \phi_a-\tfrac{\pi}{m}\leq\theta<\phi_a+\tfrac{\pi}{m}\}.
\]
and let 
\begin{equation}\label{eq:upspt}
\Upsilon_a^{\mathrm{pt}} := \{(r,\theta) \mid \phi_a-\tfrac{3}{2}k\theta_{\ell'}\leq\theta<\phi_a+\tfrac{3}{2}k\theta_{\ell'}\}.    
\end{equation}

Thus, $\Upsilon_a^{\mathrm{pt}}$ is a sector of central angle $3k\theta_{\ell'}$
with the same bisector as $\Upsilon_a$. 

We consider a band~$\calB_a$ of points of $\Upsilon^{\mathrm{pt}}_a$ with radial distance to the 
origin between $h_{\ell'}-\delta$ and~$h_{\ell'}$ where (with hindsight) we set $\delta=\ln\frac{9}{8}$. Equivalently, $\calB_a:=\Upsilon_a^{\mathrm{pt}}\cap (B_{O}(h_{\ell'})\setminus B_{O}(h_{\ell'}{-}\delta))$. See Figure~\ref{fig:upsilonA}~(b). 

Partitioning~$\Upsilon^{\mathrm{pt}}_a$ into sectors of central angle $\frac14\theta_{\ell'}$ we end up with $12k$
sectors whose intersections with $\calB_a$ partitions the latter into $\calB'_{1},...,\calB'_{12k}$ (these parts depend on the index $a$, but we omit this index from our notation in order to avoid over cluttering).
Without loss of generality, the indexing of the $\calB'_{j}$'s reflects the order in which each part is encountered when moving anti-clockwise around the origin. See Figure~\ref{fig:band}~(a).

Our next result establishes that, with non-negligible probability, there is an induced path~$\calP'_a$ in $\calG_{\alpha,\nu}(n)$ of length $\Theta(k)$ all of whose vertices are in $\calB_a$, 
thus achieving \textbf{Goal 1}.
\begin{claim}\label{cover:cl:goal1}
Let $\calG:=\calG_{\alpha,\nu}(n)$.
For $C'>0$ sufficiently large, 
with probability at least $e^{-\calO(C'k)}$, the event $\calE'_a$ defined as follows holds: For all $j\in [12k]$,
\begin{itemize}
\item if $j-1$ is a multiple of $3$, then
there is exactly one vertex of $\HRG$, say $v'_{j}$, that belongs to~$\calB'_{j}$, and 
\item 
if $j-1$ is not a multiple of $3$, then there is no vertex of $\HRG$ in $\calB'_{j}$.
\end{itemize}
Moreover, if $\calE'_a$ holds, then $V(\calG)\cap \calB_{a}$ induces a path $\calP'_a$ in $\calG$ of length $\Theta(k)$.
\end{claim}
\begin{proof}
By Lemma~\ref{prelim:lem:angles}
and~\ref{prelim:lem:ballsMeasure}, for all $j\in [12k]$, the expected number of vertices in~$\calB'_{j}$ is
\[
\EE|V(\HRG)\cap\calB'_{j}| = n\frac{\theta_{\ell'}}{4}\mu(B_{O}(h_{\ell'})\setminus B_{O}(h_{\ell'}{-}\delta))
  = \Theta(n\theta_{\ell'}e^{-\alpha(R- h_{\ell'})})
  = \Theta(\nu e^{(1-\alpha)(R-h_{\ell'})}).
\]
By our choice of $h_{\ell'}$, we have that $R-h_{\ell'}=\frac{1}{1-\alpha}\ln(\frac{2C'}{\nu})+\Theta(1)$ and conclude that $\lambda:=\EE|V(\HRG)\cap\calB'_{j}|=\Theta(C')$.
The probability that any single one of the $\calB'_{j}$'s does not contain a vertex of~$\HRG$ 
is $e^{-\lambda}$, and the probability that it contains exactly one vertex is $\lambda e^{-\lambda}$.
Since the~$\calB'_{j}$'s are disjoint, the number of vertices of $\calG$ contained within
each one is independent. Thus, we can choose $C'$ large enough so the probability that $\calE'_a$ occurs is at least $(\lambda e^{-3\lambda})^{4k} = e^{-\calO(C'k)}$.

For the `moreover' part of the statement, if $\calE'_a$ occurs, then 
$V(\calG) \cap \calB_a = \{ v'_{3j-2} \mid j\in [4k]\}$, so it has size~$4k$. It suffices to show that it also induces a path in $\calG$.
Indeed, note that the angle at the origin between $v'_{3j-2}$ and $v'_{3j+1}$ is at most $\theta_{\ell'}$.
Since both $v'_{3j-2}$ and $v'_{3j+1}$ belong to $B_O(h_{\ell'})\setminus B_O(h_{\ell'}-\delta)$, 
by definition of $\theta_R(\cdot,\cdot)$, both vertices are within distance at most $R$ of each other, so there is an edge between them in $\calG$. On the other hand, the angle at the origin spanned by $v'_{3j-2}$ and $v'_{3j+4}$ is at least  $\frac{5}{4}\theta_{\ell'}=\frac54(1+o(1))e^{-\delta}\theta_{R}(h_{\ell'}{-}\delta,h_{\ell'}{-}\delta)$, so by our choice of $\delta$, this angle is at least $\frac{10}{9}(1+o(1))\theta_{\ell'}$. Hence, $v'_{3j-2}$ and $v'_{3j+4}$ are 
at distance at least $R$ from each other and no edge between them exists in $\calG$. 
\end{proof}

\begin{remark}\label{cover:rem:goal1}
To determine whether the event $\calE'_a$ defined in Claim~\ref{cover:cl:goal1} is valid only requires exposing the region $\calB_a$. The probabilistic guarantees in Claim~\ref{cover:cl:goal1} are also dependent solely on the distribution of vertices in 
$\calB_a$.
\end{remark}

Now, recall that the points in a level $\ell'$ tile (respectively, half-tile) span an angle at the origin equal to $\theta_{\ell'}$ (respectively, $\frac12\theta_{\ell'})$. Also, note that points within each $\calB'_j$ span an angle at the origin equal to $\frac14\theta_{\ell'}$. Without loss of generality, we may assume, with hindsight, that the tiling was defined so that $H:=\calB'_1\cup\calB'_2$ is one of its level $\ell'$ half-tiles.\footnote{This requires arguing with respect to a different tiling for each value of $a\in [m]$, but causes no problem.}
Next, we recursively specify a sequence of half-tiles of $\calF(c)$. First, we
set~$s=\ell'$ and let $H^{(1)}_{s}$ be the half-tile that is contiguous to $H$ obtained by rotating the latter clockwise by the adequate angle (i.e., by $\frac12\theta_{\ell'}$).  Given a half-tile $H^{(i)}_{s}$ at level $s$ of~$\calF(c)$, $i\in [5]$, let $H^{(i+1)}_{s}$ be the half-tile, also at level $s$, that is contiguous to $H^{(i)}_{s}$ and is again obtained by rotating the latter clockwise by the adequate angle (i.e., by $\frac12\theta_s$). Next, let $H^{(1)}_{s-1}$ be the parent half-tile of $H^{(6)}_{s}$. Note that $H^{(1)}_{s-1}$ is at level $s-1$. The process is repeated with $s-1$ instead of~$s$ until reaching level $\ell$.

Henceforth, the union over $s\in\{\ell,...,\ell'\}$ and $i\in [6]$ of the $H_s^{(i)}$'s will be denoted $\calH_a$ (see Figure~\ref{fig:band}~(b)).

\begin{figure}
\begin{tikzpicture}[scale=0.85, rotate=175,every node/.style={inner sep=0,outer sep=0}]

  \clip (13.5,31) rectangle (31.7,41.1);
  \def\R{21}
  \def\c{20}        
  \def\angStrt{80}  
  \def\dlt{0.85}
  \def\dlth{0.25}
  \node (O) at (\c,\c) {$O$};

  \draw[gray] (O) -- ++(\angStrt-\dlt:\R);
  \draw[gray] (O) -- ++(\angStrt+27*\dlt:\R);
 
  \draw[gray!30,fill] (O) ++(\angStrt-\dlt:\c+\dlth) arc (\angStrt-\dlt:\angStrt+27*\dlt:\c+\dlth) -- ++(\angStrt+27*\dlt:\dlth) arc (\angStrt+27*\dlt:\angStrt-\dlt:\c+2*\dlth) -- cycle;

  \foreach \i in {0,1,...,12,19,20,...,27} {
    \draw (O) ++(\angStrt+\dlt*\i:\c+\dlth) -- ++(\angStrt+\dlt*\i:\dlth) {};
  }
  \foreach \t in {0,1,2} {
    \pgfmathparse{\t}\edef\offset{\pgfmathresult}
    \pgfmathparse{-11*\t-1+\offset}\edef\init{\pgfmathresult}
    \pgfmathparse{\init-2}\edef\step{\pgfmathresult}
    \pgfmathparse{\init-12}\edef\endd{\pgfmathresult}
    \draw[fill,gray!60] (O) ++(\angStrt+\dlt*\init:\c-2*\t*\dlth) arc (\angStrt+\dlt*\init:\angStrt+\endd*\dlt:\c-2*\t*\dlth) -- ++(\angStrt+\endd*\dlt:2*\dlth) arc (\angStrt+\endd*\dlt:\angStrt+\init*\dlt:\c-2*\t*\dlth+2*\dlth) -- cycle; 
    \foreach \i in {\init,\step,...,\endd} {
      \draw (O) ++(\angStrt+\dlt*\i:\c-2*\t*\dlth) -- ++(\angStrt+\dlt*\i:2*\dlth);
  }

  \foreach \t in {0,2,4, 6, 12} {    
    \draw[dashed] (O) ++(\angStrt-32.5:\c-\t*\dlth) arc (\angStrt-32.5:\angStrt+27.5:\c-\t*\dlth);
  }
    
  \draw[dotted] (O) ++(\angStrt-22.5:\R) arc (\angStrt-22.5:\angStrt+27.5:\R);  
  \draw (O) ++(\angStrt-20.0:\R) arc (\angStrt-20.0:\angStrt+25.0:\R);

  \draw[dotted,red] (O) ++(\angStrt-22.5:\c+2*\dlth) arc (\angStrt-22.5:\angStrt+27.5:\c+2*\dlth);  
  \draw[red,thick] (O) ++(\angStrt-20.0:\c+2*\dlth) arc (\angStrt-20.0:\angStrt+25.0:\c+2*\dlth);
  
  \draw[dotted,red] (O) ++(\angStrt-22.5:\c-14*\dlth) arc (\angStrt-22.5:\angStrt+27.5:\c-14*\dlth);  
  \draw[red,thick] (O) ++(\angStrt-20.0:\c-14*\dlth) arc (\angStrt-20.0:\angStrt+25.0:\c-14*\dlth);

  \draw[Latex-Latex] (O) ++(\angStrt-\dlt:\c-30*\dlth) arc(\angStrt-\dlt:\angStrt+27*\dlt:\c-30*\dlth) node[midway,above] {$3k\theta_{\ell'}$};
  
  \draw[thick,dotted] (O) ++(\angStrt+15*\dlt:\c-10*\dlth) -- ++(\angStrt+19*\dlt:2*\dlth);
  \draw[thick,dotted] (O) ++(\angStrt-14*\dlt:\c-10*\dlth) -- ++(\angStrt-14*\dlt:2*\dlth);
 
  \draw (O) ++(\angStrt-14*\dlt:\c-22*\dlth) node {$\Upsilon_{a}^{\mathrm{ct}}$};
  \draw (O) ++(\angStrt+15*\dlt:\c-22*\dlth) node {$\Upsilon_a^{\mathrm{pt}}$};
  \draw (O) ++(\angStrt+35*\dlt:\c-22*\dlth) node {$\Upsilon_{a}^{\mathrm{bf}}$};
 
  \draw[red] (O) ++(\angStrt-28*\dlt:\c-14*\dlth) node {\tiny ${\rho}$};
  \draw[red] (O) ++(\angStrt-27.5*\dlt:\c+2*\dlth) node {\tiny ${\rho'}$}; 
  \draw (O) ++(\angStrt-27.5*\dlt:\c+4*\dlth) node {\tiny ${R}$}; 
  \draw (O) ++(\angStrt+33.5*\dlt:\c-13*\dlth) node[rotate=12] {\tiny{level\! $\ell\!\!+\!\!1$}};
  \draw (O) ++(\angStrt+32.5*\dlt:\c+1*\dlth) node[rotate=12] {\tiny{level\! $\ell'$}};
  
  \draw[fill] (O) ++(\angStrt+26.5*\dlt:\c+0.3) node (AA) {};
  \draw[fill] (O) ++(\angStrt+23.4*\dlt:\c+0.33) node (BB) {};
  \draw[fill] (O) ++(\angStrt+20.7*\dlt:\c+0.4) node (BBA) {};
  \draw[fill] (O) ++(\angStrt+11.5*\dlt:\c+0.25) node (CC) {};
  \draw[fill] (O) ++(\angStrt+8.5*\dlt:\c+0.41) node (DD) {};
  \draw[fill] (O) ++(\angStrt+5.3*\dlt:\c+0.3) node (EEA) {};
  \draw[fill] (O) ++(\angStrt+2.7*\dlt:\c+0.365) node (EE) {};  
  \draw[fill] (O) ++(\angStrt-0.7*\dlt:\c+0.31) node (A) {}; 
  \draw[fill] (O) ++(\angStrt-2.7*\dlt:\c+0.1) node (B) {};
  \draw[fill] (O) ++(\angStrt-4.7*\dlt:\c+0.4) node (C) {};
  \draw[fill] (O) ++(\angStrt-6.7*\dlt:\c+0.4) node (D) {};
  \draw[fill] (O) ++(\angStrt-8.7*\dlt:\c+0.2) node (E) {}; 
  \draw[fill] (O) ++(\angStrt-10.7*\dlt:\c+0.3) node (F) {};
  \draw[fill] (O) ++(\angStrt-12.7*\dlt:\c+0.1) node (G) {};
  \draw[fill] (O) ++(\angStrt-11.3*\dlt:\c-0.2) node (H) {};
  \draw[fill] (O) ++(\angStrt-14.8*\dlt:\c-0.4) node (I) {}; 
  \draw[fill] (O) ++(\angStrt-16.4*\dlt:\c-0.3) node (J) {};
  \draw[fill] (O) ++(\angStrt-18.4*\dlt:\c-0.3) node (K) {};
  \draw[fill] (O) ++(\angStrt-20.5*\dlt:\c-0.1) node (L) {};  
  \draw[fill] (O) ++(\angStrt-22.3*\dlt:\c-0.25) node (M) {}; 
  \draw[fill] (O) ++(\angStrt-22.4*\dlt:\c-0.9) node (N) {};
  \draw[fill] (O) ++(\angStrt-24.5*\dlt:\c-0.7) node (OO) {};
  \draw[fill] (O) ++(\angStrt-26.4*\dlt:\c-0.6) node (P) {};  
  \draw[fill] (O) ++(\angStrt-28.6*\dlt:\c-0.55) node (Q) {};
  \draw[fill] (O) ++(\angStrt-30.8*\dlt:\c-0.65) node (R) {};  
  \draw[fill] (O) ++(\angStrt-32.6*\dlt:\c-0.75) node (S) {}; 
  \foreach \aa/\dd in {26.5/0.3, 23.4/0.33, 20.7/0.4,  11.5/0.25, 5.3/0.3, 2.7/0.365, -0.7/0.31, -2.7/0.1, -4.7/0.4, -6.7/0.4, -8.7/0.2, -10.7/0.3, -12.7/0.1, -11.3/-0.2, -14.8/-0.4, -16.4/-0.3, -18.4/-0.3, -20.4/-0.1, -22.3/-0.25, -22.5/-0.9, -24.5/-0.7, -26.4/-0.6, -28.6/-0.55, -30.8/-0.65, -32.6/-0.75} {
    \draw[fill,black] (O) ++(\angStrt+\aa*\dlt:\c+\dd) circle (0.04);
  }
  \foreach \aa/\bb in { -0.4/1, 2.8/4, 5.5/7, 8.6/10 } { 
    \draw (O) ++(\angStrt+\aa*\dlt:\c+0.75) node {$_{\calB'_{\bb}}$};
    }
  \node at (DD) {\large $\star$};
  \draw (AA) to [in=135,out=45] (BB) to [in=-135, out=-45] (BBA) to [in=-135, out=45] (CC) to [in=0,out=45] (DD) -- (EEA) -- (EE) -- (A) -- (B) -- (C) -- (D) -- (E) -- (F) -- (G) -- (H) -- (I) -- (J) -- (K) -- (L) -- (M) -- (N) -- (OO) -- (P) -- (Q) -- (R) -- (S);  
}
\end{tikzpicture}\caption{Picture not to scale. For simplicity, we have illustrated the case where $h_{\ell}=\rho$ and $h_{\ell'}=\rho'$.
(a) The lightly shaded region corresponds to band~$\calB_a$. Each division of the shaded region corresponds to a $\calB'_j$. 
(b) The darker shaded region correspond to~$\calH_a$, i.e., the union of all half-tiles $H_s^{(i)}$ with $s\in\{\ell,...,\ell'\}$ and $i\in [6]$.
(c) Vertex $v_a$ is represented as a star.
(d) Vertices in $\calP_a$ are represented as a curved segment.
The existence of the path represented as a piecewise linear segment guarantees that the path $\calP_a$ belongs to the center component.}\label{fig:band}
\end{figure}
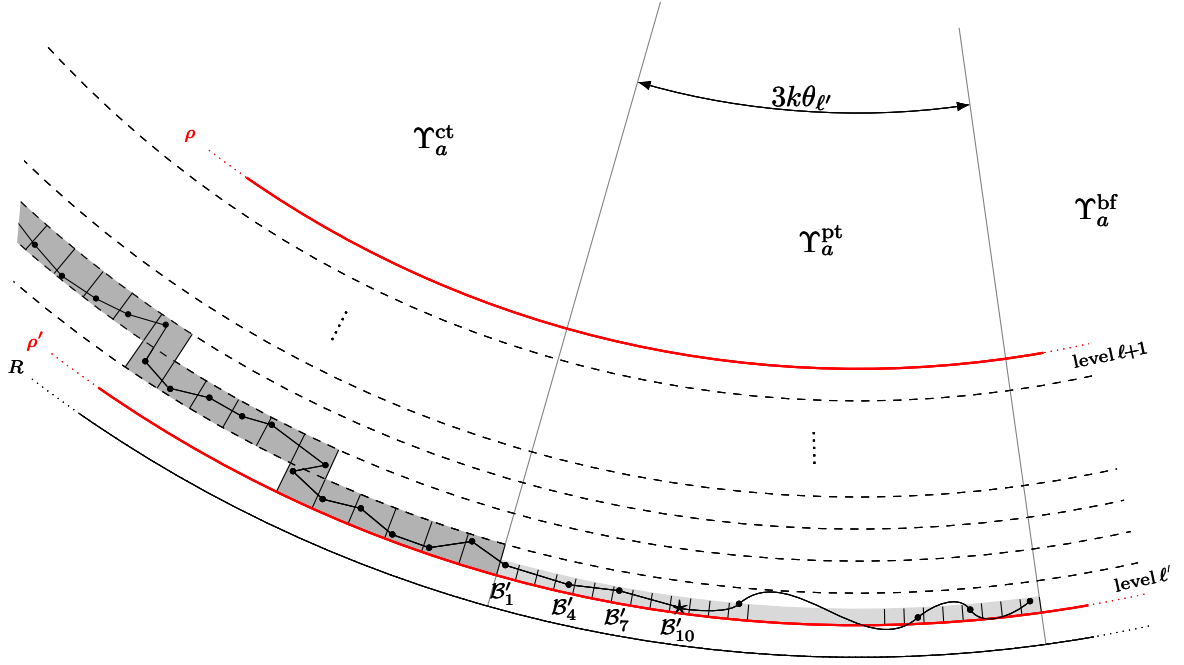
Our next result establishes that, with some positive constant probability, every vertex in~$H^{(1)}_{\ell'}$ is connected  
to a vertex in $V(\HRG_{\alpha,\nu}(n))\cap B_O(\rho)$
via a path all of whose vertices lie in $\calH_a$.
Since a.a.s.~the latter vertices belong to the center component of $\HRG_{\alpha,\nu}(n)$ (by Lemma~\ref{flow:lem:nonfaultyTiles} and Remark~\ref{rem:centerAndClique}) and vertices in $H^{(1)}_{\ell'}$ are neighbours of those in $H$ (because $H$ and $H^{(1)}_{\ell'}$ are sibling half-tiles)
the result allows us to attain \textbf{Goal~2}. 
\begin{claim}\label{cover:cl:goal2}
Let $\calG:=\calG_{\alpha,\nu}(n)$.
For $C'>0$ sufficiently large, with probability at least~$\frac12$
the event $\calE''_a$ defined as follows holds: For all $s\in\{\ell,...,\ell'\}$ and $i\in [6]$ the half-tile~$H^{(i)}_s$ contains at least one vertex of $\HRG$. Moreover, when $\calE''_a$ holds, a.a.s.~every vertex in $V(\calG)\cap H_{\ell'}^{(1)}$ belongs to the center component of~$\calG$.
\end{claim}
\begin{proof}
Because the $H_s^{(i)}$'s are disjoint, independence and standard arguments imply that
\begin{align*}
\PP(\calE''_a) & = \prod_{s=\ell}^{\ell'}\prod_{i\in [6]}\big(1-\exp\big({-}\EE|V\cap H_{s}^{(i)}|\big)\big).
\end{align*}
Since the expected number of vertices in a tile is twice that of a half-tile of the same level, by  Claim~\ref{flow:claim:measure}, we have  
$\EE|V\cap H_{s}^{(i)}|\geq \frac{1}{4}\nu e^{(1-\alpha)(R-h_s)}$. Thus, since $h_{\ell'}\leq\rho'$, by definition of 
$\rho'$, we get $\EE|V\cap H_s^{(i)}|\geq \frac14\nu e^{(1-\alpha)(R-h_{\ell'})}\geq \frac12 C'$.
So, recalling that $\ln(1-x)\geq -\frac{x}{1-x_0}$ for $0\leq x\leq x_0<1$ and provided $C'\geq\ln 4$, 
\[
\ln\PP(\calE''_a) \geq -12\sum_{s=\ell}^{\ell'}\exp\big({-}\tfrac{1}{4}\nu e^{(1-\alpha)(R-h_s)}\big)\geq
-12\sum_{s=\ell}^{\ell'}\exp\big({-}\tfrac{1}{2}C'e^{(1-\alpha)(h_{\ell'}-h_s)}\big).
\]
By Claim~\ref{flow:claim-h}, $h_{\ell'}-h_s\geq (\ell'-s)\ln 2-\epsilon$, so taking $\gamma:=\frac{1}{2}C' e^{-(1-\alpha)\epsilon}$ and $\beta:=(1-\alpha)\ln 2\geq \frac12(1-\alpha)$, 
\[
\sum_{s=0}^{\ell'-\ell}\exp\big(-\gamma e^{\beta s}\big) \leq e^{-\gamma}+\int_1^{\ell'-\ell+1}\exp\big(-\gamma e^{\beta x}\big)\mathrm{d}x
\leq e^{-\gamma}+\tfrac{1}{\beta}\int_1^{\infty}y^{-1}e^{-\gamma y}\mathrm{d}y.
\]
Since $y^{-1}e^{-\gamma y}\leq e^{-\gamma y}$ if $y\geq 1$, we get that 
\[
\sum_{s=0}^{\ell'-\ell}\exp\big(-\gamma e^{\beta s}\big) \leq e^{-\gamma}+\tfrac{1}{\beta\gamma}e^{-\gamma}
\leq \tfrac{2}{1-\alpha}\big(1+\tfrac{1}{\gamma}\big)e^{-\gamma}
=:\xi.
\]
Taking $C'$ to be a large constant, we can make $\gamma$ large and $\xi$ as small as needed, say small
enough so that 
$\PP(\calE''_a)\geq e^{-12\xi}\geq\frac12$ as sought.

For the `moreover' part, observe that by Remark~\ref{flow:remark:key},
if $\calE''_a$ occurs, then any $v\in V(\calG)\cap H_{\ell'}^{(s)}$ is connected via a path to a vertex in $V(\calG)\cap H_{\ell}^{(1)}$.
Since Lemma~\ref{flow:lem:nonfaultyTiles} and Remark~\ref{rem:centerAndClique} imply that, a.a.s., all vertices in $B_O(\rho)\supseteq H_{\ell}^{(1)}$ belong to~$\calC_{\alpha,\nu}(n)$, the sought after conclusion follows.
\end{proof}

\begin{remark}\label{cover:rem:goal2}
Similar to what we observed in Remark~\ref{cover:rem:goal1}, ascertaining whether $\calE''_a$
defined in Claim~\ref{cover:cl:goal2} holds only requires exposing region $\calH_a$, and again,  the probabilistic guarantees in said claim are also dependent solely on the distribution of vertices in~$\calH_a$.
\end{remark}
If $\calE'_a$ (as defined in Claim~\ref{cover:cl:goal1}) occurs, there is exactly one vertex
  of $\HRG:=\calG_{\alpha,\nu}(n)$ in~$\calB'_{10}$, henceforth denoted~$v_a$ (see Figure~\ref{fig:band} (c)), 
  and such vertex is an end of the subpath $\calP_a$ of $\calP'_a$ induced by the vertices of $\HRG$ in $\calB_a\setminus(\calB'_1\cup ...\cup\calB'_9)$ (see Figure~\ref{fig:band} (d)).
If $\calE''_a$ defined in Claim~\ref{cover:cl:goal2} also holds, then 
$\calP_a$ belongs to the center component.
Unfortunately, exposing regions other than~$\calB_a$ and $\calH_a$ might prevent $\calP_a$ being a dangling path in $\calC_{\alpha,\nu}(n)$. 
However, the two next results, together with Claim~\ref{cover:cl:empty}, 
ensure that $\calP_a$ is a dangling path of the center component after exposing said regions (thus, also fulfilling \textbf{Goal~3} stated above).
\begin{claim}\label{cover:cl:goal3a}
Let $\calG:=\calG_{\alpha,\nu}(n)$.
For any $C'>1$, with probability $\exp(-\calO(k\cdot(C')^{\frac{1}{1-\alpha}}))$, the event~$\calE'''_a$ defined as follows holds: There are no vertices of $\calG$ in  
\[
\big(\bigcup_{p\in\calB_a}B_p(R)\setminus\calR_O\big)\setminus (\calB_a\cup \calH_a).
\]
\end{claim}
\begin{proof}
Let $q$ and $s$ be the two (distinct) points closest to the origin that belong to the boundaries of both $\calB_a$ and $\Upsilon_a^{\mathrm{pt}}$.
We observe that if $p\in\calB_{a}$, then $B_{p}(R)\setminus\calR_O$ is contained in $(\Upsilon_a^{\mathrm{pt}}\cup B_q(R)\cup B_s(R))\setminus\calR_O$.
Indeed, consider $\widetilde{p}\in B_p(R)\setminus\calR_O$.
If $\widetilde{p}\in\Upsilon_a^{\mathrm{pt}}$, then we are done. 
On the other hand, if $\widetilde{p}\not\in\Upsilon_a^{\mathrm{pt}}$, then it must be the case that the angle between $p$ and $\widetilde{p}$ is at least
as large as the smallest angle, say $\phi$, between either $q$ and $\widetilde{p}$, or
$s$ and $\widetilde{p}$. 
Since $\widetilde{p}\in B_p(R)$, the angle between $p$ and $\widetilde{p}$ is at most~$\theta_R(r_p,r_{\widetilde{p}})$.
Recalling that  $r_p\geq r_q=r_s$, by Remark~\ref{prelim:rem:monotonicity}, 
we have $\theta_R(r_p,r_{\widetilde{p}})\leq \theta_R(r_q,r_{\widetilde{p}})=\theta_R(r_s,r_{\widetilde{p}})$.
Hence, $\phi$ is at most $\theta_R(r_q,r_{\widetilde{p}})=\theta_R(r_s,r_{\widetilde{p}})$ which implies that $\widetilde{p}$ belongs either to $B_q(R)$ or $B_s(R)$ and establishes the claimed observation. 
As an immediate consequence, we obtain that 
the probability that $\calE_a'''$ occurs is at least the probability 
that $V(\calG)\cap((\Upsilon^{\mathrm{pt}}_a\cup B_q(R)\cup B_s(R))\setminus\calR_O)$ is empty.
To lower bound said probability first note that,
\[
\mu((\Upsilon^{\mathrm{pt}}_a\cup B_q(R)\cup B_s(R))\setminus\calR_O)
\leq \mu(\Upsilon^{\mathrm{pt}}_a)+\mu(B_q(R))+\mu(B_s(R)).
\]
Next observe that, by  Claim~\ref{flow:claim-h}, we have $\mu(\Upsilon_a^{\mathrm{pt}})=3k\theta_{\ell'}=\Theta(k\frac{\nu}{n}(\frac{2C'}{\nu})^{\frac{1}{1-\alpha}})=\Theta(\frac{k}{n}(C')^{\frac{1}{1-\alpha}})$.
Moreover, by Lemma~\ref{prelim:lem:ballsMeasure} and our choice of $\rho'$, both $\mu(B_q(R))$ and $\mu(B_s(R))$ equal 
$\calO(e^{-\frac12\rho'})=\calO(\frac{\nu}{n}(\frac{2C'}{\nu})^{\frac{1/2}{1-\alpha}})$.
Hence,  
\[
\PP(\calE'''_a) 
= \exp(-n\mu((\Upsilon^{\mathrm{pt}}_a\cup B_q(R)\cup B_s(R))\setminus\calR_O))
= \exp(-\calO(k\cdot (C')^{\frac{1}{1-\alpha}})),
\]as claimed.
\end{proof}

\begin{claim}\label{cover:cl:goal3b}
All points in $\calB_a$ within distance at most $R$ of some point in $\calH_a$ belong to $\calB'_1\cup ...\cup\calB'_9$.
\end{claim}
\begin{proof}
Recall that $\calH_a$ is the union of all $H^{(i)}_s$ with $i\in [6]$ and 
  $s\in\{\ell,...,\ell'\}$.

Let $p\in \bigcup_{i\in [6]} H^{(i)}_{\ell'}$ and $q\in\calB_a$ be two points at distance at most $R$. 
Since $r_p\geq h_{\ell'-1}$ and $r_q\geq h_{\ell'}-\delta$, by Remark~\ref{prelim:rem:monotonicity}, we get
that $\theta_R(r_p,r_q)\leq \theta_R(h_{\ell'-1},h_{\ell'}-\delta)$.
Thus, by our choice of $\delta$ and $\epsilon$, by Claim~\ref{flow:claim-h} and Lemma~\ref{prelim:lem:angles},
\[
\tfrac94\theta_{\ell'}\geq\tfrac94\cdot 2e^{\frac12(R-2h_{\ell'}-\epsilon)}
\geq \tfrac94 e^{\frac12(R-h_{\ell'}-h_{\ell'-1})}
= \sqrt{\tfrac98}\cdot 2e^{\frac12(R-(h_{\ell'}-\delta)-h_{\ell'-1})}> \theta_R(h_{\ell'-1},h_{\ell'}-\delta).
\]
Hence, $\frac94\theta_{\ell'} \geq \theta_R(r_p,r_q)$.
By definition, each of $\calB'_1,...,\calB'_9$ spans an angle at the origin equal to~$\frac14\theta_{\ell'}$.
Hence, we conclude that $q\in \calB'_1\cup ...\cup \calB'_9$. 

Similarly, for values of $s<\ell'$ and a point $p\in\bigcup_{i\in [6]} H^{(i)}_{s}$ within distance at most $R$ from~$q\in\calB_a$, we have
$\theta_R(r_p,r_q)\leq \theta_R(h_{s-1},h_{\ell'}-\delta)$.
Since $h_{\ell'}-\delta\geq h_{\ell'-1}\geq h_{s-1}$, 
by Claim~\ref{flow:claim-h},
we get $\theta_R(h_{s-1},h_{\ell'}-\delta)\leq\theta_R(h_{s-1},h_{s-1})=\theta_{s-1}$.
However, $\cup_{j=s+1}^{\ell'}\sum_{i\in [6]} H^{(i)}_s$ spans an angle at the origin which is at least
\[
\sum_{j=s+1}^{\ell'}(5\cdot\tfrac12\theta_j) 
= \tfrac52\theta_{\ell'}\sum_{j=0}^{\ell'-s-1}2^j
= \tfrac52(2^{\ell'-s}-1)\theta_{\ell'}\geq 2^{\ell'-s+1}\theta_{\ell'}-\tfrac94\theta_{\ell'}=\theta_{s-1}-\tfrac94\theta_{\ell'},
\]
where in the last inequality we used that $s<\ell'$. Hence, again we have that $q\in \calB'_1\cup ...\cup \calB'_9$. 
\end{proof}

We want $\Upsilon_a^{\mathrm{ct}}$ (respectively, $\Upsilon_a^{\mathrm{bf}}$) to be a sector clockwise (respectively, anti-clockwise) contiguous to
$\Upsilon_a^{\mathrm{pt}}$ such that the following two conditions hold: (1) $\calH_a$ is completely contained in~$\Upsilon_a^{\mathrm{ct}}$, and (2) $(\bigcup_{p\in\calB_a}B_p(R)\setminus\calR_O)$ is completely contained in $\Upsilon_a^{\mathrm{ct}}\cup\Upsilon_a^{\mathrm{pt}}\cup\Upsilon_a^{\mathrm{bf}}$.
The first condition is desirable so the set of vertices in $\calH_a$ as well as any path involving them is contained in $\Upsilon^{\mathrm{ct}}$. We require the second condition so that, provided 
$(\bigcup_{p\in\calB_a}B_p(R)\setminus\calR_O)\setminus(\calB_a\cup\calH_a)$ contains no vertices, we can
guarantee that vertices in $\calB_a$ may have as neighbours only vertices that belong to $\calH_a \cup \calR_O$. 
Since  
all $H_s^{(i)}$'s are contained within a sector contiguous to $\Upsilon_a^{\mathrm{pt}}$ of central angle
\[
6\sum_{j=\ell}^{\ell'}\frac{\theta_{j}}{2} = 3\theta_{\ell'}\sum_{j=0}^{\ell'-\ell}2^j
\leq 3\theta_{\ell'}2^{\ell'-\ell+1}=6\theta_\ell,
\]
and $(\bigcup_{p\in\calB_a}B_p(R)\setminus\calR_O)$ is contained within
a sector of central angle equal to the one of $\Upsilon_a^{\mathrm{pt}}$ (i.e., $3k\theta_{\ell'}$) plus $2\theta_R(h_{\ell'}-\delta,\rho_O)$,
it suffices to set the central angles of $\Upsilon_a^{\mathrm{ct}}$ and 
$\Upsilon_a^{\mathrm{bf}}$ equal to $\max\{6\theta_{\ell},\theta_R(h_{\ell'}-\delta,\rho_O)\}$. Recall the definition of  
$\Upsilon_a^{\mathrm{pt}}$, given by \eqref{eq:upspt}, and define:
\begin{align*}
\Upsilon_a^{\mathrm{ct}} & :=
\{ (r,\theta) \mid \phi_a+\tfrac{3}{2}k\theta_{\ell'} \leq \theta < \phi_a+\tfrac{3}{2}k\theta_{\ell'}+\max\{6\theta_{\ell},\theta_R(h_{\ell'}-\delta,\rho_O)\}, \\
\Upsilon_a^{\mathrm{bf}} & :=
\{ (r,\theta) \mid \phi_a-\tfrac{3}{2}k\theta_{\ell'}-\max\{6\theta_{\ell},\theta_R(h_{\ell'}-\delta,\rho_O)
\leq \theta<\phi_a-\tfrac{3}{2}k\theta_{\ell'}  \}.
\end{align*}
The next result shows that under mild conditions on $k$ we can fix $m=n^{\Omega(1)}$ so that 
$\Upsilon_a^{\mathrm{ct}}$,
$\Upsilon_a^{\mathrm{pt}}$,
and $\Upsilon_a^{\mathrm{bf}}$ are all contained in $\Upsilon_a$.
\begin{claim}\label{cover:cl:upsilona}
If $k=\calO(\log n)$, then
there is a sufficiently large constant $C''>0$ such that choosing $m:=\lfloor \frac{2\pi}{C''}\cdot n^{1-\frac{1+\eta}{2\alpha}}\rfloor$ guarantees that
$\Upsilon_a^{\mathrm{ct}}\cup\Upsilon_a^{\mathrm{pt}}\cup\Upsilon_a^{\mathrm{bf}}\subseteq\Upsilon_a$.
\end{claim}
\begin{proof}
Since $\Upsilon_a$ has the same bisector as $\Upsilon_a^{\mathrm{pt}}$, for $\Upsilon_a^{\mathrm{ct}}\cup\Upsilon_a^{\mathrm{pt}}\cup\Upsilon_a^{\mathrm{bf}}\subseteq\Upsilon_a$ to hold, we need that
\begin{equation}\label{eq:centralAngle}
3k\theta_{\ell'}+2\max\{6\theta_\ell,\theta_R(h_{\ell'}-\delta,\rho_O)\} \leq 2\pi/m.
\end{equation}
By our choice of $\ell'$ and $k$, we have $k\theta_{\ell'}=\calO(n^{-1}\log n)$.
By our choice of $\ell$, we have $\theta_\ell=\calO(n^{-1}(\log n)^{\frac{1}{1-\alpha}})$.
Recalling that $\rho_O=(1{-}\frac{1+\eta}{2\alpha})R$, since $\delta=\ln\frac{9}{8}=\Theta(1)$, by our choice of $\ell'$ and Lemma~\ref{prelim:lem:angles},
\[
\theta_R(h_{\ell'}-\delta,\rho_O) = \Theta(e^{\frac12(R-\rho'-\rho_0})
   = \Theta(n^{\frac{1+\eta}{2\alpha}}/n).
\]
Summarizing, for some sufficiently large $C''>0$ the left-hand side of~\eqref{eq:centralAngle} is upper bounded by $C''n^{\frac{1+\eta}{2\alpha}}/n \leq 2\pi/m$.
\end{proof}

The next result states that, with large probability, a HRG contains many dangling paths of length $\Omega(\ln n)$.
\begin{lemma}\label{cover:lem:inducedpaths}
For every $0<\zeta<1-\frac{1}{2\alpha}$ there is a positive constant $\Lambda:=\Lambda(\zeta)$
such that with probability at least $1-\Omega(n^{-\alpha(1-\frac{1}{2\alpha}-\zeta)})$ there are $\Omega(n^{\zeta})$
maximal dangling paths of length at least~$\Lambda\ln n$ in $\calC_{\alpha,\nu}(n)$.
\end{lemma}
\begin{proof}
Let $\chi:=\frac12(1-\frac{1}{2\alpha}-\zeta)$.
Note that $0<\zeta+\chi<1-\frac{1}{2\alpha}$. 
Fix $\eta$ in the definition of $\rho_O$ so that $1-\frac{1+\eta}{2\alpha}=\zeta+\chi$, thus $\eta:=\alpha(1-\frac{1}{2\alpha}-\zeta)<2\alpha-1$ so $\rho_O$ satisfies the hypothesis of Claim~\ref{cover:cl:empty}.  

Let  $m:=\lfloor \frac{2\pi}{C''}n^{\zeta+\chi}\rfloor$ and set $k:=\lceil\widetilde{\chi}\ln n\rceil$ for $\widetilde{\chi}>0$ to be chosen later.

For $a\in [m]$ and our choice of~$k$, let $X_a$ denote the indicator of the simultaneous occurrence of   the three events $\calE_a'$,
  $\calE_a''$ and $\calE_a'''$ as defined in Claims~\ref{cover:cl:goal1}, \ref{cover:cl:goal2} and \ref{cover:cl:goal3a}, respectively.
By said claims, $\EE(X_a)=e^{-\calO(k)}$ for all $a\in [m]$.
By Remarks~\ref{cover:rem:goal1} and~\ref{cover:rem:goal2}, Claims~\ref{cover:cl:goal3a} and~\ref{cover:cl:upsilona}, the random variable $X_a$ depends only on the distribution of vertices in region $\calR_a$.
Thus, since $\calR_1,...,\calR_m$ are disjoint, the random variables $X_1,...,X_m$ are independent.
Hence, if we let $X:=\sum_{a\in [m]} X_a$ and choose~$\widetilde{\chi}>0$ sufficiently small so that 
\[
\EE(X)=m\cdot e^{-\calO(k)}=e^{(\zeta+\chi)\ln n}\cdot e^{-\calO(\widetilde{\chi}\ln n)}=\Omega(n^{\zeta}),
\]
we get, by Chernoff's bound (see~\cite[Theorem A.1.7]{AlonSpencer}), 
that $\PP(X\leq\frac12\EE(X)) =\exp(-\Omega(n^{\zeta}))$.
Finally, observe that if $V(\calG_{\alpha,\nu}(n))\cap R_O$ is empty
(which, by Claim~\ref{cover:cl:empty}, happens with probability $1-\calO(n^{-\eta})$), the occurrence of $\calE_a'\cap\calE''_a\cap\calE'''_a$
together with Claim~\ref{cover:cl:goal3b} as well as the discussion preceding Claim~\ref{cover:cl:goal3a} imply that $\calP_a$
is a dangling path in $\calC_{\alpha,\nu}(n)$ of length $k-3=\Omega(\log n)$. 
It follows that,
with probability at least $1-\calO(n^{-\eta})$, in $\calC_{\alpha,\nu}(n)$, there are 
at least $\frac12\EE(X)=\Omega(n^{\zeta})$ maximal dangling paths
of length $\Omega(\log n)$ where the constant inside the asymptotic notation depends on $\zeta$. 
\end{proof}

We now have all the necessary ingredients to derive the lower bounds matching the bounds in Proposition~\ref{cover:prop:maxHit}. Together these two propositions prove Theorem~\ref{thm:maxhitcov}.

\begin{proposition}\label{cover:prop:minHit}
If $\calC:=\calC_{\alpha,\nu}(n)$, then a.a.s.~$t_{\mathsf{hit}}(\calC)=\Omega(n\ln n)$ and $t_{\mathsf{cov}}(\calC)=\Omega(n\ln^2 n)$. 
\end{proposition}
\begin{proof}  
  The lower bound on $t_{\mathsf{hit}}(\calG)$ is a direct consequence of the claimed lower bound on $t_{\mathsf{cov}}(\calG)$
  and Matthew's bound, so we concentrate on proving the bound on $t_{\mathsf{cov}}(\calG)$.

  Let $\zeta$ and 
  $\Lambda:=\Lambda(\zeta)$ be as in 
  the statement of Lemma~\ref{cover:lem:inducedpaths}. Let $U\subseteq V(\calC)$ be the collection of vertices $u$ 
  of $\calC$ which have degree $2$ in $\calC$ and are end vertices of a maximal dangling path $P_u$ of $\calC$ of length at least
  $\Lambda\ln(n)$.
  Now, let $u,v\in U$, $u\neq v$,
  and contract into a single vertex all vertices 
  of $\calC$ except for those that 
  belong to $P_{u}$ or $P_{v}$.
  One obtains a path of length at least $2\Lambda\ln(n)$.
  Rayleigh's Monotonicity Law (RML) (Theorem~\ref{thm:RML})
  implies that the contraction of edges does 
  not increase the effective resistance between any of the remaining pair of vertices. Since the effective resistance between the endpoints of a path with $\ell$ edges is $\ell$ (this follows directly from our working definition of effective resistance -- see~\eqref{eq:resdef}), by RML we get that $\Res{u}{v}\geq 2\Lambda\ln(n)$ for all distinct $u,v\in U$.  
  Since $\Exu{u}{\tau_{v}}+ \Exu{v}{\tau_{u}}=2|E(\calC)|\cdot \Res{u}{v}$ by Lemma~\ref{lem:commutetime}, 
  \[
  \kappa_{U} := \min_{u,v\in U}[\Exu{u}{\tau_{v}}+ \Exu{v}{\tau_{u}}]\geq 2|E(\calC)|\cdot\Lambda\ln(n).
  \]
  By Theorem~\ref{prelim:thm:giantEqualCentralComp}, we know that a.a.s.~$\calC$ is of order
  $\Theta(n)$, and since $\calC$ is connected, we also have that a.a.s.~$|E(\calC)|=\Omega(n)$.
  Thus, we conclude that $\kappa_{U}=\Omega(n\ln n)$ a.a.s., so the lower bound on $\tcov(\mathcal{C})$ follows directly 
  from Lemma~\ref{cover:lem:inducedpaths} and Theorem~\ref{cover:thm:kklv}.
\end{proof}
 We also remark that Proposition~\ref{cover:prop:minHit} shows that the maximum effective resistance between any two vertices in the center component is a.a.s.~$\Theta(\log n)$.

\section{Commute Times}\label{sec:commute}
In previous sections we analyzed parameters of simple random walks on the giant of the HRG, with a focus on parameters which depend on the graph as a whole and not the start and end vertices of the walks. In this section we give  (under mild assumptions) tight estimates (up to polylogarithmic factors) for the commute time between a pair $u, v$ of vertices added to~$\HRG_{\alpha,\nu}(n)$. The estimates depend on the radius of the added vertices.

We address the upper and lower bounds in two separate sections below. First, we establish the upper bound via the commute time identity and bounding effective resistances from above via the energy dissipated by a specially crafted network flow.
Then, we consider the lower bound, establishing it by exhibiting small edge-cuts whose removal leaves $u$ and $v$ in distinct connected components. 
Put together, these two bounds yield Theorem~\ref{ap:thm:commut}.

\subsection{Upper bound}
Let $\HRG:=\HRG_{\alpha,\nu}(n)$ and consider $\rho$ as given by~\eqref{flow:eqn:rho}.
Fix two points in $B_O(\rho)$ and 
add to $\HRG$ vertices~$s$ and $t$ at the previously chosen points (one vertex per point).
By Lemma~\ref{flow:lem:nonfaultyTiles} and Proposition~\ref{flow:prop:energy}, a.a.s., the effective resistance between $s$ and $t$ is $\calO(1)$. However, the proof of Proposition~\ref{flow:prop:energy} shows something stronger. 
Specifically, that for $s,t\in V(\HRG)\cap (B_{O}(\rho)\setminus B_{O}(\frac{R}{2}))$ it holds that  
$\Res{s}{t}=\calO(e^{-(1-\alpha)(R-r_s)}+e^{-(1-\alpha)(R-r_t)})$.
On the other hand, it is well known (and follows easily from Lemma~\ref{prelim:lem:ballsMeasure}) that the degree $d(w)$ of  a vertex $w\in V(\HRG)\cap B_O(\rho)$ is, w.h.p.,~$\Theta(e^{\frac12(R-r_w)})$. 
Summarizing, we have that if $s,t\in V(\HRG)\cap (B_{O}(\rho)\setminus B_{O}(\frac{R}{2}))$, then a.a.s.,\footnote{If $w\in\{s,t\}$ is in $V(\calG)\cap B_O(\frac{R}{2})$, the proof of Proposition~\ref{flow:prop:energy} establishes the claim but with $d(w)$ replaced by $\max\{d(w),\sqrt{n}\}$. However, this is an artifact of the simplified way we defined the flow $f_{s,t}$ in Section~\ref{ssec:flow} which sufficed for proving the aforementioned result.}
\[
\Res{s}{t} = 
\calO(d(s)^{-2(1-\alpha)}+d(t)^{-2(1-\alpha)}).
\]
The goal of this section is to strengthen this bound.

\smallskip
We start by giving a high level overview of the argument used to prove the upper bound, this relies on building a unit $(s,t)$-flow $f_{s,t}$ which yields an upper bound for $\Res{s}{t}$. Observe that if only $m$ of the edges incident to either $s$ or $t$ carry non-zero flow, then it is easy to see that $\Res{s}{t}=\Omega(1/m)$, or see the Nash-Williams inequality~~\cite[\S 2.5]{LyonsPeres}. On the other hand, the further away from the boundary (closer to the origin of $B_O(R)$) the larger (in expectation) number of neighbours both $s$ and~$t$ have (in fact, in the HRG, a constant fraction of the neighbours  of a vertex are within constant distance from the boundary of $B_O(R)$ and this number decreases exponentially fast as a function of the distance to the boundary). This suggests that a `good' flow should start by pushing flow towards the neighbours  of $s$ close to the boundary. However, the closer to the boundary the neighbours  are, the smaller the angle of the sector that contains them. In order for the flow to dissipate little energy it cannot traverse exceedingly large paths. However, in the HRG, geodesic paths tend to approach the origin (more so the larger the angle formed at the origin between the end vertices of the path). Unfortunately, the paths starting at vertices within a sector of small central angle merge as they approach the origin. So, there is a trade-off in designing low energy unit $(s,t)$-flows; route flow using as many neighbours  of vertices $s$ and $t$ versus  distribute flow in a sector of sufficiently large central angle. Our construction relies on using neighbours  of $w\in\{s,t\}$ whose radial distance is roughly $R-(2\alpha-1)(R-r_w)$. Flow is sent from neighbours  of $s$ towards the origin until it all reaches a single half-tile (we do the same for $t$ but in the reverse direction). Finally, we route the flow between the relevant pair of half-tiles using the adaptation of the flow of Section~\ref{ssec:flow} discussed in Remark~\ref{flow:remark:useful}.

\smallskip
Formally, let $\calF(c)$ be the tiling determined by $c>0$ (see Section~\ref{ssec:tiling}).
Assume $w\in\{s,t\}$ belongs to a tile of level $\ell_w$ of the tiling~$\calF(c)$ where $(1-\frac{1}{2\alpha})R\leq r_w<\rho$.
As usual, fix $\ell$ with respect to $\rho$ as in~\eqref{validflow:eqn:ldef}. In particular, recall that
$h_{\ell}\leq\rho$. 
Furthermore, let 
\begin{equation}\label{resist:eqn:deflPrime}
\ell'_w:=\min\big\{
\ell+1, 
\min\{ \ell'\in\NN \mid R-h_{\ell'} \leq (2\alpha-1)(R-r_w)\big\}.
\end{equation}
Since $h_i\to\infty$ as $i\to\infty$, it follows that $\ell'_w$ is well-defined. 
Moreover, we claim that $r_w\leq h_{\ell'_w}$. Indeed,
the claim is obvious if $\ell'_w=\ell+1$ since $r_w<\rho< h_{\ell+1}$.
If $\ell'_w\neq\ell+1$, then
  $R-h_{\ell'_w}\leq (2\alpha-1)(R-r_w)< R-r_w$, so again the claim holds.
A direct consequence of our previous observation is that $\ell_w\leq\ell'_w\leq\ell+1$.  
Next, define 
\begin{equation}\label{resist:eqn:defk}
  k_w:=
  \max\{ k\in\NN \mid k\leq\ell'_w \text{ and } \theta_R(h_{\ell'_w- k},h_{\ell'_w-k})\leq \theta_R(r_w,h_{\ell'_w})\}.
\end{equation} 
Note that $k_w$ is well-defined since
$r_w\leq h_{\ell'_w}$
and so, by Remark~\ref{prelim:rem:monotonicity}, we know that, 
$\theta_R(h_{\ell'_w},h_{\ell'_w})\leq\theta_R(r_w,h_{\ell'_w})$.
Also, observe that $k_w\leq\ell'_w-\ell_w$, since otherwise
$\theta_R(h_{\ell_w-1},h_{\ell_w-1})\leq\theta_R(r_w,h_{\ell'_w})$,
contradicting $r_w, h_{\ell'_w}\ge h_{\ell_w-1}$.
This and Remark~\ref{prelim:rem:monotonicity} imply that $\theta_R(h_{\ell_w-1},h_{\ell_w-1})>\theta_R(r_w,h_{\ell'_w})$.  

For future reference, we establish the following fact relating $k_w$, $\ell'_w$ and $r_w$.
\begin{fact}\label{resist:fct:kw}
If $(1-\frac{1}{2\alpha})R< r_w<\rho$,  then
$\displaystyle
2^{k_w}e^{R-h_{\ell'_w}}
=\Theta(e^{\frac12(R-r_w)}e^{\frac12(R-h_{\ell'_w})}).$
\end{fact}
\begin{proof}
First, we observe that $h_{\ell'_w}\geq\frac{R}{2\alpha}$.
Indeed, if $\ell'_w=\ell+1$, then $h_{\ell'_w}=\rho$ and the claim follows by our choice of $\rho$.
By hypothesis on $r_w$, if $\ell'_w\neq \ell+1$, then 
$R-h_{\ell'_w}\leq (2\alpha-1)(R-r_w)\leq (1-\frac{1}{2\alpha})R$, and the claim immediately follows.

Next, we argue that $k_w\neq\ell'_w$.
Otherwise, by definition of $k_\omega$, we would have $\pi=\theta(h_0,h_0)\leq\theta_R(r_w,h_{\ell'_w})$, implying that $r_w+h_{\ell'_w}\leq R$.
However, by hypothesis and the observation from the previous paragraph, $r_w+h_{\ell'_w}>R$, a contradiction. 

Now, by definition of $h_i$, we know that $\theta_R(h_i,h_i)=\frac12\theta_R(h_{i-1},h_{i-1})$ for all $i>1$.
Hence, 
$2^{k_w}\theta_{R}(h_{\ell'_w},h_{\ell'_w})=\theta_{R}(h_{\ell'_w-k_w},h_{\ell'_w-k_w})$.
By choice of $\ell'_w$ and $k_w$ 
we have $\theta_R(h_{\ell'_w-k_w},h_{\ell'_w-k_w})\leq\theta_R(r_w,h_{\ell'_w})
<\theta_R(h_{\ell'_w-k_w-1},h_{\ell'_w-k_w-1})$. 
 Lemma~\ref{prelim:lem:angles} and Claim~\ref{flow:claim-h} yield
 that $\theta_R(h_{\ell'_w-k_w},h_{\ell'_w-k_w})=\Theta(\theta_R(h_{\ell'_w-k_w-1},h_{\ell'_w-k_w-1}))$
 and we get $\theta_R(h_{\ell'_w-k_w},h_{\ell'_w-k_w})=\Theta(\theta_R(r_w,h_{\ell'_w}))$.
 Summarizing, $2^{k_w}\theta_R(h_{\ell'_w},h_{\ell'_w})=\Theta(\theta_R(r_w,h_{\ell'_w}))$.
The result follows applying Lemma~\ref{prelim:lem:angles} twice. 
\end{proof}

Now, let $T_w$ be the tile at level $\ell'_w-k_w$ that intersects the ray starting at the origin $O$ and passing through the point where $w$ is located. Do not confuse $T_w$ with the tile containing~$w$ (which in previous sections was denoted $T(w)$), they could be the same tile, but typically are not.
Also, let $H_w$ be the half-tile of $T_w$ intersecting the ray starting at the origin $O$ and passing through $w$.
Our next result shows that all level $\ell'_w$~descendants of~$T_w$ are contained in $B_w(R)$.
\begin{lemma}
If $(1-
\frac{1}{2\alpha})R\leq r_w<\rho$,  then every tile in $\pi^{-k_w}(\{T_w\})$ is a level $\ell'_w$ tile contained in $B_w(R)$.
\end{lemma}
\begin{proof}
Consider an arbitrary tile $T\in\pi^{-k_w}(\{T_w\})$.
Since $T$ is a $k_w$-th generation descendant of the level~$\ell'_w-k_w$ tile $T_w$, it follows that $T$ is at level $(\ell'_w-k_w)+k_w=\ell'_w$. This establishes the first part of the lemma. To prove the second part, 
let $\Upsilon_{T_w}$ be the smallest sector containing~$T_w$. Thus, the central angle of $\Upsilon_{T_w}$ is $\theta_{\ell'_w-k_w}\leq\theta_R(h_{\ell'_w-k_w},h_{\ell'_w-k_w})$. 
Consider the sector $\Upsilon_w$ whose central angle is double that of~$\Upsilon_{T_w}$, but whose bisector is the ray starting from the origin~$O$ and passing through the point where vertex $w$ was added.
Note that $\Upsilon_{T_w}\subseteq\Upsilon_w$ and thus~$T\subseteq\Upsilon_w\cap B_{O}(h_{\ell'_w})$.
Hence, every point $q\in T$ is within an angle at most 
$\theta_R(h_{\ell'_w-k_w},h_{\ell'_w-k_w})\leq\theta_{R}(r_w,h_{\ell'_w})$ of the bisector of $\Upsilon_w$.
By Remark~\ref{prelim:rem:monotonicity}, since $r_q<h_{\ell'_w}$ (because $q$ belongs to the level~$\ell'_w$ tile $T$), we know that $\theta_R(r_w,h_{\ell'_w})\leq\theta_R(r_w,r_q)$ implying that $q\in B_w(R)$ as claimed.
\end{proof}

We now construct a $(\{w\},V\cap H_w)$-flow $f_{w,V\cap H_w}$ in $\calG=(V,E)$.
The idea of the construction is to distribute one unit of flow from $w$ among all vertices located in a half-tile 
$H$ which is both a descendant of $H_w$
and belongs 
 to 
$\pi^{-k_w}(\{T_w\})$, so that each such half-tile $H$
receives $1/2^{k_w}$ units of flow. The aim of this is to spread the flow evenly among the vertices
in $V\cap H$, note however that since distinct half-tiles might contain distinct number of vertices,  vertices belonging to distinct half-tiles might receive different amounts of flow.
Subsequently, we initially set $s:=0$ and for every pair of twin half-tiles which are descendants of $H_w$ and whose union is a tile~$T\in \pi^{-(k_w-s)}(\{T_w\})$, we re-distribute (re-balance) flow among all vertices in $V\cap T$ so each one ends up with exactly the same amount of flow.
We then spread all the flow at a vertex in~$V\cap T$  evenly among its neighbours in the parent half-tile of $T$ (by construction of the tiling, every vertex in the parent half-tile of $T$ is a neighbour of vertices in $V\cap T$).
Provided~$s<k_w$, the  process is repeated,  but replacing $s$ by $s+1$ until $s$ becomes $k_w$.
Formally, for every half-tile $H$ which is both a descendant of $H_w$ and belongs to a tile in $\pi^{-k_w}(\{T_w\})$,
\[
f_{w,H_w}(\vec{wv}) := \frac{1/2^{k_w}}{|V\cap H|}\qquad
\text{for all $v\in V\cap H$.}
\]
Moreover, for every tile $T$ which is a descendant of $T_w$ whose level is $\ell'_w-k_w+s$ where $0\leq s< k_w$,  
if $H$ is a half-tile of $T$ such that $|V\cap H|<|V\cap (T\setminus H)|$, then
\[
f_{w,H_w}(\vec{uv}) := 
\Big(\frac{1/2^{s}}{|V\cap H|}-\frac{1/2^{s}}{|V\cap(T\setminus H)|}\Big)\cdot\frac{1}{|V\cap T|}
\qquad \text{if $u\in V\cap H$ and $v\in V\cap (T\setminus H)$}. 
\]
Finally, if $T$ is a descendant of tile $T_w$ whose level is $\ell'_w-k_w+s$ where $1\leq s\leq k_w$ and $H$ is the parent half-tile of~$T$, then
\[
f_{w,H_w}(\vec{vw}) := \frac{1/2^{s-1}}{|V\cap T|}\cdot\frac{1}{|V\cap H|}
\qquad
\text{for all $v\in V\cap T$ and $w\in V\cap H$.}
\]
The mapping $f_{w,H_w}$ is extended to all of $\vec{E}(\calG)$ so it satisfies the antisymmetric property of flows, i.e. $f_{w,H_w}(\vec{uv})=-f_{w,H_w}(\vec{vu})$, and so it is null for all remaining oriented edges $\vec{uv}\in \vec{E}(\calG)$ where $f_{w,H_w}(\vec{uv})$ remains undefined.

The following result establishes several of the properties that hold for $f_{w,H_w}$. 
\begin{lemma}\label{resist:lem:energyBnd}
Let $\epsilon, c:=c(\epsilon)$ be as in Claim~\ref{flow:claim:measure} and consider the tilling $\calF(c)$.
If none of the tiles of $\calF(c)$ intersecting $B_O(\rho)$ is faulty and a vertex $w$ is added to $\HRG:=\HRG_{\alpha,\nu}(n)$ at a point located within $B_O(\rho)$, then $f_{w,H_w}$ is an $\calF(c)$-compatible balanced unit $(\{w\},V\cap H_w)$-flow in $\HRG$.
Moreover,
\[
\calE(f_{w,H_w}) = \calO\big(\min\big\{e^{-2\alpha(1-\alpha)(R-r_w)},\tfrac{1}{\ln n}\,e^{-(1-\alpha)(R-r_w)}\big\}\big).
\]
\end{lemma}
\begin{proof}
The proof that $f_{w,H_w}$ is an $\calF(c)$-compatible balanced unit flow is similar to the proof of Lemma~\ref{flow:lem:fs0} where the same property was shown but for another flow and hence we omit it.
To show the bound on the energy dissipated by the flow, denote by $\pi^{-p}(\{H_w\})$ the collection  of $p$-th generation half-tiles descendants of $H_w$.
Also, abbreviate half-tile parent as \emph{h-t.prnt.}
Say $\{H,H'\}$ is a \emph{half-tile partition} of tile $T$ (abbreviated \emph{h-t.part.~of $T$}) if $H, H'$ are twin half-tiles of $T$.
By definition of $f_{w,H_w}$, 
\begin{equation}\label{resist:eqn:energyUBnd} 
 \begin{aligned}   \calE(f_{w,H_w}) 
    & = \sum_{H\in\pi^{-k_w}(\{H_w\})}\frac{(1/2^{k_w})^2}{|V\cap H|} + \sum_{p=0}^{k_w-1}\, \sum_{\substack{H\in \pi^{-p}(\{H_w\}) \\\text{h-t.prnt.~of $T$}}}\frac{(1/2^{p})^2}{|V\cap T|}\cdot\frac{1}{|V\cap H|}\\[3pt]
   & \qquad + \sum_{p=0}^{k_w-1} \,
    \sum_{\substack{\text{$\{H,H'\}$ h-t.part.~of $T$}\\ \text{s.t.~$H, H'\in\pi^{-{p+1}}(\{H_w\})$}}}
    \Big(\frac{1/2^{p}}{|V\cap H|}-\frac{1/2^{p}}{|V\cap H')|}\Big)^2\cdot\frac{|V\cap H|\cdot|V\cap H'|}{|V\cap T|^2}. 
\end{aligned}\end{equation}
In order to bound each of the summations above recall that 
since $w\in B_O(\rho)$, it holds that~$\ell'_w\leq\ell+1$.

To upper bound the first summation in~\eqref{resist:eqn:energyUBnd}, observe that the $2^{k_w}$ half-tiles in $\pi^{-k_w}(\{H_w\})$ intersect $B_O(\rho)$ so, by hypothesis, they are not sparse. 
Using Claim~\ref{flow:claim:measure}, it follows that the first summation is $\calO(2^{-k_w}e^{-(1-\alpha)(R-h_{\ell'_w})})$.

Consider next the second summation in~\eqref{resist:eqn:energyUBnd}.
Observe that there are $2^p$ distinct tiles $T$ whose half-tile parent $H$ belongs to $\pi^{-p}(\{H_w\})$  and that all of them are at level $\ell'_w-k_w+p+1$, hence again each of them intersects $B_O(\rho)$.
Thus, neither $T$ is faulty nor $H$ is sparse (by hypothesis).
Using again Claim~\ref{flow:claim:measure}
and the fact that $h_{\ell'_w-k_w+p}=h_{\ell'_w}-(k_w+p)\ln 2+\Theta(1)$ (see Claim~\ref{flow:claim-h}), it follows that the 
second summation in~\eqref{resist:eqn:energyUBnd} is 
\begin{align*}
\Theta\Big(\sum_{p=1}^{k_w}2^{-p}e^{-2(1-\alpha)(R-h_{\ell'_w-k_w+p})}\Big)
& = \Theta\Big((2^{k_w}e^{R-h_{\ell'_w}})^{-2(1-\alpha)}\sum_{p=1}^{k_w}2^{-2(\alpha-\frac12)p}\Big) \\
& = \calO\big((2^{k_w}e^{R-h_{\ell'_w}})^{-2(1-\alpha)}\big).
\end{align*}
Finally, consider the third summation in~\eqref{resist:eqn:energyUBnd}.
As for the previous two summations, all the tiles and half-tiles involved intersect $B_O(\rho)$, so they are non-faulty (by hypothesis).
Applying Claim~\ref{flow:claim:measure} as in the previous cases, we get that this last summation is of order at most  the second one.
Summarizing,
\[
\calE(f_{w,H_w}) = \calO\big(e^{\alpha(R-h_{\ell'_w})}/(2^{k_w}e^{R-h_{\ell'_w}})
+(2^{k_w}e^{R-h_{\ell'_w}})^{-2(1-\alpha)}\big).
\]
Fact~\ref{resist:fct:kw} and the previous displayed equation imply that
\[
\calE(f_{w,H_w})=\calO\big(e^{(\alpha-\frac12)(R-h_{\ell'_w})}/e^{\frac12(R-r_w)}+
(e^{R-h_{\ell'_w}}e^{R-r_w})^{-(1-\alpha)}\big).
\]
To conclude, note that by the definition of $\ell'_w$ 
(see~\eqref{resist:eqn:deflPrime}) and our estimates for the $h_i$'s (see Claim~\ref{flow:claim-h}), 
either $R-h_{\ell'_w}=(2\alpha-1)(R-r_w)+\Theta(1)$
or $h_{\ell'_w}=\rho+\Theta(1)$.
Replacing in the last displayed equation, for the former case, we get 
$\calE(f_{w,H_w})=\calO(e^{-2\alpha(1-\alpha)(R-r_w)})$ and for the latter
$\calE(f_{w,H_w})=\calO(\tfrac{1}{\ln n}\,e^{-(1-\alpha)(R-r_w)}\big)$.
\end{proof}

The following result gives the desired upper bound.
\begin{proposition}\label{ap:prop:upperBnd}
    Let $\epsilon, c:=c(\epsilon)$ be as in Claim~\ref{flow:claim:measure}. Then, for every $d>0$ there is a sufficiently large $C>0$ such that for a pair of vertices $s,t$ added to $\HRG:=\HRG_{\alpha,\nu}(n)$ at points located within $B_O(\rho)\setminus B_O((1-\frac{1}{2\alpha})R)$ where $\rho:=\rho(C)$ is as in~\eqref{flow:eqn:rho}, 
    with probability $1-o(1/n^d)$, 
    if
    both $R-r_s$ and $R-r_t$ are at least 
    $(R-\rho)/(2\alpha-1)$, then
    \[
\Res{s}{t}
=\calO\Big(\frac{1}{d(s)^{4\alpha(1-\alpha)}}+\frac{1}{d(t)^{4\alpha(1-\alpha)}}\Big).
    \]
    If $w\in\{s,t\}$ is such that 
    $w\in B_O(\rho)$ and $(2\alpha-1)(R-r_w)<R-\rho$, 
    then the claim holds but with $1/d(w)^{4\alpha(1-\alpha)}$ above replaced by $\frac{1}{\ln n} \cdot 1/d(w)^{2\alpha}$. 
\end{proposition}
\begin{proof}
By definition of effective resistance, it suffices to show that with probability $1-o(1/n^d)$ there is an $f_{s,t}$ unit $\calF(c)$-compatible $(s,t)$-flow in $\HRG$ such that $\calE(f_{s,t})=\calO(d(s)^{-4\alpha(1-\alpha)}+d(t)^{-4\alpha(1-\alpha)})$.

By Lemma~\ref{flow:lem:nonfaultyTiles}, there is a sufficiently large constant $C>0$ such that for $\rho:=\rho(C)$ as defined in~\eqref{flow:eqn:rho}, with probability $1-o(1/n^d)$, every tile in $\calF(c)$ intersecting $B_O(\rho)$ is non-faulty.
By hypothesis, if $w\in\{s,t\}$, then $w\in B_O(\rho)$.
Thus, there is a $f_{w,H_w}$ flow in $\calG$ satisfying 
the properties guaranteed by Lemma~\ref{resist:lem:energyBnd}.

By Remark \ref{flow:remark:useful} (following Proposition~\ref{flow:prop:energy}) and since $H_w$ intersects $B_O(\rho)$ and is at level $\ell'_w-k_w$,  there is an $\calF(c)$-compatible balanced unit $(H_s,H_t)$-flow $f_{H_s,H_t}$ in $\calG$ such that
\[
\calE(f_{H_s,H_t}) = \calO(e^{-2(1-\alpha)(R-h_{\ell'_s-k_s})}+e^{-2(1-\alpha)(R-h_{\ell'_t-k_t})}).
\]
Assuming $k_w=\ell'_w$,
since $R-h_0=\frac{R}{2}>\alpha(R-r_w)$ if and only if $r_w>(1-\frac{1}{2\alpha})R$, by hypothesis,  
\[
e^{-2(1-\alpha)(R-h_{\ell'_w-k_w})}=e^{-2(1-\alpha)(R-h_{0})}
< e^{-2\alpha(1-\alpha)(R-r_w)}.
\]
Assuming $k_w<\ell'_w$, then
by Claim~\ref{flow:claim-h}, we have $h_{\ell'_w-k_w}=h_{\ell'_w}-k_w\ln 2+\Theta(1)$ 
and thus, 
using Fact~\ref{resist:fct:kw}, 
\[
e^{-2(1-\alpha)(R-h_{\ell'_w-k_w})}
=\Theta\big((2^{k_w}e^{R-h_{\ell'_w}})^{-2(1-\alpha)}\Big)
= \Theta\big((e^{R-r_w}e^{R-h_{\ell'_w}})^{-(1-\alpha)}\big).
\]
Recalling that $R-h_{\ell'_w}=(2\alpha-1)(R-r_w)+\Theta(1)$ or 
$h_{\ell'_w}=\rho+\Theta(1)$.
In the former case (the other case is handled similarly and left to the reader) we get that
\[
e^{-2(1-\alpha)(R-h_{\ell'_w-k_w})}
= \Theta\big(e^{-2\alpha(1-\alpha)(R-r_w)}\big).
\]
Summarizing, 
\begin{equation}\label{resist:eqn:energyUBnd2}
\calE(f_{H_s,H_t})=\calO(e^{-2\alpha(1-\alpha)(R-r_s)}+e^{-2\alpha(1-\alpha)(R-r_t)}).
\end{equation}
To conclude, let $f'_{H_t,t}$ be the flow in $\calG$ obtained from $f_{t,H_t}$ by reversing the flows along all edges of $\calG$, that is, $f'_{H_t,t}(\vec{uv})=-f_{t,H_t}(\vec{uv})$ for all $\vec{uv}\in\vec{E}(\calG)$.
It is easy to check that $f'_{H_t,t}$ is indeed a $(V\cap H_t,\{t\})$-flow, actually an $\calF(c)$-compatible balanced unit flow.
By Claim~\ref{prelim:clm:concatFlows}, we get that
$f_{s,t}:=f_{s,H_s}+f_{H_s,H_w}+f'_{H_t,t}$ is 
a unit $(s,t)$-flow in $\calG$ such that
\[
\calE(f_{s,t}) = \calO\big(\calE(f_{s,H_s})+\calE(f_{H_s,H_t})+\calE(f'_{H_t,t})\big)
= \calO(e^{-2\alpha(1-\alpha)(R-r_s)}+e^{-2\alpha(1-\alpha)(R-r_t)}).
\]
where the last equality is a consequence of~\eqref{resist:eqn:energyUBnd2}
and Lemma~\ref{resist:lem:energyBnd}.
The desired conclusion follows from the well known fact (which is an immediate consequence of Lemma~\ref{prelim:lem:ballsMeasure}) that the degree $d(w)$ of every vertex $w\in V(\HRG)\cap B_O(\rho)$ is, w.h.p.,~$\Theta(e^{\frac12(R-r_w)})$. 
\end{proof}
The preceding result together with the commute time identity (see Lemma~\ref{lem:commutetime}) and
the fact that the number of edges of the center component of $\calG_{\alpha,\nu}(n)$ is $\Theta(n)$ (see Lemma~\ref{prelim:lem:volumeCenterComp}) yield the  upper bound of Theorem~\ref{ap:thm:commut}.

\subsection{Lower bound}
As in the previous section, let $\HRG:=\HRG_{\alpha,\nu}(n)$, consider $s, t\in V(\calG)\cap B_O(R)$ and let $\rho$ be as in~\eqref{flow:eqn:rho}. The goal of this section is to obtain a lower bound for effective resistances matching, within polylogarithmic factors, the upper bound of Proposition~\ref{ap:prop:upperBnd}. 

One classical approach to derive a lower bound on $\Res{s}{t}$ is 
to observe that the cardinality of a subset $F$ of the edges of $\HRG$ whose
removal leaves $s$ and $t$ in separate connected components gives 
a lower bound for $\Res{s}{t}$. 
Indeed, the energy dissipated by any unit $(s,t)$-flow must be at least $1/|F|$, so by definition of effective resistance $\Res{s}{t}\geq 1/|F|$, see the Nash-Williams inequality~\cite[\S 2.5]{LyonsPeres}.
Clearly, the smaller the size of $F$, the better the lower bound.
To achieve this section's goal, we thus focus on finding small $(s,t)$-edge-cuts in $\HRG$.
We start by identifying small edge-cuts that would isolate a vertex $v$ added to $\HRG$ at a point $p\in B_O(R)$ from some vertex in the center component of $\HRG$. 

Throughout our ensuing discussion we assume the radial coordinate of~$p$ is $r_p=r$.
Also, we let $\omega:=\omega(n)$ be a positive function going to infinity arbitrarily slowly and
define~$\phi_r:=2\pi\nu e^{-\omega}/(n\mu(B_O(r)))$.

\begin{claim}\label{ap:clm:phi}
If $(1-\frac{1}{2\alpha})R\leq r<R$, then
  $\phi_r=2\pi(1+o(1))(\nu/n)e^{\alpha(R-r)-\omega}=\calO(e^{-\omega})$.
\end{claim}
\begin{proof}
  Direct consequence of our choice of $\phi_r$, the hypothesis on $r$ and Lemma~\ref{prelim:lem:ballsMeasure}.
\end{proof}
Let $\Upsilon_p$ be the sector of central angle $\phi_r$ on whose bisector $p$ lies.
Let $V:=V(\calG)$ and $v$ be a vertex added at $p$.
We are interested in determining a good upper bound (that holds a.a.s.) for
the edge-cut in $\calG$ induced by $(\{v\}\cup V)\cap\Upsilon_p$, that is,
if for a graph $G$ and a set of its nodes $S\subseteq V(G)$ we define
$\delta_G(S):=\{uv\in E(G) \mid u\in S, v\in V(G)\setminus S\}$, then we seek a
good upper bound for $|\delta_{\calG}((\{v\}\cup V)\cap\Upsilon_p)|$.
The following result reduces (under some mild conditions) the task of bounding
the edge-cut induced by $(\{v\}\cup V)\cap\Upsilon_p$ to the one of bounding
the edge-cut induced by $V\cap\Upsilon_p$.
\begin{lemma}\label{ap:lem:cutEasy} 
If $p\in B_O(R)$ is such that $(1-\frac{1}{2\alpha})R<r_p\leq R-\omega/(1-\alpha)(\alpha-\frac12)$,
  then a.a.s.,
  \[
  |\delta_{\calG}((\{v\}\cup V)\cap\Upsilon_p)|
  = |\delta_{\calG}(V\cap\Upsilon_p)|+\calO(e^{2\alpha(1-\alpha)(R-r)}).
  \]
\end{lemma}
\begin{proof}
Let $r:= r_p$.
  By Lemma~\ref{prelim:lem:angles} and since $R=2\ln(n/\nu)$,
  \[
  \theta_R(R,r)
  = 2(1+o(1))e^{-\frac12 r}
  = 2(\nu/n)(1+o(1))e^{\frac12(R-r)}.
  \]
  By hypothesis on $r$ we have $\frac12(R-r)\leq \alpha(R-r)-\omega/(1-\alpha)$ which together with Lemma~\ref{prelim:lem:angles} implies
  that $\theta_R(R,r)= o(\phi_r)$.
  On the other hand, by definition of $\theta_R(\cdot,\cdot)$, we
  have $\theta_R(R-r,r)=\pi\geq\phi_r$.
  Thus, by monotonicity and continuity of the mapping $x\mapsto\theta_R(x,r)$
  (see Remark~\ref{prelim:rem:monotonicity}),
  there is a unique $R-r\le x\le R$ such that $\theta_R(x,r)=\phi_r$.

We claim that any
vertex $u\in V\setminus B_O(x)$ at distance at most $R$ from $v$
must necessarily be in~$\Upsilon_p$.  Indeed,
since $\theta_R(r,\cdot)$ is non-decreasing
(again by Remark~\ref{prelim:rem:monotonicity}) we have
$\theta_R(r_p,r_u)=\theta_R(r_v,r_u)\leq\theta_R(r,x)=\phi_r$
and thus $u$ is in
$\Upsilon_p$ (actually, in its closure).
Hence, if we denote by $E_{\calG}(S,T):=\{uv\in E(\calG) \mid u\in S, v\in T\}$ the set of edges of $\calG$ with one
  end-vertex in $S$ and the other in $T$,   then we have
\[
  |\delta_{\calG}((\{v\}\cup V)\cap\Upsilon_p)|
  \leq |\delta_{\calG}(V\cap\Upsilon_p)| + |E_{\calG}(\{v\},V\cap B_O(x))|.
\]

Since $x+r\ge R$, by Lemma~\ref{prelim:lem:angles},
  we have $n\theta_R(r,x)=\Theta(e^{\frac12(R-x)+\frac12(R-r)})$.
On the other hand, by Claim~\ref{ap:clm:phi},
  $n\phi_r=\Theta(e^{\alpha(R-r)-\omega})$.
So $\theta_R(r,x)=\phi_r$ implies that 
$R-x=(2\alpha-1)(R-r)+2\omega+\Theta(1)$.
By Lemma~\ref{prelim:lem:ballsMeasure},
since $|E_{\calG}(\{v\},V\cap B_O(x))|$
equals the number of vertices in $B_O(x)$ at distance at most $R$
from the location of $v$, the expectation of
$|E_{\calG}(\{v\},V\cap B_O(x))|$ is 
\[
n\mu(B_p(R)\cap B_O(x))= \Theta(e^{\frac12(R-r)-(\alpha-\frac12)(R-x)})
  =\Theta(e^{2\alpha(1-\alpha)(R-r)-(2\alpha-1)\omega})=\omega(1).
\]
By hypothesis on $r$ we have $2\alpha(1-\alpha)(R-r)\geq 2\alpha\omega/(\alpha-\frac12)\geq 2\alpha\omega
$ so the probability that
$|E_{\calG}(\{v\},V\cap B_O(x))|$ exceeds its expectation by a factor
  of at least $\delta:=e^{\frac32}$ is at most $e^{-\delta\omega/2}=o(1)$.
\end{proof}
Motivated by the previous result, we focus on bounding $|\delta_{\calG}(V\cap\Upsilon_p)|$ from above. 

Let $\calA_a$ be the annulus of points of $B_O(R)$
at radial distance between $a$ and $a+1$ of the origin,
that is, $\calA_a:=B_O(R)\cap (B_O(a+1)\setminus B_O(a))$.
Also, let $\calO_{a,b}$ be the collection of points in $\calA_a$ outside 
$\Upsilon_p$ (that is, in $B_O(R)\setminus\Upsilon_p$)
that are at distance at most $R$
from a point belonging to $\calA_b\cap B_O(R)\cap\Upsilon_p$.
Analogously, let $\calI_{a,b}$ be the collection of points in $\calA_b$
inside $B_O(R)\cap\Upsilon_p$ that are at distance at most $R$
from any point belonging to $\calA_a\cap B_O(R)\setminus\Upsilon_p$.
Clearly,
\[
\delta_{\calG}(V\cap\Upsilon_p)
= \bigcup_{a=0}^{\lfloor R\rfloor}\bigcup^{\lfloor R\rfloor}_{b=0}
E_{\calG}(V\cap\calO_{a,b},V\cap\calI_{a,b})neighbour{.}
\]
Let $b_{min}=\lfloor r\rfloor$.   
Observe that the expected number of vertices of $\calG$
  that are located
  in $\Upsilon_p\cap B_O(b_{min})$ is
\[
n\mu(\Upsilon_p\cap B_O(b_{min}))
=\frac{\phi_r}{2\pi}n\mu(B_O(b_{min}))
=\frac{\phi_r}{2\pi}n\mu(B_O(r))(1+o(1))e^{\alpha(b_{min}-r)}=\calO(e^{-\omega})=o(1)
\]
where the second equality is due to Lemma~\ref{prelim:lem:ballsMeasure} and
the third equality is by Claim~\ref{ap:clm:phi} and since
$b_{min}\leq r$.
As the expected number of points in $\Upsilon_p\cap B_O(b_{min})$ is $o(1)$, Markov's inequality gives us that a.a.s.~$V\cap\Upsilon_p\cap B_O(b_{min})=\emptyset$.
  Similarly, for $a_{min}=\lceil (1-\frac{1}{2\alpha})R-\omega\rceil$
  it also holds that a.a.s.~$V\cap B_O(a_{min})=\emptyset$.
From our ongoing discussion, we get that a.a.s.,
\[
\delta_{\calG}(V\cap\Upsilon_p) = \bigcup_{a=a_{min}}^{\lfloor R\rfloor}\bigcup^{\lfloor R\rfloor}_{b=b_{min}} E_{\calG}(V\cap\calO_{a,b},V\cap\calI_{a,b}),
\]
so if we let $O_{a,b}=|V\cap\calO_{a,b}|$ and $I_{a,b}=|V\cap\calI_{a,b}|$, then a.a.s.,
\begin{equation}\label{ap:eqn:sum}
|\delta_{\calG}(V\cap\Upsilon_p)| \leq \sum_{a=a_{min}}^{\lfloor R\rfloor}\sum_{b=b_{min}}^{\lfloor R\rfloor} O_{a,b}\cdot I_{a,b}.
\end{equation}
To upper bound the right-hand side of \eqref{ap:eqn:sum}, we will need to estimate the expected value of $O_{a,b}$ and $I_{a,b}$, or equivalently, compute $\mu(\calO_{a,b})$
and $\mu(\calI_{a,b})$. We claim that if $a+b\geq R$, then
\[
\mu(\calO_{a,b})=\Theta\big(\theta_R(a,b)\cdot \mu(B_O(a))\big)
\qquad\text{and}\qquad
\mu(\calI_{a,b})=\Theta\big(\min\{\phi_r,2\theta_R(a,b)\}\cdot \mu(B_O(b))\big).
\]
Indeed, observe that $\calO_{a,b}$ is contained in the union of two sectors (each one of central angle~$\theta_R(a,b)$). By  Lemma~\ref{prelim:lem:ballsMeasure} we have $\mu(\calA_{a})=\Theta(B_{O}(a))$.
The claimed formula for $\mu(\calO_{a,b})$ follows. 
A similar argument holds for $\mu(\calI_{a,b})$ but taking into account that $\calI_{a,b}$ is contained in the sector $\Upsilon_p$ of central angle $\phi_r$.

To bound the summation in~\eqref{ap:eqn:sum}, we will use the following technical claim which allows us to ignore summation terms corresponding to pairs $(a,b)$ satisfying a specific constraint.

\begin{claim}\label{ap:clm:emptyRgn}
  Let $\omega'=\omega'(n)$ be a positive function that goes to infinity with $n$ and let $\varphi:=\varphi(n)=2\ln(\omega' R)$
  be such that $\varphi\ge 2\alpha\omega$.
  With probability $\calO(1/\omega')$, there is pair $(c,d)$ such that 
  $a_{min}\leq c\leq R$, $b_{min}\leq d\leq R$, 
  and $R-d\leq (2\alpha-1)(R-c)-\varphi$, for which either $I_{c,d}\neq 0$
  or~$O_{c,d}\neq 0$.
\end{claim}
\begin{proof}
Since $c\geq a_{min}\geq (1-\frac{1}{2\alpha})R-\omega$ and by hypothesis
$2(1-\alpha)R\leq c+d-2\alpha c-\varphi\leq c+d-(2\alpha-1)R$,
we have that $c+d\ge R$.
Thus, 
\[
n\max\{\mu(\calO_{c,d}),\mu(\calI_{c,d})\}
=\Theta(n\theta_R(c,d)\mu(B_O(d))
= \Theta(e^{\frac12(R-d)-(\alpha-\frac12)(R-c)}).
\]
Since $\EE(O_{c,d})=n\mu(\calO_{c,d})$
(respectively, $\EE(\calI_{c,d})=n\mu(\calI_{c,d})$), by Markov's inequality, the probability that $O_{c,d}\neq 0$ (respectively, $I_{c,d}\neq 0$) is $\calO(e^{\frac12(R-d)-(\alpha-\frac12)(R-c)})$.
By a union bound,
summing over $R-(2\alpha-1)(R-c)+\varphi\leq d\leq R$ and then over the
at most $\lceil R\rceil=\calO(\ln n)$ possible values taken by $c$ we
get that with probability $\calO(e^{-\frac12\varphi}\cdot\ln n)$ there
exists a pair $(c,d)$ such that $O_{c,d}\neq 0$ (respectively, $I_{c,d}\neq 0$).
The claim follows by our choice of $\varphi$.
\end{proof}

Finally, we can bound the size of the edge-cut induced 
by $(\{v\}\cup V(\HRG))\cap\Upsilon_p$.  
\begin{lemma}\label{ap:lem:edgecutVertex}
Suppose $(1-\frac{1}{2\alpha})R<r<R-\omega/(1-\alpha)(\alpha-\frac12)$.
Then, a.a.s.~it holds that
  \[
  |\delta_{\calG}((\{v\}\cup V)\cap\Upsilon_p)|=\calO(e^{2\alpha(1-\alpha)(R-r)}\cdot
  (\ln n)^{3}\cdot(\ln\ln n)^2).
  \]
\end{lemma}
\begin{proof}
By Lemma~\ref{ap:lem:cutEasy}, it suffices to upper bound
  $|\delta_{\HRG}(V\cap\Upsilon_p)|$.
Let $\varphi$ and $\omega'$ be as in Claim~\ref{ap:clm:emptyRgn}.
Henceforth, let $\Omega$ be the set of pairs $(a,b)$ such that
$a_{min}\le a\leq \lfloor R\rfloor$, $b_{min}\le b\le \lfloor R\rfloor$,
$R-b\ge (2\alpha-1)(R-a)-\varphi$, and
$R-a\ge (2\alpha-1)(R-b)-\varphi$.

We first bound the sum of $O_{a,b}\cdot I_{a,b}$
over the pairs $(a,b)\in\Omega$ such that $2\theta_R(a,b)\leq\phi_r$.
By Claim~\ref{ap:clm:phi} we know that $\phi_r=o(1)$,
  so by Lemma~\ref{prelim:lem:angles},
if $2\theta_R(a,b)\leq \phi_r$, then $a+b\ge R$ and
$\theta_R(a,b)=\Theta(e^{\frac12(R-a-b)})$.
Moreover, $R-d\geq (2\alpha-1)(R-c)-\varphi$ when $(c,d)\in\Omega$, so
\begin{align*}
  n\mu(\calO_{a,b})
  & =  \Theta(e^{-(\alpha-\frac12)(R-a)+\frac12(R-b)})
  =  \Omega(e^{-\frac12\varphi}),
  \\
  n\mu(\calI_{a,b}) & = \Theta(e^{-(\alpha-\frac12)(R-b)+\frac12(R-a)})
    =  \Omega(e^{-\frac12\varphi}).
\end{align*}
By Lemma~\ref{prelim:lem:devBnd} and a union bound over the at most
$\lceil R\rceil^2$ pairs $(a,b)\in\Omega$ such that $2\theta_R(a,b)\le\phi_r$,
the probability that some such $O_{a,b}$ (respectively, $I_{a,b}$)
exceeds its expectation by a factor $e^{\frac12\varphi}\cdot 2\ln(\omega' R^2)$ is at most $\calO(1/\omega')$.
Hence, a.a.s.,
\begin{align*}
\sum_{\substack{(a,b)\in\Omega \text{ s.t. } \\ 2\theta_R(a,b)\leq\phi_r}}
\!\!\! O_{a,b}\cdot I_{a,b}
& \leq
4e^{\varphi}(\ln(\omega' R^2))^2
\sum_{\substack{(a,b)\in\Omega \text{ s.t. } \\ 2\theta_R(a,b)\leq\phi_r}}
\!\!\! n^2\mu(\calO_{a,b})\mu(\calI_{a,b})
\\ & = e^{\varphi}(\ln(\omega' R^2))^2
\sum_{\substack{(a,b)\in\Omega \text{ s.t. } \\ 2\theta_R(a,b)\leq\phi_r}}
\!\!\! \calO(e^{(1-\alpha)(2R-a-b)}).
\end{align*}
A direct consequence of Claim~\ref{ap:clm:phi} is that
\[
2\theta_R(a,b)\leq \phi_r \text{ if and only if }
R-\tfrac12(a+b)\le \alpha(R-r)-\omega+\Theta(1).
\]
Hence, summing over the pairs $(a,b)\in\Omega$
such that $a+b=s$ and then over the valid
values of $s\geq 2R-2\alpha(R-r)+2\omega+\Theta(1)$
(observing that there are at most
$\lfloor R\rfloor=\calO(\ln n)$ pairs
$(a,b)\in\Omega$ for which $a+b=s$ for a given $s$), we get
\[
\sum_{\substack{(a,b)\in\Omega \text{ s.t. } \\ 2\theta_R(a,b)\leq\phi_r}}
O_{a,b}\cdot I_{a,b}
=
\calO(e^{2\alpha(1-\alpha)(R-r)-2(1-\alpha)\omega}\cdot\ln n\cdot e^{\varphi}(\ln(\omega' R^2))^2).
\]

Next, we bound the sum of $O_{a,b}\cdot I_{a,b}$ over the pairs $(a,b)\in\Omega$
such that $2\theta_R(a,b)>\phi_r$.
Observe that,
\begin{align*}
  n\mu(\calO_{a,b}) & 
  = \begin{cases}
    \Theta(e^{\frac{R}{2}-\alpha(R-a)}), & \text{if $a+b\le R$,} \\
    \Theta(e^{-(\alpha-\frac12)(R-a)+\frac12(R-b)}), & \text{if $a+b> R$.}
    \end{cases}
\end{align*}
Since $a\geq a_{min}$ we have
$\frac{R}{2}-\alpha(R-a)\geq \frac{R}{2}-\alpha(R-a_{min})$.
Moreover, if $a+b\leq R$, then 
  $-(\alpha-\frac12)(R-a)+\frac12(R-b)\geq
\frac{R}{2}-\alpha(R-a_{min})$.
Since $a_{min}\geq (1-\frac{1}{2\alpha})R-\omega$, we conclude that
$n\mu(\calO_{a,b})=\Omega(e^{-\alpha\omega})$ independent of whether
$a+b$ is smaller or larger than $R$.

Since by assumption $2\theta_R(a,b)>\phi_r$, from
  Claim~\ref{ap:clm:phi}, we get 
\[
  n\mu(\calI_{a,b}) = \Theta(n\phi_r\mu(B_O(b+1)))
  = \Theta(e^{\alpha(R-r)-\alpha(R-b)-\omega})=\Omega(e^{-\omega}).
\]
So, by Chernoff's bound (see~\cite[Theorem A.1.7]{AlonSpencer}) and a union bound over all the at most 
$\lceil R\rceil^2$ pairs $(a,b)\in\Omega$ such that
$2\theta_R(a,b)>\phi_r$,
the probability that some such $O_{a,b}$ (respectively, $I_{a,b}$)
exceeds its expectation by a factor $e^{\alpha\omega}2\ln(\omega'R^2)$
  (respectively,
  $e^{\omega}\cdot 2\ln(\omega'R^2)$) is at most $\calO(1/\omega')$.
Hence, a.a.s.,
\begin{align}
\sum_{\substack{(a,b)\in\Omega \text{ s.t. } \\ 2\theta_R(a,b)>\phi_r}}
O_{a,b}\cdot I_{a,b}
& \leq
e^{(1+\alpha)\omega}(\ln(\omega'R^2))^2
\sum_{\substack{(a,b)\in\Omega \text{ s.t. } \\ 2\theta_R(a,b)>\phi_r, a+b\le R}}
\calO\Big(e^{\frac{R}{2}+\alpha(R-r)-\alpha(2R-a-b)-\omega}\Big)
\notag \\ & \qquad 
+
e^{(1+\alpha)\omega}(\ln(\omega'R^2))^2
\sum_{\substack{(a,b)\in\Omega \text{ s.t. }\\ 2\theta_R(a,b)>\phi_r, a+b> R}}
\calO\Big(e^{\alpha(R-r)-(\alpha-\frac12)(2R-a-b)-\omega}\Big). \label{eq:lastboy}
\end{align}
Recall that $2\theta_R(a,b)>\phi_r$ is equivalent to
$a+b\leq 2R-2\alpha(R-r)+2\omega+\Theta(1)$.
Thus, summing over the pairs $(a,b)\in\Omega$ such
that $a+b=s$ and then over the valid
values of $s\leq 2R-2\alpha(R-r)+2\omega+\Theta(1)$ (observing that there are at most $\lfloor R\rfloor=\calO(\ln n)$
pairs
$(a,b)\in\Omega$ for which $a+b=s$ for a given $s$), the second sum
in the right-hand side of \eqref{eq:lastboy} above is
\[
\calO\Big(e^{\alpha(R-r)-2(\alpha-\frac12)\alpha(R-r) - 2(1-\alpha)\omega}\cdot\ln n\Big)
= \calO\Big(e^{2\alpha(1-\alpha)(R-r)-2(1-\alpha)\omega}\cdot\ln n\Big)
\]
Similarly, since $r\geq (1-\frac{1}{2\alpha})R$ implies that
$\alpha(R-r)-(\alpha-\frac12)R\leq 2\alpha(1-\alpha)(R-r)$,
summing over the pairs $(a,b)\in\Omega$ such that $a+b=s$ and
then over the valid values of $s\leq R$, the first summation in the
right-hand side of \eqref{eq:lastboy} above is
\[
\calO\Big(e^{\alpha(R-r)-(\alpha-\frac12)R-\omega}\cdot\ln n\Big)
=
\calO\Big(e^{2\alpha(1-\alpha)(R-r)-\omega}\cdot\ln n\Big).
\]
Our ongoing discussion yields (using that $\frac12<\alpha<1$ implies
that $2(1-\alpha)<1$ and $3\alpha-1<2\alpha$)
\[
\sum_{\substack{(a,b)\in\Omega \text{ s.t. } \\ 2\theta_R(a,b)>\phi_r}}
O_{a,b}\cdot I_{a,b}
=
\calO(e^{2\alpha(1-\alpha)(R-r)}\ln n\cdot e^{2\alpha\omega}(\ln(\omega' R^2))^2).
\]
Since $\varphi\geq 2\alpha\omega$, we conclude
that a.a.s.~$|\delta(V\cap\Upsilon_p)|\leq \calO(e^{2\alpha(1-\alpha)(R-r)}\cdot\ln n\cdot e^{\varphi}(\ln(\omega' R^2))^2)$.
Taking any function $\omega'$ satisfying $\omega(1)=\omega'=\calO(R)$ yields the result.
\end{proof}

\begin{proposition}
Let $p$ and $q$ be points at radial coordinates $r_p$ and $r_q$, respectively, where $(1-\frac{1}{2\alpha})R < r_p,r_q<R-\omega/(1-\alpha)(\alpha-\frac12)$.
If $s$ and $t$ are vertices added to $\HRG:=\HRG_{\alpha,\nu}(n)$ at points $p$ and $q$ respectively, and the angle formed by $p$ and $q$ at the origin is at least $\min\{\phi_{r_p},\phi_{r_q}\}$, then a.a.s.,
\[
\Res{s}{t} = \Omega\big((e^{-2\alpha(1-\alpha)(R-r_p)}+e^{-2\alpha(1-\alpha)(R-r_q)})/(\ln n\ln\ln n)^{3}\big).
\]
\end{proposition}
\begin{proof}
Clearly, the removal of the edges $\delta_{\HRG}(V(\HRG)\cap\Upsilon_p)$
leaves $s$ and $t$ in distinct components.
By Lemma~\ref{ap:lem:edgecutVertex} and since the inverse of the cardinality of an edge-cut between $s$ and $t$ is a lower bound on $\Res{s}{t}$ we obtain the claimed result.
\end{proof}
Recall that (by Lemma~\ref{prelim:lem:ballsMeasure}) for $\rho$ as in~\eqref{flow:eqn:rho} the degree of a vertex $w\in V(\HRG)\cap B_O(\rho)$ is, a.a.s., $\Theta(e^{\frac12(R-r_w)})$.
Thus, the preceding result together with the commute time identity (see Lemma~\ref{lem:commutetime}) and
the fact that a.a.s.~the number of edges of the center component of $\calG_{\alpha,\nu}(n)$ is $\Theta(n)$ (see Lemma~\ref{prelim:lem:volumeCenterComp}) yield the  lower bound of Theorem~\ref{ap:thm:commut}.

\section{Concluding remarks}\label{sec:conclussion}

 In this paper we determined the order of the maximum hitting time, cover time, target time and effective resistance between two uniform vertices, the results hold in expectation and a.a.s.~(w.r.t.~to the HRG). Our main finding to take away is that (in expectation) there are order $\log n$ gaps between each of the quantities. This indicates that most vertices in the giant are well-connected to the center (vertices at distance at most $R/2$ of the origin) of the HRG, but a significant proportion is not.
 We also gave sharp estimates for commute times for pairs of vertices $s,t$ added to the HRG for almost all the relevant positions at which $s, t$ could be placed, provided the angular coordinates of the vertices satisfy some (mild) condition. We speculate that there could be settings in which the commute times between vertices whose angular distances are very close might be relevant. Our estimates do not hold in this case and might be worth deriving.
 
We restricted our study to the giant component of the graph in the regime $1/2<\alpha <1$. Even though this is arguably the most interesting regime it would still be interesting to determine the aforementioned quantities on the other smaller components or when $\alpha \notin (1/2,1)$. Another problem we leave open is to discover the leading constants hidden behind our asymptotic notation. If the expression for the clustering coefficient of the HRG~\cite{fountoulakisClusteringHyperbolicModel2021} is anything to go by these constants may have very rich and complex expressions as functions of $\alpha$ and $\nu$. 
 
 An interesting problem is to determine the order of the meeting time, that is, the expected time it takes two (lazy) random walks to occupy the same vertex when started from the worst case start vertices~\cite{KanadeMS19}. Finally, the mixing time of a (lazy) random walk on the giant HRG is known up to polylogarithmic factors by~\cite{KiwiMitscheSpectral}. Closing this gap is of great importance, but it may well be quite challenging.

 \section*{Acknowledgements} M.K.~gratefully acknowledges the support of grant GrHyDy ANR-20-CE40-0002 and BASAL funds for centers of excellence from ANID-Chile grants ACE210010 and FB210005. M.S.~gratefully acknowledges the support of the DIAMANT PhD travel grant for a 2-week research stay in Santiago de Chile. J.S.~was supported by EPSRC project EP/T004878/1: Multilayer Algorithmics to Leverage Graph Structure and the ERC Starting Grant 679660 (DYNAMIC MARCH).

\bibliographystyle{plain}
\bibliography{refs}

\end{document}